\documentclass[11pt]{amsart}

\usepackage[text={475pt,660pt},centering]{geometry}

\usepackage{color}
\usepackage{esint,amssymb}
\usepackage{graphicx}
\usepackage{MnSymbol}
\usepackage{mathtools}
\usepackage[colorlinks=true, pdfstartview=FitV, linkcolor=blue, citecolor=blue, urlcolor=blue,pagebackref=false]{hyperref}
\usepackage{microtype}

\usepackage{bm}
\usepackage{scalerel} 
\usepackage{dsfont}
\usepackage{mathrsfs}
\usepackage{eufrak}
\usepackage{comment}
\usepackage[font={footnotesize}]{caption}
\usepackage{enumitem}

\definecolor{darkgreen}{rgb}{0,0.5,0}
\definecolor{darkblue}{rgb}{0,0,0.7}
\definecolor{darkred}{rgb}{0.9,0.1,0.1}

\makeatletter
\newtheorem*{rep@theorem}{\rep@title}
\newcommand{\newreptheorem}[2]{%
\newenvironment{rep#1}[1]{%
 \def\rep@title{#2 \ref{##1}}%
 \begin{rep@theorem}}%
 {\end{rep@theorem}}}
\makeatother

\newtheorem{theorem}{Theorem}
\newreptheorem{theorem}{Theorem}
\newtheorem{proposition}{Proposition}
\newreptheorem{proposition}{Proposition}
\newtheorem{lemma}[proposition]{Lemma}
\newreptheorem{lemma}{Lemma}

\theoremstyle{remark}

\theoremstyle{definition}
\newtheorem{definition}[proposition]{Definition}
\newtheorem{remark}[proposition]{Remark}
\newtheorem{conjecture}[proposition]{Conjecture}

\numberwithin{equation}{section}
\numberwithin{proposition}{section}

\newcommand{\Z}{\mathbb{Z}}
\newcommand{\N}{\mathbb{N}}
\newcommand{\R}{\mathbb{R}}
\newcommand{\Wpar}{W^{1,2 + \delta}_{\mathrm{par}}}
\newcommand{\Hpar}{H^1_{\mathrm{par}}}

\newcommand{\aL}{\underline{L}}
\newcommand{\aH}{\underline{H}}
\newcommand{\aW}{\underline{W}}

\newcommand{\norm}[1]{\left\Vert{#1}\right\Vert}

\newcommand{\D}{\mathcal{D}}
\newcommand{\Dr}[1]{\D_{e_{#1}}}
\newcommand{\CQ}[1]{I_{#1} \times (\C_\infty \cap B_{#1})}
\newcommand{\CQy}[1]{I_{#1} \times (\C_\infty \cap B_{#1}(y))}
\newcommand{\CQx}[1]{I_{#1} \times (\C_\infty \cap B_{#1}(x))}
\newcommand{\sigk}{\bar{\sigma}^2}
\newcommand{\dsigk}{\frac{1}{2}\bar{\sigma}^2}

\newcommand{\E}{\mathbb{E}}
\renewcommand{\P}{\mathbb{P}}
\newcommand{\F}{\mathcal{F}}
\newcommand{\Zd}{\mathbb{Z}^d}
\newcommand{\Rd}{{\mathbb{R}^d}}
\newcommand{\Bd}{\mathcal{B}_d}
\newcommand{\Bdo}{\overrightarrow{\Bd}}

\newcommand{\Bda}{\mathcal{B}^{\a}_d}
\newcommand{\itr}{\mathrm{int}}

\newcommand{\ep}{\varepsilon}

\renewcommand{\a}{\mathbf{a}}

\newcommand{\ahom}{{\overbracket[1pt][-1pt]{\a}}}  

\renewcommand{\subset}{\subseteq}

 \newcommand{\cu}{\square}

\renewcommand{\fint}{\strokedint}

\DeclareMathOperator{\dist}{dist}

\DeclareMathOperator*{\osc}{osc}

\DeclareMathOperator{\cv}{Cv}

\DeclareMathOperator{\supp}{supp}

\DeclareMathOperator{\size}{size}
\DeclareMathOperator{\intr}{int}

\renewcommand{\tilde}{\widetilde}

\newcommand{\indc}{\mathds{1}}

\newcommand{\B}{\mathcal{B}}
\newcommand{\C}{\mathscr{C}}
\newcommand{\Pa}{\mathcal{P}}

\newcommand{\G}{\mathcal{G}}

\newcommand{\M}{\mathcal{M}}

\renewcommand{\O}{\mathcal{O}}
\newcommand{\A}{\mathcal{A}}

\newcommand{\pc}{\mathfrak{p}_{\mathfrak{c}}}
\newcommand{\p}{\mathfrak{p}}

\newcommand{\T}{\mathcal{T}}

\begin{document}

\title[Green's functions on percolation clusters]{Quantitative homogenization of the parabolic and elliptic Green's functions on percolation clusters}

\begin{abstract}
We study the heat kernel and the Green's function on the infinite supercritical percolation cluster in dimension $d \geq 2$ and prove a quantitative homogenization theorem for these functions with an almost optimal rate of convergence. These results are a quantitative version of the local central limit theorem proved by Barlow and Hambly in~\cite{BH}. The proof relies on a structure of renormalization for the infinite percolation cluster introduced in~\cite{AD2}, Gaussian bounds on the heat kernel established by Barlow in~\cite{Ba} and tools of the theory of quantitative stochastic homogenization. An important step in the proof is to establish a $C^{0,1}$-large-scale regularity theory for caloric functions on the infinite cluster and is of independent interest.
\end{abstract}

\author[P. Dario]{Paul Dario}
\address[P. Dario]{School of Mathematical Sciences, Tel Aviv University, Ramat Aviv, Tel Aviv 69978, Israel}
\email{pauldario@mail.tau.ac.il}

\author[C. Gu]{Chenlin Gu}
\address[C. Gu]{DMA, Ecole normale sup\'erieure, PSL Research University, Paris, France}
\email{chenlin.gu@ens.fr}

\date{\today}

\maketitle

\begin{figure}[h!]
\centering
\includegraphics[scale=0.5]{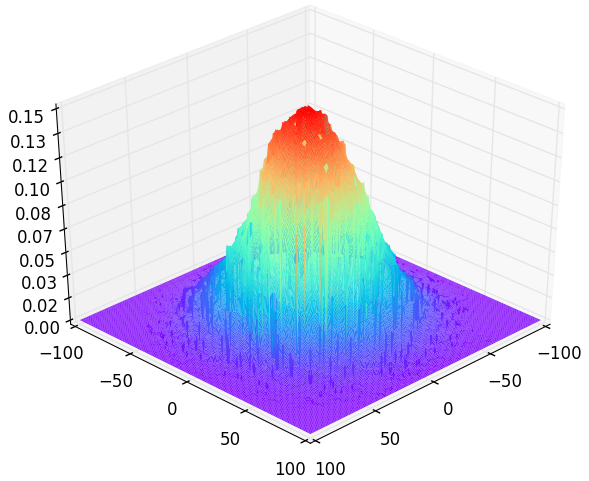}
\includegraphics[scale=0.5]{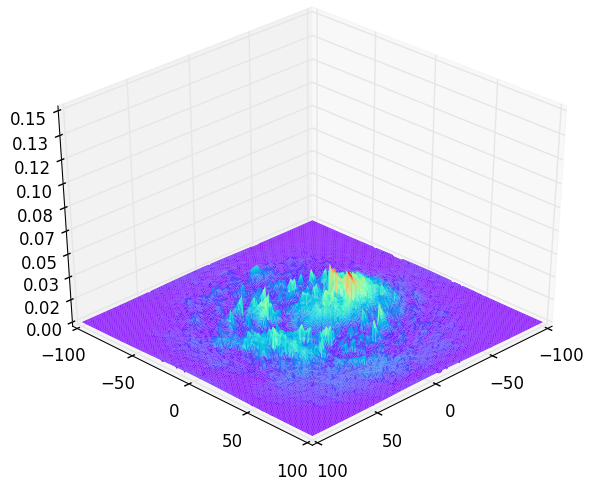}

\caption{The figure on the left represents the density distribution of the function $t^{\frac{d}{2}}p(t, \cdot, 0)$ where the map $p$ is the 2-dimensional heat kernel on the infinite percolation cluster with probability $\p = 0.7$ at time $t = 1000$; it is similar to a Gaussian distribution. 
The figure on the right is the error between the map $t^{\frac{d}{2}} p(t, \cdot, 0)$ and the normalized Gaussian heat kernel $\theta(p)^{-1}t^{\frac{d}{2}} \bar{p}(t, \cdot)$ defined in~\eqref{def.heatkernelbarp}; it is small compared to density distribution on the left.}
\end{figure}

\newpage
\setcounter{tocdepth}{1}
\tableofcontents

\section{Introduction}

\subsection{General introduction and main results}

In this article, we study the continuous-time random walk on the infinite cluster of the supercritical Bernoulli bond percolation of the Euclidean lattice $\Zd$, in dimension $d \geq 2$. The model considered is a specific case of the general random conductance model and can be described as follows. We let $\B_d$ be the set of bonds of $\Zd$, i.e., the set of unordered pairs of nearest neighbors of $\Zd$. We denote by $\Omega$ the set of functions from $\B_d$ to the set of non-negative real numbers $[0, \infty)$. A generic element of $\Omega$ is denoted by $\a$ and called \emph{an environment}.

For a given environment $\a \in \Omega$ and a given bond $e \in \B_d$, we call the value $\a(e)$ \emph{the conductance} of the bond $e$. We fix an ellipticity parameter $\lambda \in (0, 1]$ and add some randomness to the model by assuming that the collection of conductances $\left\{ \a(e) \right\}_{e \in \B_d}$ is an i.i.d. family of random variables whose law is supported in the set $\{ 0 \} \cup [\lambda, 1]$. We define $\p := \P \left[ \a(e) \neq 0 \right]$ and assume that 
\begin{equation*}
    \p > \p_c(d),
\end{equation*}
where $\p_c(d)$ is the bond percolation threshold for the lattice $\Zd$. This  assumption ensures that, almost surely, there exists a unique infinite connected component of edges with non-zero conductances (or cluster) which we denote by $\C_\infty$ (see~\cite{BK}). This cluster has a non-zero density which is given by the probability $\theta(\p):= \P \left[ 0 \in \C_\infty \right]$. The model of continuous-time random walk considered in this article is \emph{the variable speed random walk} (or VSRW) and is defined as follows. Given an environment $\a \in \Omega$ and a starting point $y \in \C_\infty$, we endow each edge $e \in \B_d$ with a random clock whose law is exponential of parameter $\a(e)$ and assume that they are mutually independent. We then let $\left( X_t \right)_{t \geq 0}$ be the random walk which starts from $y$, i.e., $X_0 = y$, and, when $X(t) = x$, the random walker waits at $x$ until one of the clocks at an adjacent edge to $x$ rings, and moves across the edge to the neighboring point instantly. We then restart the clocks. This construction gives rise to a continuous-time Markov process on the infinite cluster~$\C_\infty$ whose generator is the elliptic operator $\nabla \cdot \a \nabla$ defined by, for each function $u : \C_\infty \rightarrow \R$ and each point $x \in \C_\infty$,
\begin{equation} \label{generatorVSRW}
    \nabla \cdot \a \nabla u(x) := \sum_{z \sim x} \a(\{x,z\}) \left( u(z) - u(x) \right).
\end{equation}
We denote the transition density of the random walk by
\begin{equation*}
    p \left( t , x , y \right)  = p^{\a} \left( t , x , y \right) := \P^{\a}_{y} \left( X_t = x \right),
\end{equation*}
and often omit the dependence in the environment $\a$ in the notation. The transition density can be equivalently defined as the solution of the parabolic equation
\begin{equation} \label{intdefheatker}
    \left\{ \begin{array}{cc}
          \partial_t p(\cdot , \cdot , y) - \nabla \cdot \a \nabla p(\cdot , \cdot , y) = 0 &~\mbox{in}~ (0 , \infty) \times \C_\infty, \\
          p(0 , \cdot , y) = \delta_y &~\mbox{in}~  \C_\infty.
    \end{array} \right.
\end{equation}
Due to this characterization, we often refer to the transition density $p$ as the \emph{heat kernel} or the \emph{parabolic Green's function}.

There are other related models of random walk on supercritical percolation clusters which have been studied in the literature, two of the most common ones are:
\begin{enumerate}[label=(\roman*)]
    \item \emph{The constant speed random walk} (or CSRW), the random walker starts from a point $y \in \C_\infty$. When $X(t) = x$, it waits for an exponential time of parameter $1$ and then jumps to a neighboring point $z$ according to the transition probability
    \begin{equation} \label{transkerne.P}
        P(x , z) = \frac{\a \left( \{x , z\} \right)}{\sum_{w \sim x} \a \left( \{x ,w\} \right)}.
    \end{equation}
    This construction also gives rise to a continuous-time Markov process whose generator is given by, for each function $u : \C_\infty \to \R$ and each point $y \in \C_\infty$,
\begin{equation*}
     \frac{1}{\sum_{z \sim x} \a \left( \{x ,z\} \right)} \sum_{z \sim x} \a(\{x,z\}) \left( u(z) - u(x) \right).
\end{equation*}
    \item \emph{The simple random walk} (or SRW), the random walk $\left( X_n \right)_{n \in \N}$ is indexed on the integers, it starts from a point $y \in \C_\infty$, when $X_n = x$, the value of $X_{n+1}$ is chosen randomly among all the neighbors of~$x$ following the transition probability~\eqref{transkerne.P}.
\end{enumerate}

These processes have similar, although not identical, properties and have been the subject of interest in the literature. In the case of the percolation cluster, i.e., when the environment $\a$ is only allowed to take the values $0$ or $1$, an annealed invariance principle was proved in~\cite{de1989invariance} by De Masi, Ferrari, Goldstein and Wick. In~\cite{SS}, Sidoravicius and Sznitman proved a quenched invariance principle for the simple random walk in dimension $d \geq 4$. This result was extended to every dimension $d \geq 2$ by Berger and Biskup in~\cite{BB} (for the SRW) and by Mathieu and Piatnitski in~\cite{MP} (for the CSRW).

For the VSRW, a similar quenched invariance principle holds: there exists a deterministic diffusivity constant $\bar{\sigma} > 0$ such that, for almost every environment, the following convergence holds in the Skorokhod topology
\begin{equation} \label{intro.inv.princ}
    \ep X_{\frac{\cdot}{\ep^2}} \overset{\mathrm{(law)}}{\underset{\ep \rightarrow 0}{\longrightarrow}} \bar{\sigma} B_{\cdot},
\end{equation}
where $B_{\cdot}$ is a standard Brownian motion. From a homogenization perspective, the diffusivity $\sigk$ of the limiting Brownian motion is related to the homogenized coefficient $\ahom$ associated to the elliptic and parabolic problems on the percolation cluster by the identity $\ahom = \frac 12 \theta(\p) \sigk$ (see the formula~\eqref{def.sigkB1} of Appendix~\ref{appendixb}).

The properties of the heat kernel $p$ on the infinite cluster have been investigated in the literature. In~\cite{MR}, Mathieu and Remy proved that, almost surely, the heat kernel decays as fast as $t^{-d/2}$. These bounds were extended in~\cite{Ba} by Barlow who established Gaussian lower and upper bounds for this function; we will recall his precise result in Theorem~\ref{barlow} below.

In the article~\cite{BH}, Barlow and Hambly proved a parabolic Harnack inequality, a local central limit theorem for the CSRW, and bounds on the elliptic Green's function on the infinite cluster. Their main result can be adapted to the case of the VSRW, and reads as follows: if we define, for each $t \geq 0$ and $x \in \Rd$,
\begin{equation} \label{def.heatkernelbarp}
    \bar p (t , x) := \frac{1}{\left( 2 \pi  \sigk t \right)^{d/2}} \exp \left( - \frac{|x|^2}{2\ \sigk t} \right),
\end{equation}
the heat kernel with diffusivity $\bar \sigma$, then, for each time $T>0$, the following convergence holds, $\P$-almost surely on the event $\{ 0 \in \C_\infty \}$,
\begin{equation} \label{eq:bhllt}
    \lim_{n \to \infty} \left| n^{d/2} p(nt , g^\omega_n ( x), 0) - \theta(\p)^{-1} \bar p (t , x) \right| = 0,
\end{equation}
uniformly in the spatial variable $x \in \Rd$ and in the time variable $t \geq T$, where the notation $g^\omega_n ( x)$ means the closest point to $\sqrt{n} x$ in the infinite cluster under the environment $\omega$.

The main result of this article is a quantitative version of the local central limit theorem for the VSRW and is stated below.

\begin{theorem} \label{mainthm}
For each exponent $\delta > 0$, there exist a positive constant $C < \infty$ and an exponent $s > 0$, depending only on the parameters $d , \lambda, \p $ and $\delta$, such that for every $y \in \Zd$, there exists a non-negative random time $\T_{\mathrm{par}, \delta}(y)$ satisfying the stochastic integrability estimate
\begin{equation*}
    \forall T \geq 0, ~\P \left( \T_{\mathrm{par}, \delta}(y) \geq T  \right) \leq C \exp \left( - \frac{T^s}{C} \right),
\end{equation*}
such that, on the event $\{ y \in \C_\infty \}$, for every $x \in \C_\infty$ and every $t \geq \max \left( \T_{\mathrm{par}, \delta}(y) , |x-y| \right)$,
\begin{equation} \label{mainthmmainest}
    \left| p (t , x , y) - \theta(\p)^{-1} \bar p (t , x - y) \right| \leq C t^{- \frac{d}{2} - \left(\frac{1}{2} - \delta\right)} \exp \left( - \frac{|x-y|^2}{Ct} \right).
\end{equation}
\end{theorem}

\begin{remark} \label{remark1.1}
The heat kernel $p$ does not exactly converge to the heat kernel $\bar p$ and there is an additional normalization constant $\theta(\p)^{-1}$ in~\eqref{mainthmmainest}. A heuristic reason explaining why such a term is necessary is the following: since $p(t , \cdot , y)$ is a probability measure on the infinite cluster, one has
\begin{equation*}
    \sum_{x \in \C_\infty} p \left(t , x , y \right) = 1.
\end{equation*}
One also has, by definition of the heat kernel $\bar p$,
\begin{equation*}
    \int_{\Rd} \bar p (t ,x -y) \, dx = 1.
\end{equation*}
Since the infinite cluster has density $\theta(\p)$, we expect that
\begin{equation*}
    \sum_{x \in \C_\infty} \bar p (t , x - y) \simeq \theta(\p) \int_{\Rd} \bar p (t , x - y) \, dx = \theta(\p),
\end{equation*}
and we refer to Proposition~\ref{prop:DensityContrentrationGaussian} for a precise statement. As a consequence, we cannot expect the maps $p$ and $\bar p$ to be close since they have different mass on the infinite cluster; adding the normalization term $\theta(\p)^{-1}$ ensures that the mass of $ \theta(\p)^{-1} \bar p$ on the infinite cluster is approximately equal to $1$.
\end{remark}

As an application of this result, we deduce a quantitative homogenization theorem for the elliptic Green's function on the infinite cluster. In dimension $d \geq 3$, given an environment $\a \in \Omega$ and a point~$y \in \C_\infty$, we define the Green's function $g(\cdot,y)$ as the solution of the equation
\begin{equation*}
    -\nabla \cdot \a \nabla g(\cdot , y) = \delta_y ~\mbox{in}~\C_\infty~\mbox{such that}~ g(x,y) \underset{x \rightarrow \infty}{\longrightarrow} 0.
\end{equation*}
This function exists, is unique almost surely and is related to the transition probability $p$ through the identity
\begin{equation} \label{introgreenellparfo}
    g(x,y) = \int_0^\infty p(t , x , y) \, dt.
\end{equation}
In dimension $2$, the situation is different since the Green's function is not bounded at infinity, and we define $g\left( \cdot, y \right)$ as the unique function which satisfies
\begin{equation*}
    -\nabla \cdot \a \nabla g(\cdot , y) = \delta_y ~\mbox{in}~\C_\infty,~ \frac{1}{|x|} g(x,y) \underset{x \rightarrow \infty}{\longrightarrow} 0~\mbox{and}~ g(y,y) = 0.
\end{equation*}
This function is related to the transition probability $p$ through the identity
\begin{equation*}
    g(x,y) = \int_0^\infty \left(p(t , x , y)  - p(t , y , y)\right) \, dt.
\end{equation*}
In the statement below, we denote by $\bar g$ the homogenized Green's function defined by the formula, for each point $x \in \Rd \setminus \left\{ 0 \right\} $,
\begin{equation}\label{eq:ghom}
    \bar g (x) := \left\{ \begin{array}{lcl} 
   - \frac{1}{ \pi \sigk \theta(\p) } \ln \left| x \right| & \hspace{5mm} \mbox{if}~ d=2, \\
   \frac{\Gamma \left( d/2 - 1 \right)}{ \left(2 \pi^{d/2}\sigk \theta(\p) \right)} \frac{1}{|x|^{d-2}} & \hspace{5mm} \mbox{if}~ d\geq 3,
    \end{array} \right.
\end{equation}
where the symbol $\Gamma$ denotes the standard Gamma function. Theorem~\ref{mainthmell} describes the asymptotic behavior of the Green's function $g$.

\begin{theorem} \label{mainthmell}
For each exponent $\delta > 0$, there exist a positive constant $C < \infty$ and an exponent $s > 0$, depending only on the parameters $d , \lambda, \p$ and $\delta$, such that for every $y \in \Zd$, there exists a non-negative random variable $\M_{\mathrm{ell}, \delta}(y)$ satisfying
\begin{equation*}
    \forall R \geq 0, ~\P \left( \M_{\mathrm{ell}, \delta}(y) \geq R  \right) \leq C \exp \left( - \frac{R^s}{C} \right),
\end{equation*}
such that, on the event $\{ y \in \C_\infty\}$:
\begin{enumerate}
\item In dimension $d \geq 3$, for every point $x \in \C_\infty$ satisfying $|x-y| \geq \M_{\mathrm{ell}, \delta}(y)$,
\begin{equation} \label{eq:TV11522603}
    \left| g( x , y) - \bar g (x-y) \right| \leq \frac{1}{|x-y|^{1-\delta}} \frac{C}{|x-y|^{d-2}}.
\end{equation}
\medskip 
\item In dimension $2$, the limit
\begin{equation*}
    K(y):= \lim_{x \rightarrow \infty}\left( g(x , y) - \bar g (x-y)\right),
\end{equation*}
exists, is finite almost surely and satisfies the stochastic integrability estimate
\begin{equation*}
    \forall R \geq 0, ~\P \left( |K(y)| \geq R  \right) \leq C \exp \left( - \frac{R^s}{C} \right).
\end{equation*}
Moreover, for every point $x \in \C_\infty$ satisfying $|x - y| \geq \M_{\mathrm{ell}, \delta}(y)$,
\begin{equation} \label{eq:TV11512603}
    \left| g(x , y) - \bar g(x-y) - K(y) \right| \leq \frac{C}{|x-y|^{1-\delta}}.
\end{equation}
\end{enumerate}
\end{theorem}

\begin{remark}
In dimension $2$, the situation is specific due to the unbounded behavior of the Green's function, and the theorem identifies the first-order term. The second term in the asymptotic development is of constant order and is random: with the normalization chosen for the Green's function, the constant $K$ depends on the geometry of the infinite cluster and cannot be deterministic. We nevertheless expect it not to be too large and prove that it satisfies a stretched exponential stochastic integrability estimate.
\end{remark}

\begin{figure}
\centering
\includegraphics[scale=0.3]{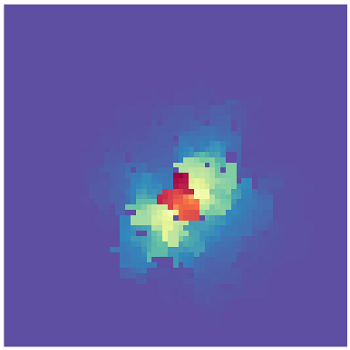}
\includegraphics[scale=0.3]{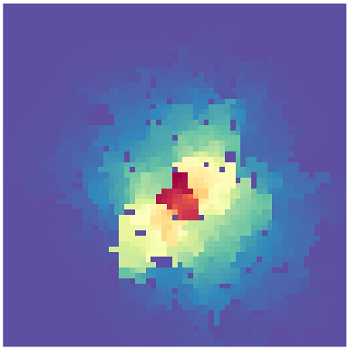}
\includegraphics[scale=0.3]{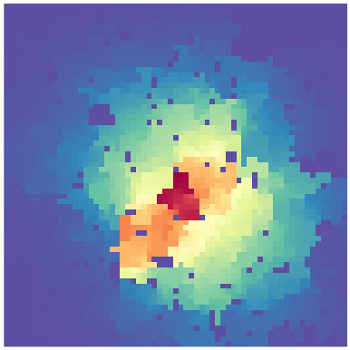}
\includegraphics[scale=0.3]{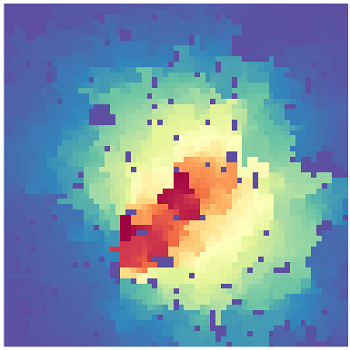}
\includegraphics[scale=0.3]{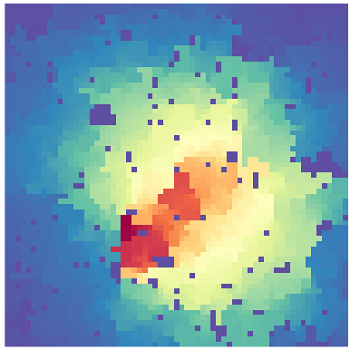}

\smallskip

\includegraphics[scale=0.3]{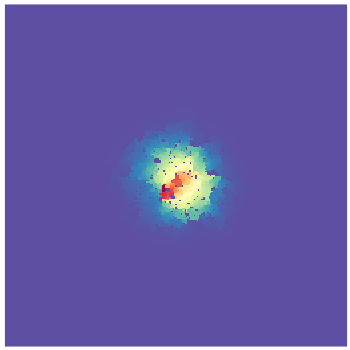}
\includegraphics[scale=0.3]{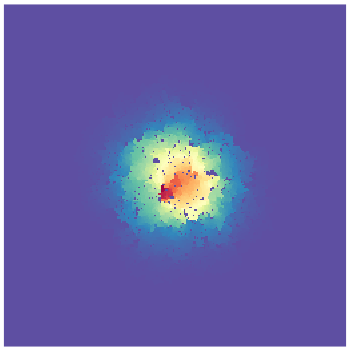}
\includegraphics[scale=0.3]{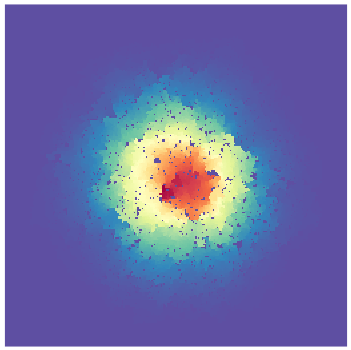}
\includegraphics[scale=0.3]{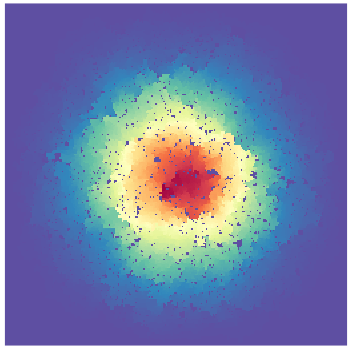}
\includegraphics[scale=0.3]{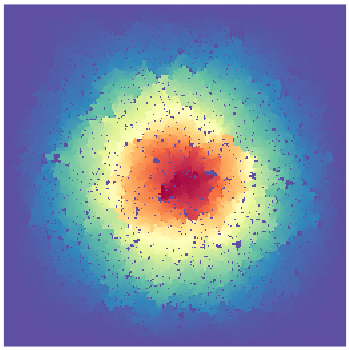}

\smallskip

\includegraphics[scale=0.3]{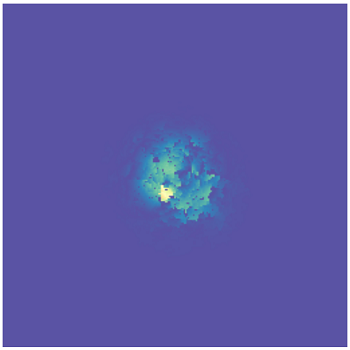}
\includegraphics[scale=0.3]{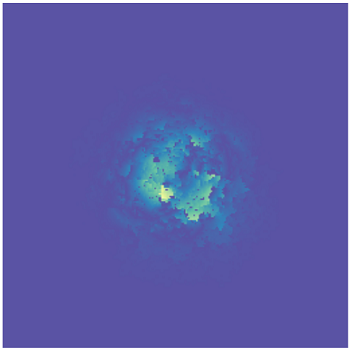}
\includegraphics[scale=0.3]{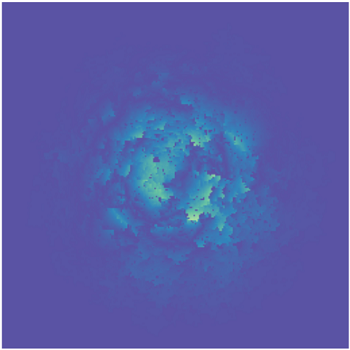}
\includegraphics[scale=0.3]{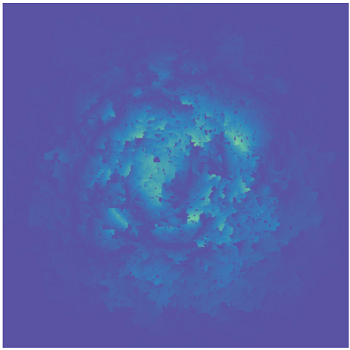}
\includegraphics[scale=0.3]{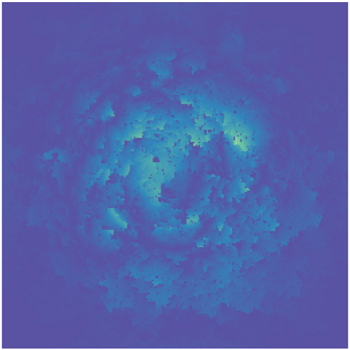}

\caption{A simulation to illustrate the convergence of the parabolic Green's function for the VSRW on the infinite cluster with $\p = 0.6$. We use different colors to represent the level sets of the map $t^{\frac{d}{2}}p(t,\cdot, y)$ in the first two rows. The figures in first row are drawn for the short times $t = 100, 200, 300, 400, 500$ in a cube of size $64 \times 64$ and the level sets of the heat kernel are perturbed by the geometry of the infinite cluster. In the second row, the figures are drawn for the long times $t = 500, 1000, 2000, 3000, 4000$ and in a cube of size $256 \times 256$; in this case, homogenization happens and the geometry of the level sets of the heat kernel is similar to the one of a Gaussian heat kernel. In the third row, we simulate the function $t^{\frac{d}{2}} \left| p (t , \cdot , y) - \theta(\p)^{-1} \bar p (t , \cdot - y) \right|\indc_{\{x \in \C_\infty\}}$ associated to the figures in the second line and we observe that the errors decay to $0$ as the time tends to infinity.}
\end{figure}

We complete this section by mentioning a potential application of these theorems. Theorem~\ref{mainthm} shows that the law of the VSRW on the infinite percolation cluster converges quantitatively to the one of the Brownian motion $\left( \bar \sigma B_t \right)_{t \geq 0}$. To go one step further in the analysis, one can try to construct a coupling between the random walk $(X_t)_{t \geq 0}$ and the Brownian motion $\left( \bar \sigma B_t \right)_{t \geq 0}$ such that their trajectories are close, i.e., such that $\sup_{0 \leq s \leq t} \vert X_s - \bar \sigma B_s\vert$ is small. This question is known as the embedding problem: a good error should be at least of order $o(\sqrt{t})$. In the case of the simple random walk on $\Zd$, the optimal result is given by the Koml\'os-Major-Tusn\'ady Approximation (see~\cite{KMT1, KMT2}) and gives an error of order $O(\log t)$. Adapting this result to the setting considered here requires to take into account the degenerate geometry of the percolation cluster; we believe that Theorem~\ref{mainthm} can be useful in this regard.

\subsection{Strategy of the proof}

On the supercritical percolation cluster, a qualitative version of Theorem~\ref{mainthm} is established by Barlow and Hambly in~\cite{BH}, where the strategy implemented is to first prove a parabolic Harnack inequality for the heat equation. From the Harnack inequality, one derives a $C^{0,\alpha}$-H\"older regularity estimate (for some small exponent $\alpha > 0$) on the heat kernel. It is then possible to combine this additional regularity with the quenched invariance principle, established on the percolation cluster in~\cite{SS, MR, BB}, to obtain the local central limit theorem.

In the present article, the strategy adopted is different and follows ideas from the theory of stochastic homogenization, more specifically the ones of~\cite[Chapter 8]{armstrong2017quantitative}. A first crucial ingredient in the proof is the first-order corrector, which can be characterized as follows: given a slope $p \in \Rd $, the corrector~$\chi_p$ is defined as the unique function (up to a constant) which is a solution of the elliptic equation
\begin{equation*}
    -\nabla \cdot \a \left( p + \nabla \chi_p \right) = 0 ~\mbox{in}~\C_\infty,
\end{equation*}
and which has sublinear oscillation, i.e.,
\begin{equation*}
    \frac 1r \osc_{x \in \C_\infty \cap B_r} \chi_p := \frac 1r \left( \sup_{x \in \C_\infty \cap B_r} \chi_p - \inf_{x \in \C_\infty \cap B_r} \chi_p \right) \underset{r \rightarrow \infty}{\longrightarrow} 0.
\end{equation*}
The corrector is defined and some of its important properties are presented in Section~\ref{homogpercclu}. We note that the use of the corrector to study random walk on supercritical percolation cluster is not new: it is a key ingredient in the proofs of the quenched invariance principle (see~\cite{SS, MR, BB}). Once equipped with this function, the analysis relies on a classical strategy in stochastic homogenization: the two-scale expansion. The general approach relies on the definition of the function
\begin{equation} \label{2scexporgproo}
    h(t , x , y) :=\theta(\p)^{-1} \left( \bar p \left( t , x - y \right) + \sum_{k= 1}^d  \partial_k \bar{p}(t,x - y) \chi_{e_k}(x) \right),
 \end{equation}
where $\left( e_k \right)_{k = \{1 , \cdots, d\}}$ denotes the canonical basis of $\Rd$ and $\bar p $ is the continuous heat kernel defined in~\eqref{def.heatkernelbarp}. The strategy is then to compute the value of
\begin{equation} \label{heateqh.espoproo}
    \partial_t h - \nabla \cdot \a \nabla h,
\end{equation}
by using the explicit formula on $h$ stated in~\eqref{2scexporgproo} and to prove that it is quantitatively small in the correct functional space (precisely, the parabolic $\aH^{-1}$ space introduced in~\eqref{eq:WeakH1}). Obtaining this result requires two types of quantitative information on the corrector:
\begin{itemize}
    \item One needs to have quantitative sublinearity of the corrector, i.e.,
    \begin{equation} \label{quantsubcorrexpoproo}
         \frac 1{r^{\alpha}} \osc_{x \in \C_\infty \cap B_r} \chi_p \underset{r \rightarrow \infty}{\longrightarrow} 0,
    \end{equation}
    for every exponent $\alpha > 0$.
    \item One needs to have a quantitative control on the flux of the corrector in the weak $\aH^{-1}$ norm,
    \begin{equation} \label{quantsubfluxpoproo}
    \frac{1}{r^{\alpha}} \left\| \a \left( p + \nabla \chi_p \right) - \dsigk p \right\|_{\aH^{-1}(\C_\infty \cap B_r)}  \underset{r \rightarrow \infty}{\longrightarrow} 0,
    \end{equation}
    for every exponent $\alpha > 0$, where $\sigk$ is the same diffusivity constant as in the definition~\eqref{def.heatkernelbarp} of the heat kernel $\bar p$.
\end{itemize}
The sublinearity of the corrector in the setting of the percolation cluster is established qualitatively in~\cite{SS, MR, BB} and quantitatively in~\cite{AD2, dario2018optimal, gu2019efficient}. The second property~\eqref{quantsubfluxpoproo} cannot be directly deduced from the results of \cite{AD2, dario2018optimal, gu2019efficient} and Appendix~\ref{appendixb} is devoted to the proof of this result.

Once one has good quantitative control over the $\aH^{-1}$-norm of $\partial_t h - \nabla \cdot \a \nabla h$, the proof of the result follows from the following two arguments:
\begin{enumerate}[label=(\roman*)]
    \item First, one shows that the function $h$ is (quantitatively) close to the function $\theta^{-1}(\p) \bar p$. This is achieved by proving that the second term in the right side of~\eqref{2scexporgproo} is small and relies on the quantitative sublinearity of the corrector stated in~\eqref{quantsubcorrexpoproo}.
    \item Second, one needs to show that the function $h$ is (quantitatively) close to the heat kernel $p$. To prove this, the strategy is to use that the map $p$ solves the parabolic equation
    \begin{equation*}
        \partial_t p - \nabla \cdot \a \nabla p = 0,
    \end{equation*}
    and subtract it from~\eqref{heateqh.espoproo} to obtain that $ \partial_t (p - h) - \nabla \cdot \a \nabla (p- h)$ is small in the $\aH^{-1}$ norm. We then use the function $(p - h)$ as a test function in the previous equation, to deduce that $(p - h)$ has to be small in the $H^{1}$-norm.
\end{enumerate}

This strategy is essentially carried out in Section~\ref{section4.2}. Nevertheless, a number of difficulties have to be treated in order to implement it. They are mainly due to three distinct causes which are listed below.

First, the heat kernel $p$ has an initial condition at time $t=0$ which is a Dirac (see the equation~\eqref{intdefheatker}). It is rather singular and causes serious troubles in the analysis. To fix this issue, one replaces the initial condition in~\eqref{intdefheatker} by a function which is smoother, but which is still a good approximation of the Dirac function. The argument is sketched in the following paragraph. We fix a large time $t> 0$ and want to prove the main estimate~\eqref{mainthmmainest} for this particular time $t$. To this end, we replace the initial condition $\delta_y$ by the function $\bar p (\tau , \cdot- y )$ for some time $\tau \ll t$, and we define 
    \begin{equation} \label{def.fctqintro}
    \left\{ \begin{array}{lcl}
    \partial_t q - \nabla \cdot \left(  \a \nabla q \right) = 0 &\mbox{in} ~(\tau , \infty) \times \C_\infty, \\
    q(\tau , \cdot,y) = \theta(\p)^{-1}\bar p(\tau , \cdot - y) &\mbox{on}~\C_\infty.
    \end{array} \right.
\end{equation}
The strategy is then to make the following compromise: we want to choose the coefficient $\tau$ small enough 
(in particular, much smaller than $t$) so that the initial data $\bar p(\tau , \cdot - y)$ is close to the Dirac function~$\delta_y$, the objective being that the function $q(t , \cdot , y)$ is close to $p(t , \cdot , y)$ (see Lemma~\ref{l.lemma4.2}); we also want to choose $\tau$ large enough so that the initial data $\bar p(\tau , \cdot)$ is smooth enough. Our choice will be $\tau = t^{1-\kappa}$ for some small exponent $\kappa > 0$. This approach is essentially the subject of Section~\ref{section4.1}.

The second difficulty is that the two-scale expansion described at the beginning of the section only yields the result for a small exponent, i.e., we obtain a result of the form
\begin{equation} \label{eq:2scasubotp}
    \left| p (t , x , y) - \theta(\p)^{-1} \bar p (t , x - y) \right| \leq C t^{-\kappa} t^{-\frac{d}{2}} \exp \left( - \frac{|x-y|^2}{Ct} \right),
\end{equation}
for a small exponent $\kappa > 0$. This result is much weaker than the near-optimal exponent $\frac 12 - \delta$ stated in Theorem~\ref{mainthm}. The strategy is thus to improve the value of the exponent by a bootstrap argument: by redoing the two-scale expansion and by using the estimate~\eqref{eq:2scasubotp} in the proof, we obtain an improved estimate of the form
\begin{equation} \label{eq:2scasubotp2}
    \left| p (t , x , y) - \theta(\p)^{-1} \bar p (t , x - y) \right| \leq C t^{-\kappa_1} t^{-\frac{d}{2}}\exp \left( - \frac{|x-y|^2}{Ct} \right),
\end{equation}
where $\kappa_1$ is a new exponent which is strictly larger than the original exponent $\kappa$. We can then redo the proof a second time and use the estimate~\eqref{eq:2scasubotp2} to obtain the inequality with an exponent $\kappa_2$ strictly larger than $\kappa_1$. An iteration of the argument shows that there exists an increasing sequence $\kappa_n$ such that, for each $n \in \N$, the following estimate holds 
\begin{equation} \label{eq:nearoptkappan}
    \left| p (t , x , y) - \theta(\p)^{-1} \bar p (t , x - y) \right| \leq C t^{-\kappa_n} t^{-\frac{d}{2}} \exp \left( - \frac{|x-y|^2}{Ct} \right).
\end{equation}
The sequence $\kappa_n$ is defined inductively (see the formula~\eqref{def:kappan}) and we can prove that it converges toward the value $\frac 12$; this is sufficient to prove the near optimal estimate stated in Theorem~\ref{mainthm}.

 The third difficulty is the degenerate structure of the environment. It is treated by defining a renormalization structure for the infinite cluster which was first introduced in~\cite{AD2}: building upon standard results in supercritical percolation, we construct a partition of the lattice $\Zd$ into cubes of different random sizes which are well-connected in the sense of Antal, Penrose and Pisztora (see~\cite{AP, PP}), using a Calder\'on-Zygmund type stopping time argument. The sizes of the cubes of the partition are random variables which measure how close the geometry of the cluster is from the geometry of the lattice: in the regions where the sizes of the cubes are small, the cluster is well-behaved and its geometry is similar to the one of the Euclidean lattice, while in the regions where the sizes of the cubes are large, the geometry of the cluster is ill-behaved (see Figure~\ref{fig:my_labelPa}). The probability to have a large cube in the partition is small and stretched exponential integrability estimates are available for these random variables (see Proposition~\ref{p.partitions} (iii) or~\cite{PP}). 
 
 This partition provides a random scale above which the geometry of the infinite cluster is similar to the one of the Euclidean lattice and it allows to adapt the tools of functional analysis needed to perform the two-scale expansion to the percolation cluster. Similar strategies using renormalization techniques where used to study random walk on the supercritical percolation cluster and we refer for instance to the work of Barlow in~\cite{Ba}, who established a Poincar\'e inequality on the percolation cluster, or to the one of Mathieu and Remy in~\cite{MR}.
 
 The general strategy to study the random walk on the infinite cluster is thus to prove that there exists a random scale above which the geometry of the infinite cluster $\C_\infty$ is similar to the geometry of the lattice $\Zd$, and to deduce from it that, above a random time which is related to the aforementioned random scale, the random walk has a behavior which is similar to the one of the random walk on $\Zd$. As a consequence, most of the results described in this article only hold above a random scale (or random time) above which the infinite cluster has renormalized. Moreover, we need to appeal to a number of random scales (or random times) in the proofs, above which some analytical tools are available: the scale $\M_{\mathrm{reg}}$ above which a $C^{0, 1}$-regularity theory is valid (see Theorem~\ref{gradbarlowintro}), the time $\T_{\mathrm{NA}}$ above which a Nash-Aronson estimate for the heat kernel is available (see Theorem~\ref{barlow}) etc. For all these random scales and times, stretched exponential integrability estimates are valid.

This strategy describes the proof of Theorem~\ref{mainthm}. Once this result is established, Theorem~\ref{mainthmell}, pertaining to the elliptic Green's function, can be deduced from it thanks to the Duhamel principle stated in~\eqref{introgreenellparfo}. This is the subject of Section~\ref{section5}.

We complete this section by describing the content and purposes of Section~\ref{section3}. To perform the analysis described in the previous paragraphs, and in particular to prove that the function $q$ defined~\eqref{def.fctqintro} is a good approximation of the heat kernel $p$, one needs to have some control over the quantities at stake. In particular, it is useful to have a good control on the heat kernel $p$ and its gradient~$\nabla p$. The first one is given by the article of Barlow~\cite{Ba}, which provides Gaussian upper and lower bounds for the heat kernel $p$ (see Theorem~\ref{barlow}). For the gradient of the heat kernel, we expect to have a behavior similar to the one of the gradient of the heat kernel on $\Rd$, i.e., a $C^{0,1}$-regularity estimate of the form
\begin{equation*}
    \left| \nabla_x p \left( t , x , y \right) \right| \leq C t^{- \frac{d}{2} -\frac{1}{2}} \exp \left( - \frac{|x-y|^2}{Ct} \right).
\end{equation*}
Section~\ref{section3} is devoted to proving a large-scale version of this estimate and is independent of Section~\ref{section4} and Section~\ref{section5}. The precise statement established in this section is the following.

\begin{theorem}  \label{gradbarlowintro}
There exist an exponent $s := s \left( d , \p, \lambda \right) > 0$, a positive constant $C := C(d , \p, \lambda ) < \infty$ such that for each point $x \in \Zd$, there exists a non-negative random variable $\M_{\mathrm{reg}}(x)$ satisfying the stochastic integrability estimate
\begin{equation} \label{eq:stochintMell}
\forall R \geq 0, ~\P \left( \M_{\mathrm{reg}}(x) \geq R  \right) \leq C \exp \left( - \frac{R^s}{C} \right),
\end{equation}
such that the following statement is valid: for every radius $r \geq \M_{\mathrm{reg}}(x)$, every point $y \in \C_\infty$ and every time $t \geq \max \left(4 r^2, |x - y|\right)$, the following estimate holds,
\begin{equation*}
\left\| \nabla_x p \left( t , \cdot , y \right) \right\|_{\underline{L}^2 \left( B_r\left(x\right) \cap \C_\infty\right)} \leq C t^{- \frac{d}{2} -\frac{1}{2}} \exp \left( - \frac{|x-y|^2}{Ct} \right),
\end{equation*}
where the notation $\underline{L}^2 \left( B_r\left(x\right) \cap \C_\infty\right)$ denotes the average $L^2$-norm over the set $ B_r\left(x\right) \cap \C_\infty$ and is defined in~\eqref{eq:LpGradient}.
\end{theorem}

\begin{remark} \label{remark1.3}
By using the symmetry of the heat kernel, a similar regularity estimate holds for the gradient in the second variable: for each point $y \in \Zd$, there exists a non-negative random variable $\M_{\mathrm{reg}}(y)$ satisfying the stochastic integrability estimate~\eqref{eq:stochintMell} such that for every radius $r \geq \M_{\mathrm{reg}}(y)$, every point $x \in \C_\infty$ and every time $t \geq \max \left(4 r^2, |x - y|\right)$,
    \begin{equation*}
\left\| \nabla_y p \left( t , x , \cdot \right) \right\|_{\underline{L}^2 \left( B_r\left(y\right) \cap \C_\infty\right)} \leq C t^{- \frac{d}{2} -\frac{1}{2}} \exp \left( - \frac{|x-y|^2}{Ct} \right).
\end{equation*}
\end{remark}

The strategy of the proof of this result relies on tools from homogenization theory, in particular the two-scale expansion and the large-scale regularity theory. It is described at the beginning of Section~\ref{section3}.

\subsection{Related results}

\subsubsection{Related results about the random conductance model}

The random conductance model has been the subject of active research over the recent years, by various authors and under different assumptions over the law of the environment. In the case of uniform ellipticity, i.e., when the environment is allowed to take values in $[\lambda, 1]$, a quenched invariance principle is proved by Osada in~\cite{Os82} (in the continuous setting) and by Sidoravicius and Sznitman in~\cite{SS} (in the discrete setting). Gaussian bounds on the heat kernel follow from~\cite{D99}. This framework is the one of the theory of stochastic homogenization and we refer to Section~\ref{section1.3.2} for further information.

In the setting when the conductances are only bounded from above, a quenched invariance principle was proved by Mathieu in~\cite{M} and by Biskup and Prescott in~\cite{BP}. In the case when the conductances are bounded from below, a quenched invariance principle and heat kernel bounds are proved in~\cite{BD} by Barlow and Deuschel. In~\cite{ABDH}, Andres, Barlow, Deuschel and Hambly established a quenched invariance principle in the general case when the conductances are allowed to take values in $[0 , \infty)$.

The i.i.d.\ assumption on the environment can be relaxed: in~\cite{andres2015invariance}, Andres, Deuschel and Slowik proved a quenched invariance principle for the random walk for general ergodic environment under the moment condition
\begin{equation} \label{momentcondpq2d}
    \E \left[ \a(e)^p \right] + \E \left[ \a(e)^{-q} \right] < \infty ~\mbox{for}~ p,q \in (1 , \infty)~\mbox{satisfying}~ \frac 1p + \frac 1q < \frac 2d.
\end{equation}
We also refer to the works of Chiarni, Deuschel~\cite{CD16}, Deuschel, Nguyen, Slowik~\cite{DNS18} and Bella and Sch\"affner~\cite{BS19} for additional quenched invariance principles in degenerate ergodic environment. The case of ergodic, time-dependent, degenerate environment is investigated by Andres, Chiarini, Deuschel, and Slowik in~\cite{ACDS} where they establish a quenched invariance principle under some moment conditions on the environment. More general models of random walks on percolation clusters with long range correlation, including
random interlacements and level sets of the Gaussian free field, are studied by Procaccia, Rosenthal
and Sapozhnikov in~\cite{procaccia2016quenched}, where a quenched invariance principle is established.

The heat kernel has been studied under various assumptions on the environment: a first important property that needs to be investigated is the question of the existence of Gaussian lower and upper bounds. Such estimates are valid in the case of the percolation cluster presented in this article and were originally proved by Barlow in~\cite{Ba}. This result also holds when the conductances are bounded from below and we refer to the works of Mourrat~\cite{mourrat2011variance} (Theorem 10.1 of the second arxiv version) and of Barlow, Deuschel~\cite{BD}. It is also known that it cannot hold in full generality: in~\cite{BBHK}, Berger, Biskup, Hoffman and Kozma established that, when the law of the conductances has a fat tail at $0$, the heat kernel can behave anomalistically due to trapping phenomenon (even though a quenched invariance principle still holds by~\cite{ABDH}). We refer to the works of Barlow, Boukhadra~\cite{BBo} and Boukhadra~\cite{Bou10, Bou09} for additional results in this direction. Gaussian estimates on the heat kernel for more general graphs were studied by Andres, Deuschel and Slowik in~\cite{andres2016heat} and~\cite{ADS19}.

The question of Gaussian upper and lower bounds on the heat kernel is related, and in many situations equivalent, to the existence of a parabolic Harnack inequality (see for instance Delmotte~\cite{D99}). On the percolation cluster, the parabolic Harnack inequality is established in~\cite{BH}. We refer to the article of Andres, Deuschel, Slowik~\cite{ADS16} for a proof of elliptic and parabolic Harnack inequalities on general graphs with unbounded weights, to the work of Sapozhnikov~\cite{Sap13} for a proof of quenched heat kernel bounds and parabolic Harnack inequality for a general class of percolation models with long-range correlations on $\Zd$ and to the articles of Chang~\cite{Ch15} and Alves and Sapozhnikov~\cite{AS18} for similar results on loop soup models.

Results on the elliptic Green's function usually follow from the ones established on the parabolic Green's function, by an application of the formula~\eqref{introgreenellparfo} in dimension larger than $3$. In dimension $2$ the situation is different and requires separate considerations; in~\cite{andres2018green}, Andres, Deuschel and Slowik characterize the asymptotics of the Green's function associated to the random walk killed upon exiting a ball under general assumptions on the environment.

Finally, we refer to~\cite{biskupsurvey} for a general review on the random conductance model.

\subsubsection{Related result about stochastic homogenization} \label{section1.3.2}

The theory of qualitative stochastic homogenization was developed in the 80's, with the works of Kozlov~\cite{K1}, Papanicolaou and Varadhan~\cite{PV1} and Yurinski\u\i~\cite{Y1} in the uniformly elliptic setting. Still in the uniformly elliptic setting, a quantitative theory of stochastic homogenization has been developed in the recent years up to the point that it is now well-understood thanks to the works of Gloria and Otto in~\cite{GO1, GO2, GO15, GO115} and Gloria, Neukamm, Otto~\cite{GNO, GO14}, building upon the ideas of Naddaf and Spencer in~\cite{NS}. These results have applications to random walks in random environment, as is explained in~\cite{EGMN}. Another approach was initiated by Armstrong and Smart in~\cite{AS}, who extended the techniques of Avellaneda and Lin~\cite{AL1, AL2} and the ones of Dal Maso and Modica~\cite{DM1, DM2}. These results were then improved in~\cite{AKM1, AKM2}, and we refer to the monograph~\cite{armstrong2017quantitative} for a detailed review of this approach.

The aforementioned works treated the case of uniformly elliptic environments and the question of the extension of the theory to degenerate environments has drawn some attention over the past few years. A number of results have been achieved and some of them are closely related to the works on the random conductance model presented in the previous section. In~\cite{neukamm2017stochastique}, Neukamm, Sch\"affner and Schl\"omerkemper proved $\Gamma$-convergence of the Dirichlet energy associated to some nonconvex energy functionals with degenerate growth. In~\cite{LNO}, Lamacz, Neukamm and Otto studied a model of Bernoulli bond percolation, which is modified such that every bond in a fixed direction is declared open. In~\cite{flegel2017homogenization}, Fleger, Heida and Slowik proved homogenization results for a degenerate random conductance model with long range jumps. In~\cite{bella2018liouville}, Bella, Fehrman and Otto studied homogenization of degenerate environment under the moment condition~\eqref{momentcondpq2d} and established a first-order Liouville theorem as well as a large-scale $C^{1 , \alpha}$-regularity estimate for $\a$-harmonic functions.  In~\cite{GHV18},  Giunti, H\"ofer and Vel\'azquez studied homogenization for the Poisson equation in a randomly perforated domain. In~\cite{AD2}, Armstrong and the first author implemented the techniques of~\cite{armstrong2017quantitative} to the percolation cluster to obtain quantitative homogenization results as well as a large-scale regularity theory.

\subsection{Further outlook and conjecture}

The results of this article present quantitative rates of convergence for the parabolic and elliptic Green's functions on the percolation cluster. We do not expect the result to be optimal: the quantitative rate of convergence $\frac{1}{2} - \delta$ and the stochastic integrability~$s$ in Theorem~\ref{mainthm} can be improved and so is the case for Theorem~\ref{mainthmell}. We expect the following conjecture to hold.

\begin{conjecture}
Fix $s \in \left(0,\frac{2(d-1)}{d} \right)$, there exists a positive constant $C < \infty$ depending on the parameters $d , \p , \lambda $ and $s$, such that, for each time $t > 0$ and each pair of points $x , y \in \Zd$ such that $|x - y| \leq t$, conditionally on the event $\{x , y \in \C_\infty\}$,
\begin{equation*}
    \left| p (t , x , y) - \theta(\p)^{-1} \bar p (t , x - y) \right| \leq \left\{ \begin{array}{lcl} 
    \O_s \left( C t^{-\frac{d}{2} - \frac{1}{2}} \exp \left( - \frac{|x-y|^2}{Ct} \right) \right) & ~\mbox{when}~ d \geq 3, \\
    \O_s \left( C \log^{\frac 12} \left( 1 + t \right) t^{-\frac{3}{2}} \exp \left( - \frac{|x-y|^2}{Ct} \right) \right) & ~\mbox{when}~ d =2,
    \end{array} \right.
\end{equation*}
where the notation $\O_s$ is used to measure the stochastic integrability and is defined in Section~\ref{section1.6.11}. For the elliptic Green's function, a similar result holds:
\begin{enumerate}
 \item In dimension $d \geq 3$, for each $x , y \in \Zd$, conditionally on the event $\{x, y \in \C_\infty\}$,
\begin{equation*}
        \left| g( x , y) -  \bar g (x-y) \right|  \leq \O_s \left( C |x - y|^{1-d} \right).
\end{equation*}
where the function $\bar g$ is defined in the equation~\eqref{eq:ghom}.
\item In dimension $2$, for each $y \in \Zd$, conditionally on the event $\{ y \in \C_\infty \}$, the limit
\begin{equation*}
    K(y):= \lim_{x \rightarrow \infty} g(x , y) -  \bar g (x-y)
\end{equation*}
exist, is finite almost surely and satisfies the stochastic integrability estimate
\begin{equation*}
    \left| K(y) \right| \leq \O_s \left( C \right).
\end{equation*}
Moreover, for every $x \in \Zd$, conditionally on the event $\{ x, y \in \C_\infty \}$, one has
\begin{equation*}
    \left| g ( x , y) -  \bar g (x-y) - K(y) \right| \leq \O_s \left( C \log^{\frac 12} \left( 1 + |x-y| \right)|x - y|^{-1} \right).
\end{equation*}
\end{enumerate}
\end{conjecture}

\begin{remark}
This statement cannot be stated with a minimal scale as in Theorems~\ref{mainthm} and~\ref{mainthmell}. This is due to the fact that the estimates scale optimally in time or space; the best possible statements involving a minimal scale are the ones of Theorems~\ref{mainthm} and~\ref{mainthmell}.
\end{remark}

This result can be conjectured from the theory of stochastic homogenization in the uniformly elliptic setting (see~\cite[Theorem 9.11 and Corollary 9.12]{armstrong2017quantitative}). There is one main difference between the results in the uniformly elliptic setting and in the percolation setting, which is the stochastic integrability: we expect that the stochastic integrability will be reduced by a factor $(d-1)/d$. This is expected because of a surface order large deviation effect which can be heuristically explained as follows. In the uniformly elliptic setting and in a given ball $B_R$, to design a bad environment, i.e., an environment on which no good control on the heat kernel is valid, it is necessary to have a number of ill-behaved edges of order of the volume of the ball. In the percolation setting, one can design a bad environment with a number of ill-behaved edges of order of the surface of the ball: given a ball of size $R$, it is possible to disconnect it into two half-balls with $c R^{d-1}$ closed edges. This should result in a deterioration of the stochastic integrability by a factor $(d-1)/d$.

\smallskip

The conjecture improves Theorems~\ref{mainthm} and~\ref{mainthmell} in two distinct directions: the spatial scaling, where the coefficient $1/2 - \delta$ is replaced by $1/2$ for the heat kernel and the coefficient $1-\delta$ is replaced by $1$ for the elliptic Green's function, and the stochastic integrability, where the exponent $s$ can take any value in the interval $ \left(0,\frac{2(d-1)}{d} \right)$. We believe that the two improvements should follow from different techniques: for the spatial integrability, we think that it should follow by an adaptation of the techniques developed in~\cite[Chapter 9]{armstrong2017quantitative}. The improvement of the stochastic integrability seems to be a much harder problem which requires separate considerations and should rely on a precise understanding of the geometry of the percolation clusters.

We complete this section by mentioning that the results of this article pertain to the variable speed random walk, but similar results, with similar proofs, should hold for other related models of random walk on the infinite cluster such as the constant speed random walk and the simple random walk. This choice is motivated by the fact that the generator of the VSRW, written in~\eqref{generatorVSRW}, is more convenient to work with than the ones of the CSRW and the SRW, which simplifies the analysis.

\subsection{Organization of the article}
The rest of this article is organized as follows. The remaining section of this introduction is devoted to the presentation of some useful notations. 

In Section~\ref{section2}, we record some preliminary results, including some results from the theory of quantitative stochastic homogenization on the infinite cluster from~\cite{AD2, dario2018optimal, gu2019efficient}: the quantitative sublinearity of the corrector and a quantitative estimate to control the $H^{-1}$-norm of the centered flux. In Section~\ref{section3}, we recall the Gaussian bounds on the heat kernel which were established by Barlow in~\cite{Ba} and establish a large-scale $C^{0, 1}$-regularity theory for the heat kernel.

In Section~\ref{section4}, we establish Theorem~\ref{mainthm}. The proof is organized in three subsections: Section~\ref{section4.1} is devoted to the proofs of three regularization steps, which can be seen as a preparation for the two-scale expansion in Section~\ref{section4.2}. The heart of the proof is Section~\ref{section4.2}, where we perform the two-scale expansion. In Section~\ref{section4.3}, we post-process the result from Section~\ref{section4.2} and deduce the result of Theorem~\ref{mainthm}.

In Section~\ref{section5}, we use Theorem~\ref{mainthm} to prove the homogenization of the elliptic Green function, i.e., Theorem~\ref{mainthmell}. 

Appendix~\ref{appendixB} and Appendix~\ref{appendixb} are devoted respectively to two technical estimates: a concentration inequality for the density of the infinite cluster in a cube and the proof of the quantitative estimate of the weak $H^{-1}$-norm of the centered flux which is stated in Section~\ref{homogpercclu}.

\subsection{Notation and assumptions}

\subsubsection{General notations and assumptions}\label{section1.6.1}
We let~$\Zd$ be the standard~$d$-dimensional hypercubic lattice and~$\Bd:=\left\{ \{ x,y\}\,:\, x,y\in\Zd, |x-y|=1 \right\}$ denote the set of bonds. We also denote by $\Bdo$ the set of \emph{oriented} bonds, or \emph{edges}, of $\Zd$. We use the notation $\{x,y\}$ to refer to a bond and $(x, y)$ to refer to an edge.

\medskip

We denote the canonical basis of $\Rd$ by $\{e_1,\ldots,e_d\}$. For a vector $p \in \Rd$ and an integer $i \in \{ 1 , \ldots, d \}$, we denote by $[p]_i $ its $i$th-component, i.e., $p = \left([p]_1 , \ldots, [p]_d \right)$. For~$x,y\in\Zd$, we write~$x\sim y$ if~$x$ and~$y$ are nearest neighbors. We usually denote a generic edge by~$e$. We fix an ellipticity parameter $\lambda \in (0,1]$ and denote by~$\Omega$ the set of all functions $\a: \Bd \to \{ 0\} \cup [\lambda,1]$, i.e.,~$\Omega = \left( \{ 0\} \cup [\lambda,1] \right)^{\Bd}$ and we denote by $\a$ a generic element of~$\Omega$. The Borel~$\sigma$-algebra on $\Omega$ is denoted by $\F$. For each $U \subseteq \Zd$, we let $\F(U)\subseteq\F$ denote $\sigma$-algebra generated by the projections $\a\mapsto \a(\{x,y\})$, for $x,y\in U$ with $x \sim y$.

\smallskip

We fix an i.i.d.~probability measure~$\P$ on $(\Omega,\F)$, that is, a measure of the form $\P = \P_0^{\Bd}$ where $\P_0$ is a measure of probability supported in the set $\{0\}\cup[\lambda,1]$ with the property that, for any fixed bond $e$,
\begin{equation*} \label{}
\p := \P_0 \left[ \a(e) \neq 0 \right] > \pc(d),
\end{equation*}
where $\pc(d)$ is the bond percolation threshold for the lattice $\Zd$. We say that a bond $e$ is \emph{open} if $\a(e) \neq 0$ and \emph{closed} if $\a(e) = 0$. A connected component of open edges is called a \emph{cluster}. Under the assumption $\p > \pc(d)$, there exists almost surely a unique maximal infinite cluster, which is denoted by $\C_\infty$ and we also note $\theta(p) := \P[0 \in \C_\infty]$. From now on, we always consider environments $\a \in \Omega$ such that there exists a unique infinite cluster of open edges. We denote by~$\E$ the expectation with respect to the measure~$\P$.

\subsubsection{Notation of $\O_s$} \label{section1.6.11}
For a parameter $s > 0$, we use the notation $\O_s$ to measure the stochastic integrability of random variables. It is defined as follows, given a random variable $X$, we write
\begin{equation}\label{eq:ODef}
X \leq \O_s(\theta) ~\mbox{if and only if}~ \E\left[\exp((\theta^{-1}X)_{+}^s)\right] \leq 2,
\end{equation} 
where $(\theta^{-1}X)_+$ means $\max\{\theta^{-1}X, 0\}$. From the inequality~\eqref{eq:ODef} and the Markov's inequality, one deduces the following estimate for the tail of the random variable $X$: for all $x > 0$, ${\P[X \geq \theta x] \leq 2 \exp(-x^s)}.$

\smallskip

Given a random variable $X$ satisfying the identity $X \leq \O_s(\theta)$, one can check that, for each $\lambda \in \mathbb{R}^+$, one has $\lambda X \leq \O_s(\lambda \theta)$. Additionally, one can reduce the stochastic integrability parameter~$s$ according to the following statement: for each $s' \in (0 ,s],$ there exists a constant $0 < C_{s'} < \infty$ such that $X \leq \O_{s'}(C_{s'}\theta).$ 

\smallskip

To estimate the stochastic integrability of a sum of random variables, we use the following estimate, which can be found in~\cite[Lemma A.4 of Appendix A]{armstrong2017quantitative}: for each exponent $s>0$ there exists a positive constant $C_s < \infty$ such that for any measure space $(E, \mathcal{S}, m)$ and any family of random variables $\{X(z)\}_{z \in E}$, one has 
\begin{equation}
\label{eq:OSum}
\forall z \in E, X(z )\leq \O_s(\theta(z)) \Longrightarrow \int_{E} X(z) m(dz) \leq \O_s\left( C_s \int_{E} \theta(z) m(dz) \right).
\end{equation}
The previous statement allows to estimate the stochastic integrability of a sum of random variables: given $X_1, \ldots , X_n$ a collection of non-negative random variables and $C_1 , \ldots, C_n$ a collection of non-negative constants such that, for any $i \in \{ 1 , \ldots , n \}, \, X_i \leq \O_s \left( C_i \right),$ one has the estimate
\begin{equation} \label{eq:OSumm}
    \sum_{i=1}^n X_i \leq \O_s \left( C_s \sum_{i = 1}^n C_i \right).
\end{equation}
The following lemma is
useful to construct minimal scales.
\begin{lemma}\cite[Lemme 2.2]{AD2} \label{lemma1.5}
Fix $K \geq 1, s >0$ and $\beta > 0$ and let $\{ X_n \}_{n \in \N}$ be a sequence of non-negative random variables satisfying the inequality $X_n \leq \O_s(K3^{-n\beta})$ for every $n \in \N$. There exists a positive constant $C(s, \beta, K)<\infty$ such that the random scale $M := \sup \{3^n \in \N : X_n \geq 1\}$ satisfies the stochastic integrability estimate $M \leq \O_{\beta s}(C)$.
\end{lemma}

\smallskip

\subsubsection{Topology, functions and integration}
For every subset $ V \subset \Zd$ and every environment $\a \in \Omega$, we consider two sets of bonds $\Bd(V)$ and $ \Bda(V)$. The first one is inherited from the set of bonds $\Bd$ of $\Zd$, the second one is inherited from the bonds of non-zero conductance of the environment $\a$. They are defined by the formulas  
$$
\Bd(V) := \left\{\{x,y\} \, : \, x,y \in V, \, x \sim y \right\}, \qquad \Bda(V) : =  \left\{\{x,y\} \, : \, x,y \in V, \, x \sim y , \, \a(\{x,y\}) \neq 0 \right\}.
$$
We similarly define the set of edges $\Bdo(V)$ and $\overrightarrow{\Bd^\a} (V)$.

\smallskip

The \textit{interiors} of a set $V$ with respect to $\Bd(V)$ and $\Bda(V)$ are defined by the formulas
\begin{equation*}
\itr(V) := \{x\in V \, : \, y \sim x \implies y \in V\}, \qquad \itr_{\a}(V) := \{x\in V \, : \, y \sim x, \, \a(\{x,y\}) \neq 0 \implies y \in V\},
\end{equation*}
and the \textit{boundaries} of $V$ are defined by $\partial V := V \backslash \itr(V)$ and $\partial_{\a} V := V \backslash \itr_{\a}(V)$. The cardinality of a subset $V \subseteq \Zd$ is denoted by $|V|$ and called the \emph{volume} of $V$. Given two sets $U,V \subseteq \Zd$, we define the distance between $U$ and $V$ according to the formula $\dist(U,V) := \min_{x \in U , y \in V} |x-y|$ and the distance of a point $x \in \Zd$ to a set $V \subseteq \Zd$ by the notation $\dist(x,V) := \min_{ y \in V} |x-y|$.

For a subset $V \subseteq \Zd$, the spaces of functions with zero boundary condition are defined by 
\begin{equation}\label{eq:C0}
C^{}_0(V) := \left\{v:V \rightarrow \mathbb{R} \, : \, v=0 \text{ on } \partial V \right\}, \qquad
C^{\a}_0(V) := \left\{v:V \rightarrow \mathbb{R} \, : \, v=0 \text{ on } \partial_{\a} V \right\}.
\end{equation}

 Given a subset $U \subset \Zd$ and a function $u : \C_\infty \cap U \rightarrow \R$ (resp. a function $F : \Bda(\C_\infty \cap U) \rightarrow \R$), the integration over the set $\C_\infty \cap U$ (resp. over $\Bda(\C_\infty \cap U)$) is denoted by  
\begin{equation}
\int_{\C_\infty \cap U} u := \sum_{x \in \C_\infty \cap U} u(x), \qquad~\mbox{resp.}~ \int_{\C_\infty \cap U} F := \sum_{e \in \Bda(\C_\infty \cap U)} F(e),
\end{equation}
which means that we only integrate on the vertices (resp. open bonds) of the infinite cluster $\C_\infty$. We extend this notation to the setting of vector-valued functions $u : \C_\infty \cap U \rightarrow \R^n$. We also let $\left( u \right)_V := \frac1{|V|}\int_V u$ denote the mean of the function $u$ over the finite subset $V \subset \Zd$.

\smallskip

Given a subset $V \subseteq \Zd$, a vector field is a function $\overrightarrow{F} : \Bdo(V) \rightarrow \R$ satisfying the anti-symmetry property $$\overrightarrow{F}(\{x,y\}) = - \overrightarrow{F}(\{y,x\}).$$

For $y \in \Zd$ and $r >0$, we denote by $B_r(y)$ the discrete Euclidean ball of radius $r > 0$ and center $y$; we often write $B_r$ in place of $B_r(0)$. A \emph{cube} $Q$ is a subset of $\Zd$ of the form
$$ Q := \Zd \cap \left( z + \left[ -N , N  \right]^d \right).
$$
We define the \emph{center} and the \emph{size} of the cube given in the previous display above to be the point $z$ and the integer $N$ respectively. The size of the cube $Q$ is denoted by $\size \left( Q \right)$. Given an integer $n \in \N$, we use the non-standard convention of denoting by $nQ$ the cube
\begin{equation} \label{eq:nonstnt}
    nQ := \Zd \cap \left( z + \left[ -nN , nN  \right]^d \right).
\end{equation}
A \emph{triadic cube} is a cube of the form
\begin{equation*}
    \cu_m(z) := z + \left( - \frac{3^m}2 ,  \frac{3^m}2 \right)^d, ~ z \in 3^m \Zd, \, m \in \N.
\end{equation*}
We usually write $\cu_m := \cu_m (0)$. Additionally, we note that $\size(\cu_m) = 3^m$, denote by ${\T_m := \left\{ z + \cu_m \, : \, z \in 3^m \Zd \right\}}$ the set of triadic cubes of size $3^m$ and by $\T$ the set of all triadic cubes, i.e., $\T := \cup_{m \in \N} \mathcal{T}_m$.

\smallskip

\subsubsection{Discrete analysis and function spaces}\label{Discanalandfunspa} In this article, we consider two types of objects: functions defined in the continuous space $\Rd$ and functions defined on the discrete space $\Zd$.

\medskip

\textit{Notations for discrete functions.}
Given a discrete subset $U \subseteq \Zd$, an environment $\a$ such that there exists an infinite cluster $\C_\infty$ of open edges, and a function $u : \C_\infty \cap U \to \R$, we define its gradient $\nabla u$ to be the vector field defined on $\Bdo$ by, for each edge $e = (x , y) \in \Bdo$,
\begin{equation} \label{eq:gradgraphnabla}
    \nabla u (e) := \left\{ \begin{array}{lcl}  
    u(y) - u(x) &~\mbox{if}~ x,y \in \C_\infty ~\mbox{and}~ \a(\{x,y \}) \neq 0, \\
    0 &~\mbox{otherwise.}
    \end{array}\right.
\end{equation}
For each $x \in \C_\infty$, we also define the norm of the gradient that ${|\nabla u|(x) := \sum_{y \sim x} \left| u(y) - u(x) \right|}$. We frequently abuse notation and write $|\nabla u(x)|$ instead of $|\nabla u|(x)$.

For a vector field $\overrightarrow{F} : \Bdo \rightarrow \R$, we define the discrete divergence operator according to the formula, for each $x \in \Zd$,
\begin{equation*}
    \nabla \cdot \overrightarrow{F}(x) := \sum_{y \sim x} \overrightarrow{F}(x,y).
\end{equation*}
By the discrete integration by parts, one has, for any discrete set $U \subset \Zd$, any functions $v \in C^{\a}_0(\C_\infty \cap U)$ and $u : \C_\infty \cap U \rightarrow \R$,
\begin{equation}\label{eq:Green}
\int_{\C_\infty \cap U} \nabla v \cdot \a \nabla u := \int_{\C_\infty \cap U}  \a(e) \nabla v(e) \nabla u(e) = \int_{\C_\infty \cap U} v(-\nabla \cdot \a \nabla u ),
\end{equation}
where the finite difference elliptic operator $- \nabla \cdot \a \nabla $ is defined in~\eqref{generatorVSRW}.

\smallskip

For $p \in [1, \infty)$, we define the $L^p(\C_\infty \cap U)$-norm  and the normalized $\aL^p(\C_\infty \cap U)$-norm by the formulas
\begin{equation}\label{eq:Lp}
\norm{u}_{L^p(\C_\infty \cap U)} := \left(\int_{\C_\infty \cap U} \vert u \vert^p\right)^{\frac{1}{p}}, \qquad \norm{u}_{\aL^p(\C_\infty \cap U)} := \left(\frac{1}{\vert \C_\infty \cap U \vert}\int_{\C_\infty \cap U} \vert u \vert^p\right)^{\frac{1}{p}}. 
\end{equation}
We also define the $L^p(\C_\infty \cap U)$-norm and the normalized $\underline{L}^p(\C_\infty \cap U)$-norm of the gradient of a function $u : \C_\infty \cap U \to \R$ by the formulas
\begin{equation}\label{eq:LpGradient}
\norm{\nabla u}_{L^p(\C_\infty \cap U)} := \left(\int_{\C_\infty \cap U} \vert \nabla u \vert^p\right)^{\frac{1}{p}}, \qquad \norm{\nabla u}_{\aL^p(\C_\infty \cap U)} := \left(\frac{1}{\vert \C_\infty \cap U \vert}\int_{\C_\infty \cap U} \vert \nabla u \vert^p\right)^{\frac{1}{p}}.    
\end{equation}
We define the normalized discrete Sobolev norm $\aW^{1,p}(\C_\infty \cap U)$ by 
\begin{equation}\label{eq:W1p}
\norm{u}_{\aW^{1,p}(\C_\infty \cap U)} := \vert \C_\infty \cap U \vert^{-\frac{1}{d}} \norm{u}_{\aL^p(\C_\infty \cap U)} + \norm{\nabla u}_{\aL^p(\C_\infty \cap U)},
\end{equation}
and the dual norm $\underline{W}^{-1,p}\left( \C_\infty \cap U \right)$,
\begin{equation}\label{eq:WeakH1}
\norm{u}_{\underline{W}^{-1,p}\left( \C_\infty \cap U \right)} := \sup_{v \in C^{\a}_0(\C_\infty \cap U), \norm{v}_{\underline{W}^{1,p'}(\C_\infty \cap U)} \leq 1}\frac{1}{\vert \C_\infty \cap U \vert} \int_{\C_\infty \cap U} uv,
\end{equation}
with $\frac1{p'} + \frac1p = 1$.
We use the notation $\aH^1(\C_\infty \cap U) := \aW^{1,2}(\C_\infty \cap U)$ and $\aH^{-1}(\C_\infty \cap U) := \aW^{-1 , 2}(\C_\infty \cap U)$.

For a function $u : \Zd \to \R$ and a vector $h \in \Zd$, we denote by $T_h(u):= u(\cdot + h)$ the translation and by $\Dr{k}$ the finite difference operator defined by, for any function $u : \Zd \to \R$,
\begin{equation*}
\Dr{k}u := \left\{ \begin{array}{lcl}  
    \Zd &\to &\R, \\
    x &\mapsto &T_{e_k} (u)(x) - u(x).
    \end{array}\right.
\end{equation*}
We also define the vector-valued finite difference operator $\D u := \left(\Dr{1}u, \Dr{2}u, \cdots \Dr{d}u \right)$. This definition has two main differences with the gradient on graph defined in~\eqref{eq:gradgraphnabla}: it is defined on the vertices of $\Zd$ (not on the edges) and it is vector-valued. This second definition of discrete derivative is introduced because it is convenient in the two-scale expansion (see~\eqref{sec42scexp.1605}).

\smallskip

Given an environment $\a \in \Omega$, and a function $u : \C_\infty \to \R$, we define the functions $\a \D u$ and $\indc_{\{ \a \neq 0\}}\D u$ by, for each $x \in \C_\infty,$
\begin{align} \label{def.Dw,aDw4121}
\a \D u (x) &= \left(\a(\{x, x+e_1\})\Dr{1}u(x), \a(\{x, x+e_2\})\Dr{2}u(x), \ldots, \a(\{x, x+e_d\})\Dr{d}u(x) \right),    \\
\indc_{\{\a \neq 0\}} \D u (x) &= \left(\indc_{\{\a(\{x, x+ e_1\}) \neq 0\}}\Dr{1}u(x), \indc_{\{\a(\{x, x+ e_2\}) \neq 0\}}\Dr{2}u(x), \ldots, \indc_{\{\a(\{x, x+ e_d\}) \neq 0\}}\Dr{d}u(x) \right). \notag
\end{align}
We extend these functions to the entire space $\Zd$ by setting, for each point $x \in \Zd \setminus \C_\infty$, 
$${\a \D u(x) = \indc_{\{\a \neq 0\}} \D u(x) = 0}.$$
It is natural to introduce the dual operator $\Dr{k}^* u := T_{-e_k}(u) - u$ and the divergence $\D^* \cdot $ defined by, for any vector-valued function $\widetilde{F} : \Zd \rightarrow \Rd$, $\widetilde{F} = \left(\widetilde{F}_1 ,\widetilde{F}_2 ,\cdots \widetilde{F}_d  \right)$, $$\D^* \cdot \widetilde{F}(x) = \sum_{k = 1}^d \Dr{k}^* \widetilde{F}_k(x).$$
By the discrete integration by parts, one has the equality, for any $v \in C^{\a}_0(\C_\infty \cap U)$,
\begin{equation}\label{eq:Green2}
\int_{\C_\infty \cap U} v \left(\D^* \cdot \a \D u\right) = \int_{\C_\infty \cap U} \D v \cdot \a \D u.
\end{equation}
In fact one can check that the identity $-\nabla \cdot \a \nabla = \D^* \cdot \a \D$ holds, which allows to interchange the two notations. 

Moreover, given a vector $x = (x_1, \ldots, x_d) \in \Rd$, we denote by $|x| = \left(\sum_{j=1}^d x_j^2\right)^{\frac{1}{2}}$ its norm. This allows to extend the definition of the Sobolev norms~\eqref{eq:Lp},~\eqref{eq:W1p} and~\eqref{eq:WeakH1} to vector-valued functions, and we note that
$$c \norm{u}_{\aW^{1,p}(\C_\infty \cap U)} \leq  \vert \C_\infty \cap U \vert^{-\frac{1}{d}} \norm{u}_{\aL^p(\C_\infty \cap U)} + \norm{\indc_{\{\a \neq 0\}}\D u}_{\aL^p(\C_\infty \cap U)} \leq C \norm{u}_{\aW^{1,p}(\C_\infty \cap U)},$$
for some constants $c , C$ which only depend on the dimension $d$.

\medskip

\textit{Notations for continuous functions.} We use the notations $\partial_k$, $\nabla$, $\Delta$ for the standard derivative, gradient and Laplacian on $\Rd$, which are only applied to smooth functions. It will always be clear from context whether we refer to the continuous or discrete derivatives. We sometime slightly abuse the notation and denote by $\vert \nabla^k \eta \vert$ the norm of $k$-th derivatives of the function $\eta$.

\medskip

\textit{Notations for parabolic functions.} For $r > 0$, we define the time interval $I_r := \left( -r^2 , 0 \right]$ and $I_r(t) := t + \left( -r^2 , 0 \right]$. We frequently use the parabolic cylinders $I_r \times B_r$ and $I_r \times \left( \C_\infty \cap B_r \right)$ and define their volumes by $$\left| \left( I_r \times B_r \right) \right| = r^{2} \times |B_r|\hspace{7mm} \mbox{and} \hspace{7mm} \left| \left( I_r \times \left( \C_\infty \cap B_r\right) \right) \right| = r^{2} \times | \C_\infty \cap B_r|.$$
Given a function $u :  I_r \times B_r  \to \R$ (resp. $v :  I_r \times \left( \C_\infty \cap  B_r \right) \to \R$), we define the integrals
\begin{equation*}
    \int_{I_r \times B_r} u := \int_{-r^2}^0 \int_{B_r} u(t,x) \, dx \, dt \hspace{5mm} \mbox{and} \hspace{5mm} \int_{I_r \times \left( \C_\infty \cap B_r\right)} v := \int_{-r^2}^0 \int_{\C_\infty \cap B_r} v(t,x) \, dx \, dt,
\end{equation*}
and denote the mean of these functions by the notation 
\begin{align*}
\left( u \right)_{I_r \times B_r} & := \frac{1}{\left| \left( I_r \times B_r \right) \right|} \int_{-r^2}^0 \int_{B_r} u(t,x) \, dx \, dt ,\\
\left( v \right)_{I_r \times \left(\C_\infty \cap B_r \right)} &:= \frac{1}{\left| \left( I_r \times \left(\C_\infty \cap B_r \right) \right) \right|} \int_{-r^2}^0 \int_{B_r} v(t,x)  \, dx \, dt.    
\end{align*}
Given a finite subset $V \subseteq \Zd$ or $V \subset \Rd$, we denote by $\partial_{\sqcup}(I_r \times V)$ the parabolic boundary of the cylinder $I_r \times V$ defined by the formula
\begin{equation*}
    \partial_{\sqcup}(I_r \times V) := \left( I_r \times \partial V \right) \cup \left( \{ -r^2\} \times V \right).
\end{equation*}
Given a real number $p \geq 1$ and a Lebesgue-measurable function $u : I_r \times V \to \R^d$, we define the norm $L^p \left( I_r \times V \right)$ and the normalized norm $\underline{L}^p \left( I_r \times V \right)$ according to the formulas
\begin{equation*}
    \left\| u \right\|_{L^p \left( I_r \times V \right)} := \left( \int_{-r^2}^0 \left\| u (t , \cdot )\right\|_{L^p \left( V \right)}^p \,dt \right)^\frac1p  ~\mbox{and}~ \left\| u \right\|_{\underline{L}^p \left( I_r \times V \right)} := \left( r^{-2} \int_{-r^2}^0 \left\| u (t , \cdot )\right\|_{\underline{L}^p \left( V \right)}^p \,dt\right)^\frac1p .
\end{equation*}
These notations are extended to the gradient of a function $u : I_r \times V \to \R^d$ by the formulas
\begin{equation*}
    \left\| \nabla u \right\|_{L^p \left( I_r \times V \right)} := \left( \int_{-r^2}^0 \left\| \nabla u (t , \cdot )\right\|_{L^p \left( V \right)}^p \, dt \right)^\frac1p ~\mbox{and}~ \left\| \nabla u \right\|_{\underline{L}^p \left( I_r \times V \right)} := \left( r^{-2} \int_{-r^2}^0 \left\| \nabla u (t , \cdot )\right\|_{\underline{L}^p \left( V \right)}^p \, dt \right)^\frac1p.
\end{equation*}
Given a real number $q \geq 1$, we also define the space $L^q \left( I_r ; W^{-1 , p}(V) \right)$ by 
\begin{equation*}
    L^q \left( I_r ; W^{-1 , p}(V) \right) := \left\{ u : I_r \times V \to \R^n \, : \, \int_{-r^2}^0 \left\|u \right\|_{\underline{W}^{-1 , p}(V)}^q \, dt < \infty  \right\},
\end{equation*}
and we equip this space with the normalized norm defined by
\begin{equation*}
    \left\| u \right\|_{\underline{L}^q \left( I_r ; \underline{W}^{-1 , p}(V) \right)} := \left( r^{-2} \int_{-r^2}^0 \left\| u(t,\cdot) \right\|_{\underline{W}^{-1 , p}(V)} \, dt \right)^\frac 1q.
\end{equation*}
We define the parabolic Sobolev space $W^{1,p}_{\mathrm{par}} (I_r \times V)$ to be the set of measurable functions $u : I_r \times V \to \R$ such that the time derivative $\partial_t u$, understood in the sense of distributions, belongs to the space $W^{-1 , p'}(I_r \times V)$ with $\frac{1}{p} + \frac{1}{p'} = 1$, i.e.,
\begin{equation*}
    W^{1,p}_{\mathrm{par}}(I_r \times V) := \left\{ u \in L^{p}\left( I_r \times V \right) \, : \, \partial_t u \in L^{p'} \left( I_r ; W^{-1 , p'}(V) \right)  \right\}.
\end{equation*}
We also make use of the notations $H^1_{\mathrm{par}} \left( I_r \times V \right) := W^{1,2}_{\mathrm{par}}(I_r \times V)$ for the $H^1$ parabolic space and $\underline{L}^2 \left( I_r ; \underline{H}^{-1}(V)\right) := \underline{L}^2 \left( I_r ; \underline{W}^{-1 , 2}(V)\right)$.

\subsubsection{Convention for constants, exponents and minimal scales/times} Throughout this article, the symbols $c$ and $C$ denote positive constants which may vary from line to line. These constants may depend only on the dimension $d$, the ellipticity $\lambda$ and the probability $\p$. Similarly we use the symbols $\alpha, \, \beta, \, \gamma, \, \delta$ to denote positive exponents which depend only on $d$, $\lambda$ and $\p$. Usually, we use the letter $C$ for large constants (whose value is expected to belong to $[1, \infty)$) and $c$ for small constants (whose value is expected to be in $(0,1]$). The values of the exponents $\alpha, \, \beta, \, \gamma, \, \delta$ are always expected to be small. When the constants and exponents depend on other parameters, we write it explicitly and use the notation $C := C(d , \p , t)$ to mean that the constant $C$ depends on the parameters $d , \p$ and $t$. We also assume that all the minimal scales and times which appear in this article are larger than $1$.

\subsection{Acknowledgments} We would like to thank Jean-Christophe Mourrat and Scott Armstrong for helpful discussions and comments. The first author is supported by the Israel Science Foundation grants 861/15 and 1971/19 and by the European Research Council starting grant 678520 (LocalOrder).

\section{Preliminaries} \label{section2}

In this section, we collect a few results from the theory of supercritical percolation which are important tools in the establishment of Theorems~\ref{mainthm} and~\ref{mainthmell}.

\subsection{Supercritical percolation} \label{supercperco}

\subsubsection{A partition of good cubes} \label{sectionPGC}

An important step to prove results on the behavior of the random walk on the infinite cluster $\C_\infty$ consists in understanding the geometry of this cluster. A general picture to keep in mind is that the geometry of $\C_\infty$ is similar, at least on large scales, to the one of the Euclidean lattice $\Zd$. To give a precise mathematical meaning to this statement, the common strategy is to implement a renormalization structure for the infinite cluster. In this article, we use a strategy, which was first introduced by Armstrong and the first author in~\cite{AD2}. It relies on the following geometric definition and lemma which are due to Penrose and Pisztora~\cite{PP}.

\begin{definition}[Pre-good cube] \label{def.2.1}
We say that a discrete cube $\cu \subseteq \Zd$ of size $N$ is \emph{pre-good} if:
\begin{enumerate}[label=(\roman*)]
    \item There exists a cluster of open edges which intersects the $2d$ faces of the cube $\cu$. This cluster is denoted by $\C_* \left( \cu \right)$;
    \item The diameter of all the other clusters is smaller than $\frac{N}{10}$.
\end{enumerate}
\end{definition}

\begin{figure}[h!]
    \centering
    \includegraphics[scale= 0.7]{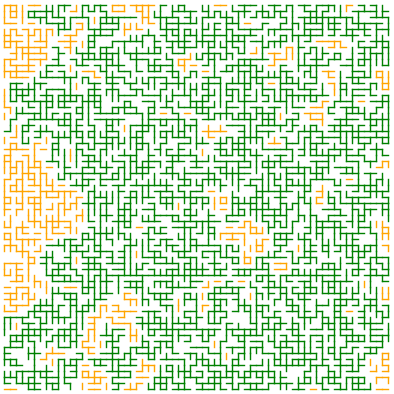}
    \caption{A pre-good cube, the cluster $\C_*(\cu)$ is drawn in green and the clusters in yellow are the small clusters.}
    \label{fig:my_label}
\end{figure}

We then upgrade this definition into the following definition of good cubes.
\begin{definition}[Good cube]
We say that a discrete cube $\cu \subseteq \Zd$ of size $N$ is \emph{good} if:
\begin{enumerate}[label=(\roman*)]
    \item The cube $\cu$ is pre-good;
    \item Every cube $\cu'$ whose size is between $N/10$ and $N$ and which has non-empty intersection with $\cu$ is also pre-good.
\end{enumerate}
\end{definition}

We note that the event ``the cube $\cu$ is good" is $\mathcal{F}\left( 2\cu\right)$-measurable. The main reason to use good cubes instead of pre-good cubes is that they satisfy the following connectivity property, which can be obtained from straightforward geometric considerations and whose proof can be found in~\cite[Lemma 2.8]{AD2}.
\begin{lemma}[Connectivity property]
\label{l.connectivity}
Let $\cu_1, \cu_2$ be two cubes of $\Zd$ which are neighbors, i.e., which satisfy
\begin{equation*} \label{}
\dist\left(\cu_1, \cu_2  \right) \leq 1,
\end{equation*}
which have comparable size in the sense that
\begin{equation*}
    \frac 13 \leq \frac{\size\left( \cu_1 \right)}{\size\left( \cu_2 \right)} \leq 3,
\end{equation*}
and which are both good. Then there exists a cluster~$\C$ such that 
\begin{equation*} \label{}
\C_*( \cu_1 ) \cup \C_*( \cu_2 )\subseteq  \C \subseteq \cu_1 \cup \cu_2.
\end{equation*}
\end{lemma}

The main interest in these definitions is that in the supercritical phase $\p > \p_c(d)$, the probability of a cube to be good is exponentially close to $1$ in the size of the cube. Such a result is stated in the following proposition and is a direct consequence~\cite[Theorem 3.2]{P} and~\cite[Theorem 5]{PP}.

\begin{proposition}
\label{l.AP}
Consider a Bernoulli bond percolation of probability $\p \in (\p_c(d) , 1 ]$. Then there exists a positive constant $C(d,\p)<\infty$ such that, for every cube $\cu \subseteq \Zd$ of size $N$, 
\begin{equation} \label{e.pgoodness}
\P \left[ \cu \ \mbox{is good} \right] \geq 1 - C \exp\left( - C^{-1}N \right).
\end{equation}
\end{proposition}

The renormalization structure we want to implement relies on the observation that $\Zd$ can be partitioned into good cubes of varying sizes. Thanks to the exponential stochastic integrability obtained by Penrose and Pisztora and stated in Proposition~\ref{l.AP}, we are able to build such a partition. The precise statement is given in the following proposition.

\begin{proposition}[Propositions 2.1 and 2.4 of~\cite{AD2}]
\label{p.partitions}
Under the assumption $\p > \p_c$, $\P$--almost surely, there exists a partition $\Pa$ of $\Zd$ into triadic cubes with the following properties:
\begin{enumerate}[label=(\roman*)]

\item All the predecessors of elements of $\Pa$ are good cubes, i.e., for every pair of triadic cubes $\cu, \cu' \in \T$, one has the property
\begin{equation*}
   \cu' \in \Pa ~\mbox{and}~ \cu' \subseteq \cu \implies \cu \mbox{ is good.}
\end{equation*}

\item Neighboring elements of $\Pa$ have comparable sizes: for every $\cu,\cu'\in \Pa$ such that $\dist(\cu,\cu') \leq 1$, we have
\begin{equation*}
\frac13 \leq \frac{\size(\cu')}{\size(\cu)} \leq 3.
\end{equation*}

\item Estimate for the coarseness of $\Pa$: if we denote by $\cu_\Pa(x)$ the unique element of $\Pa$ containing a given point $x\in\Zd$, then there exists a constant $C(\p,d)<\infty$ such that,
\begin{equation} \label{sizecubeisO1}
\size\left( \cu_\Pa(x) \right) \leq \O_1 (C).
\end{equation}
\item Minimal scale for $\Pa$. For each $q \in [1, \infty )$, there exists a constant $C := C(d,\p, q) < \infty$, a non-negative random variable $\M_{q} (\Pa)$ and an exponent $r := r(d,\p,q) > 0$ such that
\begin{equation} \label{stoint.Mpart}
\M_{q}(\Pa) \leq \O_r(C),
\end{equation}
and for each radius $R$ satisfying $R \geq \M_{q}(\Pa)$,
\begin{equation}\label{eq:Mt}
R^{-d} \sum_{x \in \Zd \cap B_R} \size \left( \cu_{\Pa} (x) \right)^{q} \leq C \quad \mbox{ and } \quad \sup_{x \in \Zd \cap B_R} \size \left( \cu_{\Pa} (x) \right) \leq R^{\frac{1}{q}}.
\end{equation}
\end{enumerate}
\end{proposition}

\begin{remark}
To be precise, this proposition is a consequence of Propositions 2.1 and 2.4 of~\cite{AD2} and of Proposition~\ref{l.AP} as is explained in~\cite[Section 2.2]{AD2}.

\smallskip

Additionally, the precise result stated in~\cite[Proposition 2.4]{AD2} is that there exists a minimal scale $\M_{q}(\Pa)$ above which $\sup_{x \in \Zd \cap B_R} \size \left( \cu_{\Pa} (x) \right) \leq R^{\frac{d}{d+q}}$, with an exponent $d/(d+q)$ instead of $1/q$. Nevertheless, it is straightforward to recover the statement of Proposition~\ref{p.partitions} from the one of~\cite[Proposition 2.4]{AD2}.
\end{remark}

\begin{figure}[h!]
    \centering
    \includegraphics[scale= 0.8]{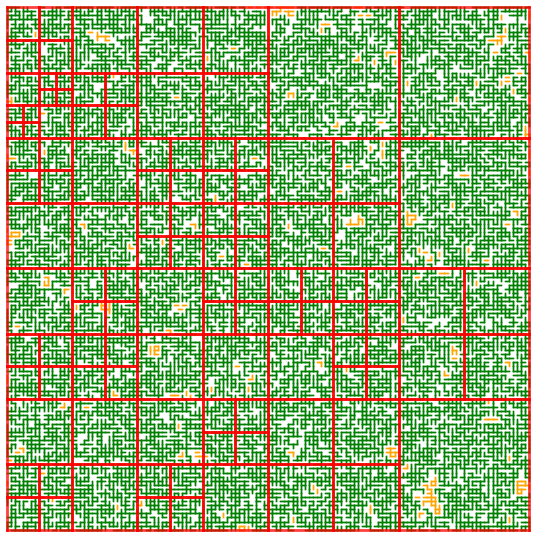}
    \caption{A realization of the partition $\Pa$, where the cluster in green is the maximal cluster and the cubes in red are elements of $\Pa$.}
    \label{fig:my_labelPa}
\end{figure}

Figure~\ref{fig:my_labelPa} (drawn with dyadic cubes instead of triadic cubes to improve readability) illustrates what this partition looks like. It allows to extend functions defined on the infinite cluster to the whole space $\Rd$, as is explained below. We consider a function $u : \C_\infty \to \R$. For each point $x \in \Zd$, we choose a point $z \left(x \right)$ in the cluster $\C_*( \cu_{\Pa}(x))$ according to some deterministic procedure (for instance we choose the one which is the closest to the center of the cube and break ties by using the lexicographical order). We then define the coarsened function $\left[ u \right]_{\Pa}$ on $\Zd$ according to the formula, for each $x \in \Zd$,
    \begin{equation} \label{def.coarsenu}
        \left[ u \right]_{\Pa}(x) := \left\{ \begin{aligned} 
        & u(x) & ~\mbox{if} ~ x \in \C_\infty, \\
        & u \left(z \left(x \right) \right) & ~\mbox{otherwise},
        \end{aligned} \right.
    \end{equation}
    and extend it to the whole space $\Rd$ by setting it to be piecewise constant on the cubes $\left[ x - \frac 12 , x + \frac 12 \right)^d$, for $x \in \Zd$. When the function $u$ is defined on the parabolic space $[0, \infty) \times \C_\infty$, we define its extension to the space $[0, \infty) \times \Zd$, which we also denote by $\left[ u\right]_{\Pa}$, according to the formula
    \begin{equation} \label{def.coarsenupar}
        \left[ u \right]_{\Pa}(t,x) := \left\{ \begin{aligned} 
        & u(t,x) & ~\mbox{if} ~ x \in \C_\infty, \\
        & u \left(t, z \left(x \right) \right) & ~\mbox{otherwise}.
        \end{aligned} \right.
    \end{equation} 
    For later purposes, we note that, given a function $u: \C_\infty \to \R$, the $L^p$-norm of the function $\nabla \left[ u\right]_\Pa$ can be estimated in terms of $L^p$-norm of the function $\nabla u$ and the sizes of the cubes in the partition $\Pa$. Specifically, one has the formula, for any radius $r \geq \size \left( \cu_\Pa (0) \right)$,
\begin{equation} \label{est.nablacoarsenL2}
    \left\| \nabla \left[ u \right]_\Pa \right\|_{L^p \left(\Zd \cap B_r\right)}^p \leq C \int_{B_r \cap \C_\infty} \size \left( \cu_\Pa (x) \right)^{pd-1} \left| \nabla u(x) \right|^p \, dx.
\end{equation}
The proof of this result can be found in~\cite[Lemma 3.3]{AD2}. Additionally, one can estimate the $L^p$-norm of the function $\left[ u \right]_\Pa$ in terms of the $L^p$-norm of the function $u$ according to the formula, for any radius $r \geq \size \left( \cu_\Pa (0) \right)$,
\begin{equation} \label{eq:estcoarseminusu}
    \left\| \left[ u \right]_\Pa \right\|_{L^p (B_r)}^p \leq \int_{\hat B_r \cap \C_\infty} \size \left( \cu_\Pa (x) \right)^{d} \left| u(x) \right|^p \, dx,
\end{equation}
where $\hat B_r$ denotes the union of all the cubes in the partition $\Pa$ which intersect the ball $B_r$, i.e., $\hat B_r := \cup \{ \cu \in \Pa ~:~ \cu \cap B_r \neq \emptyset \}$.
This estimate is a consequence of the following argument: by definition of the coarsened function $\left[ u \right]_\Pa$, one has the estimates, for any cube $\cu$ of the partition $\Pa$,
\begin{equation*}
    \left\| \left[ u \right]_\Pa \right\|_{L^p \left( \cu \right)}^p \leq \size \left( \cu \right)^d \left\| \left[ u \right]_\Pa \right\|_{L^\infty \left( \cu \right)}^p \leq  \size \left( \cu \right)^d \left\| u  \right\|_{L^\infty \left( \C_\infty \cap \cu \right)}^p \leq  \size \left( \cu \right)^d \left\|  u  \right\|_{L^p \left( \C_\infty \cap \cu \right)}^p,
\end{equation*}
where we used the discrete $L^\infty-L^p$-estimate in the third inequality. Summing over all the cubes of the set $\hat B_r$ completes the proof of the estimate~\eqref{eq:estcoarseminusu}.

\subsection{Functional inequalities on the infinite cluster} \label{sectFIIC}

In this section, we state mostly without proofs, some functional inequalities which are valid on the infinite cluster $\C_\infty$. The partition of good cubes presented in Section~\ref{sectionPGC} allows to prove these estimates and we refer to~\cite{AD2} for the details of the argument. Some of these inequalities were already proved by other renormalization technique: it is in particular the case of the Poincar\'e inequality which was established by Barlow in~\cite{Ba} (see also Mathieu, Remy~\cite{MR} and Benjamini, Mossel~\cite{BM03}).

The fact that these bounds are stated on a random graph which has an irregular nature means that they are only valid on balls of size larger than some random minimal scales, denoted by $\M_{\mathrm{Poinc}}$ and $\M_{\mathrm{Meyers}}$ in the following statements, which depend on the environment $\a$ and are large when the environment is ill-behaved.

The first functional inequality we record is the Poincar\'e inequality, it can be found in~\cite[Theorem~2.18]{Ba} for the $L^2$-version.

\begin{proposition}[Poincar\'e inequality on $\C_\infty$] \label{p.poinc}
Fix a real number $p \in \left[ \frac{d}{d-1}, \infty \right)$. There exist a constant $C := C(d , \p, p) < \infty$, an exponent $s := s(d , \p, p) > 0$ such that, for any $y \in \Zd$, there exists a non-negative random variable $\M_{L^p-\mathrm{Poinc}}(y)$ which satisfies the stochastic integrability estimate
\begin{equation} \label{stoint.MPoinc}
    \M_{L^p-\mathrm{Poinc}}(y) \leq \O_s \left( C \right),
\end{equation}
such that for each radius $R \geq \M_{L^p-\mathrm{Poinc}}(y)$ and each function $u : \C_\infty \cap B_R(y) \to \R$,
\begin{equation*}
    \left\| u - \left( u \right)_{\C_\infty \cap B_R(y)}\right\|_{\underline{L}^p \left( \C_\infty \cap B_R(y) \right)} \leq C R \left\| \nabla u \right\|_{\underline{L}^p \left( \C_\infty \cap B_R(y) \right)}.
\end{equation*}
Moreover for each function $u : \C_\infty \cap B_R(y) \to \R$, such that $u = 0$ on the boundary ${ \partial \left(\Zd \cap B_R(y)\right) \cap \C_\infty}$,
\begin{equation*}
    \left\| u \right\|_{\underline{L}^p \left( \C_\infty \cap B_R(y) \right)} \leq C R \left\| \nabla u \right\|_{\underline{L}^p \left( \C_\infty \cap B_R(y) \right)}.
\end{equation*}
\end{proposition}

\begin{remark}
This inequality is frequently used in the case $p=2$. To shorten the notation, we write $\M_{\mathrm{Poinc}}(y)$ to refer to the minimal scale $\M_{L^2-\mathrm{Poinc}}(y)$.
\end{remark}

\begin{proof}
By translation invariance of the model, we can always assume $y = 0$. The proof relies on the Sobolev inequality as stated in~\cite[Proposition 3.4]{AD2}, together with the H\"older inequality by setting $\M_{L^p-\mathrm{Poinc}}(0) := \M_{q}\left( \Pa \right)$, for a parameter $q$ chosen large enough depending only on the dimension $d$ and the exponent $p$.
\end{proof}

The second estimate we need to record is the parabolic Caccioppoli inequality. This estimate is valid on any subgraph of $\Zd$ and is used in Section~\ref{largescaelC01}.

\begin{proposition}[Parabolic Caccioppoli inequality on $\C_\infty$] \label{p.caccioppoli}
There exists a finite positive constant $C := C(d , \lambda)$ such that, for each point $y \in \Zd$, each radius $R \geq 1$ and each function ${u : I_R \times \left(\C_\infty \cap B_R(y) \right) \to \R}$ which is $\a$-caloric, i.e., which is a solution of the parabolic equation
\begin{equation*}
    \partial_t u - \nabla \cdot \a \nabla u = 0 ~\mbox{in}~  I_R \times \left( \C_\infty \cap B_R(y) \right) ,
\end{equation*}
one has
\begin{equation*}
    \left\| \nabla u \right\|_{L^2 \left( \CQy{R/2} \right)} \leq \frac{C}{R} \left\| u - \left( u \right)_{\CQy{R}}\right\|_{L^2 \left( \CQy{R} \right)}.
\end{equation*}
\end{proposition}

\begin{proof}
The proof follows the standard arguments of the Caccioppoli inequality; the fact that the function $u$ is defined on the infinite cluster does not affect the proof and we omit the details.
\end{proof}

The third estimate we record is an $L^\infty L^2$ gradient bound for $\a$-caloric functions. The proof of this result can be found in~\cite[Lemma 8.2]{armstrong2017quantitative} in the uniformly elliptic setting, the extension to the percolation cluster makes no difference in the proof.

\begin{lemma}\label{timeslicelemma}
There exists a positive constant $C := C(d , \lambda) < \infty$ such that for any radius $R \geq 1$, any point $y \in \Zd$ and any function $u: I_{R} \times \left(\C_\infty \cap B_R(y)\right) \to \R$ which satisfies
\begin{equation*}
    \partial_t u - \nabla \cdot \a \nabla u = 0 ~\mbox{in}~ I_{R} \times  \left( \C_\infty \cap B_{R}(y)\right),
\end{equation*}
one has the estimate
\begin{equation*}
    \sup_{t \in I_{ R/2}} \left\| \nabla u(t, \cdot) \right\|_{\underline{L}^2\left( \C_\infty \cap B_{R/2}(y) \right)} \leq C \left\| \nabla u \right\|_{\underline{L}^2\left( I_{R} \times \left( \C_\infty \cap B_{R}(y) \right)\right)}.
\end{equation*}
\end{lemma}

The last estimate we record in this section is the Meyers estimate for $\a$-caloric functions on the percolation cluster. This inequality is a non-concentration estimate and essentially states that the energy of solutions of a parabolic equation cannot concentrate in small volumes. It is used in Section~\ref{largescaelC01}.

\begin{proposition}[Interior Meyers estimate on $\C_\infty$] \label{p.Meyers}
There exist a finite positive constant ${C := C(d , \p, \lambda)}$, two exponents $s := s(d , \p, \lambda) > 0, ~ \delta_0 := \delta_0 (d , \p, \lambda) > 0$ such that, for each $y \in \Zd$, there exists a non-negative random variable $\M_{\mathrm{Meyers}}(y)$ which satisfies the stochastic integrability estimate
\begin{equation*}
    \M_{\mathrm{Meyers}}(y) \leq \O_s \left( C \right),
\end{equation*}
such that, for each radius $R \geq \M_{\mathrm{Meyers}}(y)$ and each function $u : I_R \times \left(\C_\infty \cap B_R(y)\right) \to \R$ solution of the equation
\begin{equation*}
    \partial_t u - \nabla \cdot \a \nabla u = 0 ~\mbox{in}~ I_R \times \left(\C_\infty \cap B_R(y)\right),
\end{equation*}
one has
\begin{equation*}
    \left\| \nabla u\right\|_{\underline{L}^{2 + \delta_0} \left( I_{R/2} \times \left( \C_\infty \cap B_{R/2}(y) \right) \right)} \leq C \left\| \nabla u \right\|_{\underline{L}^2 \left( I_R \times \left( \C_\infty \cap B_R(y) \right)\right)}.
\end{equation*}
\end{proposition}

\begin{proof}
The classical proof of the interior Meyers estimate (cf.~\cite{GG}) is based on an application of the Caccioppoli inequality, the Sobolev inequality and the Gehring's lemma (cf.~\cite{Giu}). The proof of this result on the percolation cluster for the elliptic problem is written in~\cite[Proposition 3.8]{AD2}. For the parabolic problem considered here, the proof in the case of uniformly elliptic environments can be found in~\cite[Appendix B]{ABM17}. The argument can be adapted to the percolation cluster following the strategy developed in~\cite[Section 3]{AD2}. Since the analysis does not contain any new idea regarding the method and the result can be obtained by essentially rewriting the proof, we skip the details.

\end{proof}

\subsection{Homogenization on percolation clusters} \label{homogpercclu}
In this section, we collect some results of stochastic homogenization in supercritical percolation useful in the proof of Theorem~\ref{mainthm}. The proof of this theorem is based on a quantitative two-scale expansion, which relies on two important functions: the first-order corrector and its flux. They are introduced in Sections~\ref{section2.2.1} and~\ref{section2.2.2} respectively.

\subsubsection{The first-order corrector} \label{section2.2.1}

We let $\A_1(\C_\infty)$ be the random vector space of $\a$-harmonic functions on the infinite cluster $\C_\infty$ with at most linear growth. This latter condition is expressed in terms of average $L^2$-norm and we define
\begin{equation*}
    \A_1 \left( \C_\infty \right) := \left\{ u : \C_\infty \to \R ~:~ - \nabla \cdot \left( \a \nabla u\right) = 0 ~\mbox{in}~ \C_\infty ~\mbox{and}~\lim_{r \rightarrow \infty} r^{-2} \left\| u \right\|_{\underline{L}^2 \left(\C_\infty \cap B_r \right)} = 0   \right\}.
\end{equation*}
It is known that this space is almost surely finite-dimensional and that its dimension is equal to~$(d+1)$ (see~\cite{BDKY}). Additionally, every function $u \in \A_1 \left( \C_\infty \right)$ can be uniquely written as
\begin{equation*}
    u(x) = c + p \cdot x + \chi_p(x),
\end{equation*}
where $c \in \R$, $p \in \Rd$ and $\chi_p$ is a function called the corrector; it is defined up to a constant and satisfies the quantitative sublinearity property stated in the following proposition.

\begin{proposition} \label{prop.sublin.corr}
For any exponent $\alpha > 0$, there exist an exponent $s(d , \p , \lambda, \alpha) > 0$ and a positive constant $C(d , \p , \lambda, \alpha) < \infty$ such that, for any point $y \in \Zd$, there exists a non-negative random variable $\M_{\mathrm{corr},\alpha}(y)$ satisfying the stochastic integrability estimate
\begin{equation} \label{sto.intMCorr}
    \M_{\mathrm{corr},  \alpha}(y) \leq \O_s \left( C \right),
\end{equation}
such that for every radius $r \geq \M_{\mathrm{corr},\alpha}(y)$, and every $p \in \Rd$,
\begin{equation*}
    \osc_{x \in \C_\infty \cap B_r(y)} \chi_p := \left( \sup_{x \in \C_\infty \cap B_r(y)} \chi_p - \inf_{x \in \C_\infty \cap B_r(y)} \chi_p \right) \leq C |p| r^{ \alpha}.
\end{equation*}
\end{proposition}

\begin{proof}
The proof of this result relies on the optimal scaling estimates for the corrector established in~\cite{dario2018optimal}. Indeed by~\cite[Theorem 1]{dario2018optimal}, one has the following result: there exists a constant ${C := C(d , \lambda,\p) < \infty}$ and an exponent $s := s(d , \lambda, \p) < \infty$ such that for each $x , y \in \Zd$, and each $p \in \Rd$,
\begin{equation} \label{est:optcorr1602}
    \left| \chi_p (x) - \chi_p(y) \right| \indc_{\{ x , y \in \C_\infty\}} \leq  \left\{ \begin{aligned}
    &\O_s(C\vert p \vert) &~\mbox{if}~ d \geq 3, \\
    &\O_s(C \vert p \vert \log^{\frac 12} |x-y| ) &~\mbox{if}~ d = 2.
    \end{aligned}
    \right.
\end{equation}
Proposition~\ref{prop.sublin.corr} is then a consequence of the previous estimate and an application of Lemma~\ref{lemma1.5} with the sequence of random variables
\begin{equation*}
    X_n := 3^{-\alpha n} \sup_{x \in \Zd \cap B_{3^n}(y)}  \left| \chi_p (x) - \chi_p(y) \right| \indc_{\{ x , y \in \C_\infty\}}.
\end{equation*}
To be more precise, we use the estimate~\eqref{eq:OSumm} to control the maximum of the random variables 
\begin{align} \label{eq:TL2709}
X_n & = 3^{-\alpha n} \sup_{x \in \Zd \cap B_{3^n}(y)}  \left| \chi_p (x) - \chi_p(y) \right| \indc_{\{ x , y \in \C_\infty\}} \\ & \leq 3^{-\alpha n} \left(\sum_{x \in \Zd \cap B_{3^n}(y)} \left(\left| \chi_p (x) - \chi_p(y) \right| \indc_{\{ x , y \in \C_\infty\}}\right)^{\frac{2d}{\alpha}}\right)^{\frac{\alpha}{2d}} \notag \\ & \leq  \left\{ \begin{aligned}
&\O_s\left(C |p|3^{-\frac{\alpha n}{2}}\right) &~\mbox{if}~d\geq 3, \\
&\O_s\left(C |p|\sqrt{n}3^{-\frac{\alpha n}{2}}\right) &~\mbox{if}~d=2.
\end{aligned} \right. \notag
\end{align}
Then the sequence $\{X_n\}_{n \geq 1}$ satisfies the assumption of Lemma~\ref{lemma1.5}.
\end{proof}

The fact that the corrector is only defined up to a constant causes some technical difficulties in the proofs, in particular the two-scale expansion stated in~\eqref{sec42scexp.1605} and used in the proof of Theorem~\ref{mainthm} is ill-defined in this setting. To solve this issue, we choose the following (arbitrary) normalization for the corrector: given a point $y \in \Zd$ and an environment $\a$ in the set of probability $1$ on which the corrector is well-defined, we let $x \in \C_\infty$ which is the closest to the point $y$ (and break ties by using the lexicographical order) and normalize the corrector by setting $\chi_p (x) = 0$. The choice of the point $y$ will always be explicitly indicated to avoid confusions. We note that with this normalization, the corrector is not stationary.

\subsubsection{The centered flux} \label{section2.2.2}
A second important notion in the implementation of the two-scale expansion is the centered flux; it is defined in the following paragraph.

For a fixed vector $p = (p_1 , \ldots , p_d) \in \Rd$, we consider the mapping $\a (p + \D \chi_{p} ) : \C_\infty \to \Rd$ defined by the formula, for each $x \in \Zd$,
\begin{equation*}
    \a (p + \D \chi_{p})(x) := \left(\a(\{x, x+e_1\})\left(p_1 +\Dr{1}\chi_{p}(x) \right), \ldots, \a(\{x, x+e_d\})\left(p_d + \Dr{d}\chi_{p}(x)\right) \right).
\end{equation*}
This function oscillates quickly but it is close to the deterministic slope $\frac{1}{2}\sigk p$ in the $\aH^{-1}$-norm on the infinite cluster, where $\bar{\sigma}$ is the diffusivity of the random walk introduced in~\eqref{intro.inv.princ}. This motivates the following definition: for a fixed vector $p \in \Rd$, we define the centered flux $\tilde{\mathbf{g}}_{p} : \C_\infty \rightarrow \Rd$ according to the formula
\begin{equation*}
    \tilde{\mathbf{g}}_{p} := \a (\D \chi_{p} + p) - \dsigk p.
\end{equation*}
The following proposition estimates the $\aH^{-1}$-norm of the centered flux. It it proved in Appendix~\ref{appendixb}, Proposition~\ref{prop:WeakNormFlux}.

\begin{proposition} \label{prop.sublin.flux}
For any exponent $\alpha > 0$, there exist a positive constant $C := C(d , \p , \lambda,\alpha) < \infty$ and two exponents $s := s(d , \p , \lambda,\alpha) > 0$ and $\alpha := \alpha(d , \p , \lambda)>0 $ such that, for any $y \in \Zd$, there exists a non-negative random variable $\M_{\mathrm{flux}, \alpha}(y)$ satisfying the stochastic integrability estimate
\begin{equation} \label{sto.intMflux}
    \M_{\mathrm{flux}, \alpha}(y) \leq \O_s \left( C \right),
\end{equation}
such that for each radius $r \geq \M_{\mathrm{flux}, \alpha}(y)$,
\begin{equation} \label{quantfluxsmall}
     \left\| \tilde{\mathbf{g}}_{p} \right\|_{\underline{H}^{-1} \left( \C_\infty \cap B_r(y)\right)} \leq C |p| r^{ \alpha}.
\end{equation}
\end{proposition}

\begin{remark}
We emphasize that, in this article, the previous proposition is not a property of the diffusivity $\sigk$ but its definition: building on former result from~\cite{AD2, dario2018optimal, gu2019efficient}, we prove that there exists a coefficient such that the estimate~\eqref{quantfluxsmall} is satisfied and name this coefficient $\sigk$. Thanks to the estimate~\eqref{quantfluxsmall}, we are then able to prove Theorem~\ref{mainthm} and the invariance principle~\eqref{intro.inv.princ} with the same coefficient $\sigk$. We refer to \eqref{def.sigkB1} and Remark \ref{rmk:sigk} for a more detailed discussion. 
\end{remark}

\subsection{Random walks on graphs}

In this section, we record the Carne-Varopoulos bound pertaining to the transition kernel of the continuous-time random walk which holds on any infinite connected subgraph of $\Zd$. This estimate is not as strong as the ones we are trying to establish (for instance the ones of Theorem~\ref{mainthm}, or of Theorem~\ref{barlow} proved in~\cite{Ba}) but can be applied in greater generality: it applies to any realization of the infinite cluster, i.e., to any environment $\a$ in the set of probability $1$ where there exists a unique infinite cluster, without any consideration about its geometry. From a mathematical perspective, this means that there is no minimal scale in the statement of Proposition~\ref{proposition2.11}.

\begin{proposition}[Carne-Varopoulos bound, Corollaries 11 and 12 of~\cite{Da93}] \label{proposition2.11}
Let $\G$ be an infinite, connected subgraph of $\Zd$ and $\a$ be a function from the bonds of $\G$ into $[\lambda , 1]$. For $y \in \G$, we let $p\left(\cdot, \cdot, y \right)$ be the heat kernel associated to the parabolic equation
\begin{equation*}
\left\{ \begin{array}{ll}
 \partial_t p \left(\cdot ,  \cdot , y \right) - \nabla \cdot \left( \a \nabla p \left( \cdot  , \cdot , y \right) \right) = 0 &~\mbox{in}  ~(0, \infty) \times \G, \\
 p \left( 0 , \cdot , y  \right) = \delta_y &~\mbox{in}~\G.
\end{array} \right.
\end{equation*}
Then there exists a positive constant $C := C(d, \lambda) < \infty$ such that for each point $x \in \G$,
\begin{equation} \label{est.heatkernelgraph}
p \left( t , x , y \right) \leq \left\{ \begin{aligned} 
& C  \exp \left( - \frac{|x - y|^2}{C t} \right) & ~\mbox{if}~ |x - y| \leq t, \\
& C \exp \left( - \frac{|x - y|}{C} \left( 1 + \ln \frac{|x-y|}{t} \right) \right)  & ~\mbox{if}~  |x - y| \geq t.
\end{aligned} \right.
\end{equation}
\end{proposition}

\begin{remark}\label{remark2.12}
For later use, we note that, when $|x - y| \geq t$,
\begin{equation*}
    \exp \left( - \frac{|x - y|}{C} \left( 1 + \ln \frac{|x-y|}{t} \right) \right) \leq \exp \left( - \frac{t}{C} \right).
\end{equation*}
A consequence of this inequality is that, by increasing the value of the constant $C$, one can add a factor $t^{-d/2}$ in the second line of the right side of~\eqref{est.heatkernelgraph}: for every constant $0 < C < \infty$ there exists a finite constant $C' > C$ such that, when $|x - y| \geq t$,
\begin{equation*}
    C \exp \left( - \frac{|x - y|}{C} \left( 1 + \ln \frac{|x-y|}{t} \right) \right) \leq C' t^{-d/2} \exp \left( - \frac{|x - y|}{C'} \left( 1 + \ln \frac{|x-y|}{t} \right) \right).
\end{equation*}
\end{remark}

\section{Decay and Lipschitz regularity of the heat kernel} \label{section3}

In this section, we collect and establish some estimates about the decay of the parabolic Green's function. In Section~\ref{section 3.1}, we record a result of Barlow in ~\cite{Ba}, which establishes Gaussian upper bounds on the parabolic Green's function on the infinite cluster. This result is a percolation version of the Nash-Aronson estimate~\cite{Ar67}, originally proved for uniformly elliptic divergence form diffusions. Building upon the result of Barlow, we then establish estimates on the gradient of the parabolic Green's function on the percolation cluster, stated in Theorem~\ref{gradbarlowintro}, thanks to a large-scale $C^{0,1}$-regularity estimate. The argument makes use of techniques from stochastic homogenization and follows a classical route which can be decomposed into three steps: we first establish a quantitative homogenization theorem for the parabolic Dirichlet problem (see Section~\ref{homogparaDir}), once this is achieved we prove a large-scale $C^{0,1}$-regularity estimate for $\a$-caloric functions (see Section~\ref{largescaelC01}). In Section~\ref{Greendecaygrad}, we use this regularity estimate together with the heat kernel bound of Barlow to obtain the decay of the gradient of the heat kernel stated in Theorem~\ref{gradbarlowintro}.

\subsection{Decay of the heat kernel} \label{section 3.1}

In this section, we record the result of Barlow~\cite{Ba}, who established Gaussian bounds on the transition kernel. We first introduce the following function.

\begin{definition}
Given a point $x \in \Rd$, a time $t \in (0, \infty)$ and a constant $0 < C < \infty$, we define the function $\Phi_C$ according to the formula
\begin{equation} \label{def.phiC}
    \Phi_C (t,x) := \left\{ \begin{aligned}
&   C t^{-d/2} \exp \left( - \frac{|x |^2}{C t} \right) &~ \mbox{if} ~ |x| \leq t, \\
&C t^{-d/2} \exp \left( - \frac{|x|}{C} \left( 1 + \ln \frac{|x|}{t} \right) \right) ~ &\mbox{if} ~ |x| \geq t.
    \end{aligned} \right.
\end{equation}
We note that this function is radial and increasing in the variable $C$. This function corresponds to a discrete heat kernel. For further use, we note that it satisfies the following semigroup property, for each $t_1 , t_2 \in (0, \infty)$ and each $x , y \in \Zd$
\begin{equation} \label{eq:convprop}
\int_{\Zd} \Phi_C \left( t_1 , x - z \right) \Phi_C \left( t_2 , z - y \right) \, dz \leq \Phi_{C'} \left( t_1 + t_2 , x - y \right),
\end{equation}
for some larger constant $C' > C$. This property is proved by an explicit computation or by using the semigroup property of the law of the random walk on $\Zd$ (see Remark~\ref{rmk:semigroupPhiC}). We define the function $\Psi_C$ according to the formula
\begin{equation} \label{def.psiC}
\Psi_C (t,r) := \left\{ \begin{aligned}
&   (0, \infty) \times [0 , \infty)  &\rightarrow & ~\R, \\
&(t , r) & \mapsto&~- \ln \left( \Phi_C(t,x) \right), ~\mbox{where}~ x \in \Rd~\mbox{satisfies}~|x|=r. 
    \end{aligned} \right.
\end{equation}
In particular, one has the identity,
\begin{equation*}
     \forall x \in \Rd, ~ \Phi_C(t,x) = \exp \left( - \Psi_C (t,|x|) \right).
\end{equation*}
\end{definition}
The function $\Psi_C$ satisfies the following properties:
 \begin{enumerate}[label=(\roman*)]
     \item It is decreasing in the variable $C$, increasing in the variable $r$, and continuous with respect to both variables;
     \item It is convex with respect to the variable $r$.
 \end{enumerate}
We now record the result of Barlow.

\begin{theorem}[Gaussian upper bound, Theorem 1 and Lemma 1.1 of~\cite{Ba}] \label{barlow}
There exist an exponent $s := s \left( d , \p, \lambda \right) > 0$, a positive constant $C := C(d , \p, \lambda ) < \infty$ such that for each $y \in \Rd$, there exists a random time $\T_{\mathrm{NA}}(y)$ satisfying the stochastic integrability estimate 
\begin{equation} \label{inteminnscNA}
\T_{\mathrm{NA}}(y) \leq \O_s \left( C \right),
\end{equation}
such that, on the event $\{y \in \C_\infty\}$, for every time $t  \in (0 , \infty)$ satisfying $t \geq \T_{\mathrm{NA}}(y)$, and every point $x \in \C_\infty$,
\begin{equation} \label{unifelltrankerba}
p \left( t , x , y \right) \leq \Phi_C(t,x - y).
\end{equation}
\end{theorem}

\begin{remark}
The stochastic integrability estimate~\eqref{inteminnscNA} is not stated in Theorem 1.1 of~\cite{Ba} but is mentioned in its remark equation (0.5) following the theorem.
\end{remark}

\begin{remark} \label{remlargetime}
The estimate in the regime $t \leq |x-y|$ does not require the assumption that $t$ is larger than the minimal scale $\T_{\mathrm{NA}}(y)$ and is in fact a deterministic result: it is a consequence of Proposition~\ref{proposition2.11} (proved in~\cite{Da93}) and Remark~\ref{remark2.12}.
\end{remark}

\begin{remark}\label{rmk:semigroupPhiC}
The function $\Phi_C$ can be used to obtain upper and lower bounds on the law of the random walk on the lattice $\Zd$: there exist constants $C_1 , C_2$ depending only on the dimension $d$ such that 
\begin{equation} \label{eq:TV07410904}
    \Phi_{C_1} (t , x - y) \leq p^{\Zd}(t , x , y) \leq \Phi_{C_2} (t , x - y),
\end{equation}
where we used the notation $p^{\Zd}(t , x , y) := \P_y \left[ X_t = x\right]$, and where $(X_t)_{t \geq 0}$ denotes the VSRW on $\Zd$ starting from the point $y$. We refer to the work~\cite{delmotte1999parabolic} of Delmotte and the work~\cite{Da93} for this result. The estimates~\eqref{eq:TV07410904} can then be used to prove the property~\eqref{eq:convprop}. Indeed, since the random walk $(X_t)_{t \geq 0}$ is a Markov process, its transition function $p^{\Zd}$ has the semigroup property, and we can write
\begin{align*}
\int_{\Zd} \Phi_{C_1} \left( t_1 , x - z \right) \Phi_{C_1} \left( t_2 , z - y \right) \, dz & \leq \int_{\Zd} p^{\Zd} \left( t_1 , x , z \right) p^{\Zd} \left( t_2 , z , y \right) \, dz \\ & = p^{\Zd} \left( t_1 + t_2, x ,y \right) \\ & \leq 
\Phi_{C_2} \left( t_1 + t_2 , x - y \right).
\end{align*}
This argument gives the estimate~\eqref{eq:convprop} in the case $C = C_1$, but can be easily extended to any constant $C > 0$.
\end{remark}

We complete this section by mentioning that the result of Barlow is proved for the heat kernel associated to the constant speed random walk and on the percolation cluster only, i.e., when the conductances are only allowed to take the values $0$ or $1$. The adaptation to the variable speed random walk with uniformly elliptic conductances only requires a typographical change of the proof: all the computations performed in~\cite{Ba} to obtain the upper bound~\eqref{unifelltrankerba} can be adapted to our setting and so is the case of the existing results in the literature which are used in the proof.

\subsection{Decay of the gradient of the Green's function} \label{section 3.2}
The main objective of this section is to prove Theorem~\ref{gradbarlowintro}.
The proof of this result makes use of techniques from stochastic homogenization and can be split into three distinct steps, which correspond to the three following subsections. The first idea is to prove that the parabolic Green's function is close, on large scales, to a caloric function. This is carried out in Section~\ref{homogparaDir} and the proof is based on a two-scale expansion. The analysis relies on the sublinearity of the corrector and the estimate on the $\underline{H}^{-1}$-norm of the centered flux stated in Section~\ref{homogpercclu}. This result is only necessary to establish a large-scale regularity theory for which sharp homogenization errors are not needed; we thus do not try to prove an optimal error estimate in the homogenization of the parabolic Dirichlet problem and only prove the result with an algebraic and suboptimal rate of convergence. Then, in Section~\ref{largescaelC01}, we use the homogenization estimate proved in Theorem~\ref{homogparaDir} to establish a large-scale regularity theory in the spirit of~\cite[Chapter 3]{armstrong2017quantitative} or~\cite[Section 7]{AD2}. Finally, in Section~\ref{Greendecaygrad}, we combine Proposition~\ref{heatkernelC01} and the heat kernel bound proved by Barlow and stated in Theorem~\ref{barlow} to deduce Theorem~\ref{gradbarlowintro}.

\subsubsection{Homogenization of the parabolic Dirichlet problem} \label{homogparaDir}
In this section, we prove a quantitative homogenization theorem for the parabolic Cauchy-Dirichlet problem on the infinite cluster. In the following statement, we let $\eta$ be a smooth, non-negative function supported in the ball $B_{\frac{1}{2}}(0)$, and satisfying the identity $\int \eta  = 1$. It is used as a smoothing operator in the convolution~\eqref{convolutionfftilde}. We also define the set $\cv(\Zd \cap B_r(y))$ to be the convex hull of the set $\Zd \cap B_r(y)$, i.e.,
\begin{equation}\label{eq:defCv}
\cv(\Zd \cap B_r(y)) := \left\{z \in \Rd : z = \sum_i \alpha_i x_i, \quad x_i \in \Zd \cap B_r(y), \quad 0 \leq \alpha_i \leq 1 \text{ and } \sum_i \alpha_i = 1 \right\}. \end{equation}
It is used to define the domain of the homogenized equation so that the boundary condition coincides.

\begin{theorem}\label{thm:HomoParaPerco}
Fix an exponent $\delta > 0$, then there exist a positive constant $C(d, \lambda, \p, \delta) < \infty$, two exponents $s(d , \lambda, \p, \delta) > 0,$ $\alpha(d, \lambda, \p, \delta) >0$ such that for any point $y\in \Zd$, there exists a non-negative random variable $\M_{\mathrm{hom},\delta}(y)$ satisfying
$$
\M_{\mathrm{hom}, \delta}(y) \leq \O_{s}(C)
$$
such that, for every $r > \M_{\mathrm{hom}}(y)$, and every boundary condition $f \in \Wpar(\CQy{r})$, the following statement is valid. Let $u$ be the weak solution of the parabolic equation 
\begin{equation}\label{eq:ParaPerco}
\left\{
	\begin{array}{ll}
	(\partial_t -\nabla \cdot \a \nabla) u = 0  & \text{ in } \CQy{r}, \\
	u = f & \text{ on } \partial_{\sqcup} \left( I_r \times \cv(\Zd \cap B_r(y)) \right) \cap (\CQy{r}),
	\end{array}
\right. 
\end{equation}
and $\bar{u} $ be the weak solution of the homogenized, continuous in space, parabolic equation
\begin{equation}\label{eq:ParaPercoHomogenized}
\left\{
	\begin{array}{ll}
	(\partial_t -  \dsigk \Delta) \bar{u} = 0  & \text{ in } I_r \times \cv(\Zd \cap B_r(y)), \smallskip \\
	\bar{u} = \tilde{f} & \text{ on } \partial_{\sqcup} \left( I_r \times \cv(\Zd \cap B_r(y)) \right),
	\end{array}
\right.
\end{equation} 
where the boundary condition $\tilde{f}$ is the extension of $f$ to the continuous parabolic cylinder defined by the formula
\begin{equation} \label{convolutionfftilde}
\tilde{f} := \left[ f \right]_{\Pa} \star \eta,
\end{equation}
and the extension $\left[ f \right]_{\Pa}$ is defined in the paragraph following Proposition~\ref{p.partitions}. Then, the following estimate holds
\begin{equation}\label{eq:HomoParaPerco}
\frac{1}{r} \norm{u - \bar{u}}_{\aL^2(\CQy{r})} \leq C r^{-\alpha} \norm{ \nabla f}_{\aL^{2 + \delta}(\CQy{r})}.
\end{equation}
\end{theorem}
\begin{remark} The equation~\eqref{eq:ParaPerco} is discrete in space and continuous in time, while equation~\eqref{eq:ParaPercoHomogenized} is both continuous in space and time. The solution $u$ and $\bar{u}$ coincide on the parabolic boundary $\partial_{\sqcup} \left( I_r \times \cv(\Zd \cap B_r(y)) \right) \cap (\CQy{r})$ for the equation on the clusters. All the norms in the inequality~\eqref{eq:HomoParaPerco} are discrete in space and continuous in time.
\end{remark}

\begin{remark}
The reason we define the homogenized limit to be continuous is the following: we need to use a number of results (e.g., regularity theory for the homogenized equation, the Meyers estimate) which are usually stated in the continuous setting. Moreover, one has explicit formulas for the elliptic and parabolic Green's functions and the continuous object is better behaved regarding scaling properties. On a higher level, the correct limiting object should be the continuous function as, over large-scales, the discrete lattice approximates the continuum.
\end{remark}
\begin{proof}[Proof of Theorem~\ref{thm:HomoParaPerco}]
By translation invariance of the model, we assume without loss of generality that $y= 0$ and do some preparation before the proof. We first define the minimal scale $\M_{\mathrm{hom}, \delta}(0)$ to be equal to
\begin{equation*}
    \M_{\mathrm{hom},\delta}(0) := \max \left(\M_{\mathrm{Poinc}}(0), \M_{q}(\Pa),\M_{\mathrm{corr}, \frac{1}{2}}(0),\M_{\mathrm{flux}, \frac{1}{2}}(0) \right),
\end{equation*}
where the parameter $q$ is assumed to be larger than $4d$ and will be fixed at the end of the proof. Using the stochastic integrability estimates~\eqref{stoint.Mpart},~\eqref{stoint.MPoinc},~\eqref{sto.intMCorr},~\eqref{sto.intMflux} on the four minimal scales together with the property~\eqref{eq:OSumm} of the $\O_s$ notation, one has
\begin{equation*}
    \M_{\mathrm{hom}, \delta}(0) \leq \O_s \left( C \right).
\end{equation*}
We record that under the assumption $r \geq \M_{\mathrm{hom},\delta}(0) \geq \M_{2d}(\Pa)$, one has
\begin{equation} \label{compvolumecluster}
    c r^d \leq \left| \C_\infty \cap B_r \right| \leq Cr^d,
\end{equation}
which allows to compare the number of points of the infinite cluster in the ball $B_r$ with the volume of the ball $B_r$. This estimate can be deduced by an application of the estimate~\eqref{eq:Mt} with the Cauchy-Schwarz inequality:
\begin{align*}
\vert B_r \vert^2 \leq \left(\sum_{\cu \in \Pa, \cu \cap B_r \neq \emptyset} \left(\size(\cu)\right)^{d}\right)^{2}    \leq \left(\sum_{\cu \in \Pa, \cu \cap B_r \neq \emptyset} 1\right) \left(\sum_{\cu \in \Pa, \cu \cap B_r \neq \emptyset} \left(\size(\cu)\right)^{2d}\right)  \leq C \vert\C_\infty \cap B_r \vert r^{d}. 
\end{align*}
We record the following interior regularity estimate for the homogenized function $\bar{u}$, which is standard for solutions of the heat equation (see~\cite[Theorem 9, Section 2.3]{EV10}): for every pair $(t,x) \in I_r \times \cv(\Zd \cap B_r)$, and every radii $r_1, r_2 > 0$ such that $I_{r_1}(t) \times B_{r_2}(x) \subset I_r \times \cv(\Zd \cap B_r)$, one has the inequality
\begin{equation}\label{eq:RegularityStandardbis}
\forall k,l \in \N, \qquad \vert \partial^l_t \nabla^{k+1}  \bar{u} \vert(t,x) \leq C_{k+2l}(r_1)^{-2l}(r_2)^{-k} \norm{\nabla \bar{u}}_{\aL^2(I_{r_1}(t) \times B_{r_2}(x))}.
\end{equation}  
We remark that in \cite[Theorem 9, Section 2.3]{EV10} the inequality is stated in the case when $r_1 = r_2$; The estimate~\eqref{eq:RegularityStandardbis} can be recovered by a careful investigation of the proof.

We introduce a cut-off function $\Upsilon$ in the parabolic cylinder $I_r \times B_r$ constant equal to $1$ in the interior of the cylinder and decreasing linearly to $0$ in a mesoscopic boundary layer of size $r' \ll r$,
\begin{equation}\label{eq:defGamma}
\left\{
	\begin{array}{ll}
	\Upsilon(t,x) \equiv 1 & (t,x) \in I_r \times B_r, \, \dist(x, \partial B_r) \geq 2r' ~\mbox{and}~ \dist(t, \partial I_r) \geq 2(r')^2, \\
	0 \leq \Upsilon(t,x) \leq 1 & (t,x) \in I_r \times B_r, \\
	\Upsilon(t,x) \equiv 0 & (t,x) \in I_r \times B_r, \, \dist(x, \partial B_r) \leq r'~\mbox{or}~ \dist(t, \partial I_r) \leq (r')^2.
	\end{array}
\right.
\end{equation}

The precise value of the parameter $r'$ is given by the formula $r' = r^{1-\beta}$ for some small exponent $\beta$ whose value is decided at the end of the proof. We additionally assume that the function $\Upsilon$ is smooth and satisfies the estimate
\begin{equation}\label{eq:BoundCutoff}
\forall k,l \in \N, \qquad \vert \partial^l_t \nabla^k \Upsilon \vert \leq C_{k+2l}(r')^{-(k + 2l)}.
\end{equation}
With these quantities, we can prove the following lemma.
\begin{lemma}\label{lem:ubinterior}
We have the estimate 
\begin{align} 
    \forall k,l \in \N, \qquad & \norm{\partial^l_t \nabla^k (\Upsilon \nabla \bar{u})}_{L^{\infty}(I_r \times B_r)} \leq C_{k+2l}(r')^{-(k + 2l)} \left( \frac{r}{r'}\right)^{\frac{2 + d}{2}} \left\| \nabla f \right\|_{\underline{L}^{2+\delta} \left(\CQ{r}\right)}. \label{eq:RegularityStandard1}
\end{align}
\end{lemma}
\begin{proof}
First, by using the inequality~\eqref{eq:RegularityStandardbis} and the fact that the map $\Upsilon$ is supported outside a boundary layer of size $r'$ in the parabolic cylinder $I_r \times B_r$, we obtain the estimate
\begin{align}
\forall (t,x) \in \supp(\Upsilon), \forall k,l \in \N, \quad \vert \partial^l_t \nabla^k (\Upsilon \nabla \bar{u}) \vert(t,x) &\leq C_{k+2l}(r')^{-(k+2l)} \norm{\nabla \bar{u}}_{\aL^2\left(I_{r'}(t) \times B_{r'}(x)\right)} \notag \\     
&\leq C_{k+2l}(r')^{-(k+2l)} \left( \frac{r}{r'}\right)^{\frac{2 + d}{2}} \left\| \nabla \bar{u} \right\|_{\underline{L}^{2} \left(I_r \times \cv(\Zd \cap B_r)\right)} \label{eq:RegularityStandard1Step11}.
\end{align}
The inequality~\eqref{eq:RegularityStandard1Step11} implies the $L^\infty$-estimate
\begin{equation}
 \forall k,l \in \N, \quad \left\| \partial^l_t \nabla^k (\Upsilon \nabla \bar{u}) \right\|_{L^{\infty}(I_r \times B_r)}     
\leq C_{k+2l}(r')^{-(k+2l)} \left( \frac{r}{r'}\right)^{\frac{2 + d}{2}} \left\| \nabla \bar{u} \right\|_{\underline{L}^{2} \left(I_r \times \cv(\Zd \cap B_r)\right)} \label{eq:RegularityStandard1Step1}.
\end{equation}
We then state the global Meyers estimate for the map $\bar{u}$: there exists an exponent $\delta_0 := \delta_0 \left(d , \lambda, \p \right) > 0$ such that for every $\delta' \in [0 , \delta_0]$,
\begin{equation}\label{eq:Meyes}
\norm{\nabla \bar{u}}_{L^{2 + \delta'}\left(I_r \times \cv(\Zd \cap B_r)\right)} \leq C \norm{ \nabla \tilde{f}}_{L^{2 + \delta'}\left(I_r \times \cv(\Zd \cap B_r)\right)}.
\end{equation}
A proof of this result can be found in~\cite[Proposition 5.1]{GM79}, where the statement is given for cubes instead of parabolic cylinders (the adaptation to the setting considered here does not affect the proof). Moreover, one can estimate the $L^p$-norm of the (continuous) gradient of the function $\tilde f$ in terms of the $L^p$-norm of the (discrete) gradient of the maps $f$ and the sizes of the cubes of the partition. The formula is a consequence of~\cite[Lemma 3.3]{AD2} and recalled in \eqref{est.nablacoarsenL2}: for any $p \geq 1$, and any radius $r \geq \size \left( \cu_\Pa (0) \right)$,
\begin{equation*}
    \left\| \nabla \tilde f \right\|_{L^p \left(\cv(\Zd \cap B_r)\right)} \leq \int_{B_r \cap \C_\infty} \size \left( \cu_\Pa (x) \right)^{pd-1} \left| \nabla f(x) \right|^p \, dx.
\end{equation*}
Applying the H\"older inequality to this estimate with $r \geq \M_{\mathrm{hom},\delta}(0)$, using the assumption the minimal scale $\M_{\mathrm{hom}, \delta}(0)$ is larger than the minimal scale $\M_{q}(\Pa)$ so that Proposition~\ref{p.partitions} is valid, and choosing the parameter $q$ to be large enough (larger than the value $\frac{\left( 4 + 2 \delta \right) \left( \left( 2 + \frac 12 \delta \right) d - 1 \right) }{\delta}$), one obtains the following inequality: for any $p \in \left[2 , 2 + \frac 12 \delta\right]$,
\begin{equation} \label{est:valuet}
    \left\| \nabla \tilde f \right\|_{\underline{L}^p \left(\cv(\Zd \cap B_r)\right)} \leq C \left\| \nabla f \right\|_{\underline{L}^{2+\delta} \left(\CQ{r}\right)}.
\end{equation}
Together with~\eqref{eq:Meyes}, this shows the inequality, for any exponent $\delta' \in \left[0 , \min \left( \delta_0 , \frac 12 \delta \right) \right]$,
\begin{equation}\label{eq:ubMeyersStrong}
\norm{\nabla \bar{u}}_{\underline{L}^{2 + \delta'}(I_r \times \cv(\Zd \cap B_r))}  \leq C \left\| \nabla f \right\|_{\underline{L}^{2+\delta} \left(\CQ{r}\right)}.
\end{equation}
Putting the inequality~\eqref{eq:ubMeyersStrong} back into the estimate~\eqref{eq:RegularityStandard1Step1} concludes the proof of~\eqref{eq:RegularityStandard1}.
\end{proof}

The key ingredient in the proof of Theorem~\ref{thm:HomoParaPerco} is to use a modified two-scale expansion on the percolation cluster, defined for each $(t,x) \in \CQ{r}$ by the formula
\begin{equation}\label{eq:TwoScale}
w(t,x) := \bar{u}(t,x) + \Upsilon(t,x) \sum_{k=1}^d \Dr{k} \bar{u}(t,x) \chi_{e_k}(x),
\end{equation} 
as an intermediate quantity: we prove that the function $w$ is close to both functions $u$ and $\bar u$. Here and in the rest of this section, the map $\chi_{e_k}$ is the first order corrector normalized according the procedure described in Section~\ref{homogpercclu} around the point $y = 0$.  The proof of Theorem~\ref{thm:HomoParaPerco} can be decomposed into five steps.

\medskip

\textit{Step 1: Control over $\frac{1}{r}\norm{w - \bar{u}}_{\aL^2(\CQ{r})}$.} We use the estimate~\eqref{eq:RegularityStandard1} to compute
\begin{align*}
\frac{1}{r}\norm{w - \bar{u}}_{\aL^2(\CQ{r}))} &= \frac{1}{r}\norm{\sum_{k=1}^d \Upsilon \left(\Dr{k} \bar{u}\right) \chi_{e_k}}_{\aL^2(\CQ{r}))} \\
& \leq \frac{1}{r} \norm{\Upsilon \nabla \bar{u}}_{L^{\infty}(I_r \times B_r)} \sum_{k = 1}^d \norm{\chi_{e_k}}_{\aL^2(\C_\infty \cap B_r)} \\
& \leq \frac{C}{r}  \left(\frac r{r'} \right)^{\frac{2+d}2} \left\| \nabla f \right\|_{\underline{L}^{2+\delta} \left(\CQ{r}\right)} \sum_{k = 1}^d \norm{\chi_{e_k}}_{\aL^2(\C_\infty \cap B_r)}.
\end{align*}
Using the assumption $r \geq \M_{\mathrm{hom}, \delta}(0) \geq \M_{\mathrm{corr}, \frac 12}(0)$, we deduce
\begin{equation*}
    \frac{1}{r}\norm{w - \bar{u}}_{\aL^2(\CQ{r}))} \leq C r^{-\frac 12} \left(\frac r{r'} \right)^{\frac{2+d}2} \left\| \nabla f \right\|_{\underline{L}^{2+\delta} \left(\CQ{r}\right)} .
\end{equation*}
The proof of Step 1 is complete.

\smallskip

\textit{Step 2: Control of $\frac{1}{r}\norm{w - u}_{\aL^2(\CQ{r})}$ by the norm $\norm{(\partial_t - \nabla \cdot \a \nabla) w}_{\aL^2(I_r; \aH^{-1}(\C_\infty \cap Br))}$.} 
We first note that the functions $w$ and $u$ are equal on the boundary of the parabolic cylinder $\CQ{r}$, and use the assumption $r \geq \M_{\mathrm{hom},\delta}(0) \geq \M_{\mathrm{Poinc}}(0)$ to apply the Poincar\'e inequality for each fixed time $t$ and then integrate over time. This proves
$$
\frac{1}{r}\norm{w - u}_{\aL^2(\CQ{r})} \leq \norm{\nabla (w - u) }_{\aL^2(\CQ{r})}. 
$$
Then, we use an integration by part and the uniform ellipticity of the environment on the infinite cluster
\begin{align*}
\norm{\nabla (w - u) }^2_{\aL^2(\CQ{r})} &\leq \frac{1}{\lambda \vert \CQ{r}\vert}\int_{I_r} \int_{\C_\infty \cap B_r} \nabla (w - u) \cdot \a \nabla (w - u) \\
& = \frac{1}{\lambda \left\vert \CQ{r} \right\vert}\int_{I_r} \int_{\C_\infty \cap B_r}  \left(-\nabla \cdot \a \nabla (w - u)\right) (w - u). 
\end{align*}
The fact that the functions $w,u$ have the same initial condition over $\C_\infty \cap B_r$ implies that the following integral is non-negative
\begin{align*}
\int_{I_r} \int_{\C_{\infty} \cap B_r}  \left(\partial_t (w - u)\right) (w - u) &= \frac{1}{2} \left(\norm{(w - u)(0, \cdot)}^2_{L^2(\C_{\infty} \cap B_r)} - \norm{(w - u)(-r^2, \cdot)}^2_{L^2(\C_{\infty} \cap B_r)}\right) \\
& = \frac{1}{2} \norm{(w - u)(0, \cdot)}^2_{L^2(\C_{\infty} \cap B_r)} \geq 0.
\end{align*}
We combine this formula and equation~\eqref{eq:ParaPerco} to obtain 
\begin{align*}
\norm{\nabla (w - u) }^2_{\aL^2(\CQ{r})} & \leq \frac{1}{\lambda \vert \CQ{r} \vert}\int_{I_r} \int_{\C_{\infty} \cap B_r}  \left(\left( \partial_t - \nabla \cdot \a \nabla \right) (w - u)\right) (w - u) \\
& \leq \frac{1}{\lambda} \norm{w - u  }_{\aL^2(I_r; \aH^1(\C_\infty \cap B_r))} \norm{\left(\partial_t - \nabla \cdot \a \nabla \right)w }_{\aL^2(I_r; \aH^{-1}(\C_\infty \cap B_r))}.
\end{align*}
This shows that
$$
\frac{1}{r}\norm{w - u}_{\aL^2(\C_\infty \cap B_r)} \leq C\norm{(\partial_t - \nabla \cdot \a \nabla) w}_{\aL^2(I_r; \aH^{-1}(\C_\infty \cap B_r))}.
$$

\smallskip

\textit{Step 3: Control over $\norm{(\partial_t - \nabla \cdot \a \nabla) w}_{\aL^2(I_r; \aH^{-1}(\C_\infty \cap B_r))}$.}
In this step, we adopt the finite difference notation and recall the identity $(\partial_t - \nabla \cdot \a \nabla) w = (\partial_t + \D^* \cdot \a \D) w$. To estimate the $\underline{H}^{-1}$-norm of $(\partial_t - \nabla \cdot \a \nabla) w$, the idea is to derive an explicit formula for this quantity by using the definition of $w$ given in~\eqref{eq:TwoScale} and to make a centered flux $\tilde{\mathbf{g}}_{e_k} = \a (\D \phi_{e_k} + e_k) - \frac 12 \sigk e_k$ appear.
We first calculate $\partial_t w$ and~$\D w$ and obtain the formulas
\begin{align*}
\left\{
	\begin{array}{ll}
	\partial_t w &= \partial_t \bar{u} + \sum_{k=1}^d \partial_t(\Upsilon \Dr{k} \bar{u}) \chi_{e_k},\\
	\D w &=  (1-\Upsilon)\D \bar{u} + \sum_{k=1}^d(\Upsilon \Dr{k} \bar{u})(e_k + \D \chi_{e_k}) + \sum_{k=1}^d \D(\Upsilon \Dr{k} \bar{u}) \chi_{e_k}.
	\end{array}
\right. 
\end{align*}
We combine the two equations to calculate $(\partial_t + \D^* \cdot \a \D) w$,
\begin{equation} \label{eq:190511}
\begin{split}
(\partial_t + \D^* \cdot \a \D) w &= \partial_t \bar{u} + \sum_{k=1}^d \partial_t(\Upsilon \Dr{k} \bar{u}) \chi_{e_k} +  \D^* \cdot  \left((1-\Upsilon) \a \D \bar{u}\right)  \\
& \qquad + \sum_{k=1}^d \D^* \cdot \left((\Upsilon \Dr{k} \bar{u})\a(e_k + \D \chi_{e_k})\right) + \sum_{k=1}^d \D^* \cdot \left(\a \D(\Upsilon \Dr{k} \bar{u}) \chi_{e_k}\right).
\end{split}
\end{equation}
Then, we use equation~\eqref{eq:ParaPercoHomogenized} which reads $\partial_t \bar{u} = \frac{1}{2}\sigk \Delta \bar{u}$ to replace the term $\partial_t \bar{u}$ in the equation above. Notice that here $\frac{1}{2} \sigk \Delta \bar{u}$ refers to the continuous Laplacian, but using the regularity properties on the function $\bar u$ stated in~\eqref{eq:RegularityStandard1}, we can replace this term by the discrete Laplacian $- \frac{1}{2}\sigk \D^* \cdot  \D \bar{u}$ by paying only a small error. The advantage of this operation is that we can use the two terms $- \frac{1}{2}\sigk \D^* \cdot (\Upsilon \D \bar{u})$ and $\sum_{k=1}^d \D^* \cdot \left((\Upsilon \Dr{k} \bar{u})\a(e_k + \D \chi_{e_k})\right)$ to make the flux appear: we have
\begin{align} \label{eq:1905112}
\sum_{k=1}^d \D^* \cdot \left((\Upsilon \Dr{k} \bar{u})\a(e_k + \D \chi_{e_k})\right)- \frac{1}{2}\sigk \D^* \cdot  (\Upsilon \D \bar{u}) &= \sum_{k=1}^d \D^* \cdot \left((\Upsilon \Dr{k} \bar{u}) \left(\a(e_k + \D \chi_{e_k}) - \frac{1}{2}\sigk e_k\right) \right)\notag \\
&= \sum_{k=1}^d \D^*(\Upsilon \Dr{k} \bar{u}) \cdot \tilde{\mathbf{g}}^*_{e_k}, 
\end{align} 
where  $\tilde{\mathbf{g}}^*_{e_k}$ is a translated version of the flux $\tilde{\mathbf{g}}_{e_k}$ defined by the formula, for each $x \in \C_\infty$,
\begin{equation*}
    \tilde{\mathbf{g}}^*_{e_k}(x) := \begin{pmatrix}
 T_{- e_1} \left[\a \left( \D \chi_{e_k} + e_k \right) - \dsigk e_k \right]_1 \\[3mm]
\vdots \\[3mm]
 T_{- e_d} \left[\a \left( \D \chi_{e_k} + e_k \right) - \dsigk e_k \right]_d  \\
\end{pmatrix},
\end{equation*}
where we recall the notation $\left[\a \left( \D \chi_{e_k} + e_k \right) - \dsigk e_k \right]_i$ introduced in Section~\ref{section1.6.1} for the $i$th-component of the vector ${\a \left( \D \chi_{e_k} + e_k \right) - \dsigk e_k}$.
In Appendix~\ref{appendixb}, it is proved that the translated flux $ \tilde{\mathbf{g}}^*_{e_k}$ has similar properties as the centered flux~$\tilde{\mathbf{g}}_{e_k}$. In particular, it is proved in Remark~\ref{rmk:Flux} that for every radius $r \geq \M_{\mathrm{corr},\frac 12}(0)$,
\begin{equation*}
    \norm{\tilde{\mathbf{g}}^*_{e_k}}_{\aH^{-1}(\C_\infty \cap B_r)} \leq C r^{\frac 12}.
\end{equation*}
Combining the identities~\eqref{eq:190511} and~\eqref{eq:1905112}, one obtains 
\begin{equation}\label{eq:DecomTwoScale}
\begin{split}
(\partial_t - \D^* \cdot \a \D) w &= \underbrace{\frac{1}{2}(\nabla \cdot \sigk (\Upsilon \nabla \bar{u}) - (- \D^* \cdot \sigk (\Upsilon \D \bar{u})))}_{\text{\eqref{eq:DecomTwoScale}-a}} +  \underbrace{\sum_{k=1}^d \partial_t(\Upsilon \Dr{k} \bar{u}) \chi_{e_k}}_{\text{\eqref{eq:DecomTwoScale}-b}} \\ 
&\qquad + \underbrace{\D^* \cdot  \left((1-\Upsilon) \a \D \bar{u}\right)}_{\text{\eqref{eq:DecomTwoScale}-c1}} + \underbrace{\frac{1}{2}(\nabla \cdot (\sigk (1-\Upsilon) \nabla \bar{u})}_{\text{\eqref{eq:DecomTwoScale}-c2}} +  \underbrace{\sum_{k=1}^d \D^* (\Upsilon \Dr{k} \bar{u}) \cdot \tilde{\mathbf{g}}^*_{e_k}}_{\text{\eqref{eq:DecomTwoScale}-d}} \\
& \qquad +  \underbrace{\sum_{k=1}^d \D^* \cdot \left(\a \D(\Upsilon \Dr{k} \bar{u}) \chi_{e_k}\right)}_{\text{\eqref{eq:DecomTwoScale}-e}}.
\end{split}
\end{equation}
There remains to use triangle inequality and estimate the $\aL^2(I_r; \aH^{-1}(\C_\infty \cap B_r))$-norm of each term. The following estimates will be used several times: given two functions $A: \CQ{r} \rightarrow \mathbb{R}$ and $B: I_r \times \left(\C_\infty \cap B_r\right) \rightarrow \mathbb{R}$, one has
\begin{equation}\label{eq:Simple}
\begin{split}
\norm{AB}_{\aL^2(I_r; \aH^{-1}(\C_\infty \cap B_r))} &= \sup_{\norm{v}_{\aL^2(I_r; \aH^{1}(\C_\infty \cap B_r))} \leq 1} \frac{1}{\vert \CQ{r} \vert}\int_{\CQ{r}} ABv \\
&\leq \norm{A}_{\aL^2(I_r; \aH^{-1}(\C_\infty \cap B_r))} \sup_{\norm{v}_{\aL^2(I_r; \aH^{1}(\C_\infty \cap B_r))} \leq 1} \norm{Bv}_{\aL^2(I_r; \aH^{1}(\C_\infty \cap B_r))} \\
&\leq \norm{A}_{\aL^2(I_r; \aH^{-1}(\C_\infty \cap B_r))} \left(\norm{B}_{L^\infty(\CQ{r})} + r \norm{\nabla B}_{L^\infty(\CQ{r})}\right).
\end{split}
\end{equation}
From the definition of the $\aL^2(I_r; \aH^{-1}(\C_\infty \cap B_r))$-norm, one also has the estimate
\begin{equation}\label{eq:Simpler}
\norm{A}_{\aL^2(I_r; \aH^{-1}(\C_\infty \cap B_r))} \leq r \norm{A}_{\underline{L}^2(\CQ{r})} \leq r \norm{A}_{L^\infty(\CQ{r})}. 
\end{equation}
The term \eqref{eq:DecomTwoScale}-a is a difference between a discrete derivative and a continuous derivative; it can be estimated in terms of the third derivative of the function $\bar u$. Using the estimates~\eqref{eq:RegularityStandard1} and~\eqref{eq:Simpler} shows
\begin{equation}\label{eq:DecomTwoScale-a}
\begin{split}
\norm{\text{\eqref{eq:DecomTwoScale}-a}}_{\aL^2(I_r; \aH^{-1}(\C_\infty \cap B_r))} & \leq r \norm{\nabla^2(\Upsilon \nabla \bar{u}) }_{L^{\infty}(I_r \times B_r)} \\ & \leq C r ^{-1} \left( \frac{r}{r'}\right)^{3 + \frac{d}{2}} \norm{\nabla f}_{\underline{L}^{2+\delta}(\CQ{r})}.
\end{split}
\end{equation}
A similar strategy can be used to estimate the term~\eqref{eq:DecomTwoScale}-b
\begin{align}\label{eq:DecomTwoScale-b}
\norm{\text{\eqref{eq:DecomTwoScale}-b}}_{\aL^2(I_r; \aH^{-1}(\C_\infty \cap B_r))} &\leq  \left( r
\norm{ \partial_t \nabla (\Upsilon \nabla \bar{u})}_{L^\infty(I_r \times B_r)} + \norm{ \partial_t (\Upsilon \nabla \bar{u})}_{L^\infty(I_r \times B_r)} \right) \sum_{k=1}^d \norm{ \chi_{e_k}}_{\aH^{-1}(\C_\infty \cap B_r)} \notag \\ 
& \leq C \frac{r}{(r')^3} \left( \frac{r}{r'}\right)^{\frac{2 + d}{2}} \norm{\nabla f}_{\underline{L}^{2+\delta}(\CQ{r})} \sum_{k=1}^d \norm{\chi_{e_k}}_{\aL^{2}(\C_\infty \cap B_r)} \\
& \leq C r^{-\frac{3}{2}}  \left( \frac{r}{r'}\right)^{4 + \frac{d}{2}} \norm{\nabla f}_{\underline{L}^{2+\delta}(\CQ{r})}, \notag
\end{align}
where we use the assumption $r \geq \M_{\mathrm{hom},\delta}(0) \geq \M_{\mathrm{corr},\frac 12}(0)$ to obtain the sublinearity of the corrector and the regularity estimate~\eqref{eq:RegularityStandard1} to go from the second line to the third line. 

To estimate the term \eqref{eq:DecomTwoScale}-c1, we note that the function $(1 - \Upsilon)$ is equal to $0$ outside a mesoscopic boundary layer of size $r'$ of the ball $B_r$. We thus apply the Meyers estimate~\eqref{eq:Meyes}, with the exponent $\delta' = \min \left( \delta_0 , \frac 12 \delta \right)$, and the H\"older inequality. This shows 
\begin{equation}\label{eq:DecomTwoScale-c}
\begin{split}
\norm{\text{\eqref{eq:DecomTwoScale}-c1}}_{\aL^2(I_r; \aH^{-1}(\C_\infty \cap B_r))} &\leq \norm{ (1-\Upsilon) \a \D \bar{u}}_{\aL^{2}(\CQ{r})} \\
& \leq \norm{1-\Upsilon}_{\aL^{\frac{4+2\delta'}{\delta'}}(I_r \times B_r)}  \norm{ \nabla \bar{u}}_{\aL^{2+\delta'}(I_r \times B_r)} \\
& \leq C\left(\frac{r'}{r}\right)^{\frac{\delta'}{4+2\delta'}}   \norm{\nabla f}_{\aL^{2+\delta}(\CQ{r})}.
\end{split}
\end{equation}
where we used the H\"older inequality to go from the first line to the second line and the Meyers estimate to go from the second line to the third line.

We want to apply a similar technique to treat the term \eqref{eq:DecomTwoScale}-c2 since it is also a boundary layer term. However, we should notice that here the derivative $\nabla$ is the continuous gradient defined on $\Rd$ and there is no conductance $\a$, thus we cannot apply a discrete integration by part on the cluster. We will focus on this term later in Step 4.

To estimate the term \eqref{eq:DecomTwoScale}-d, we apply the inequality~\eqref{eq:Simple}, the regularity estimate~\eqref{eq:RegularityStandard1}, and we use the assumption $r \geq \M_{\mathrm{hom}, \delta}(0) \geq \M_{\mathrm{flux}, \frac 12}(0)$. We obtain
\begin{align}\label{eq:DecomTwoScale-d}
\norm{\text{\eqref{eq:DecomTwoScale}-d}}_{\aL^2(I_r; \aH^{-1}(\C_\infty \cap B_r))} &\leq \left(\norm{ \nabla(\Upsilon \nabla \bar{u})}_{L^{\infty}(I_r \times B_r)} + r \norm{ \nabla^2(\Upsilon \nabla \bar{u})}_{L^{\infty}(I_r \times B_r)}\right)  \sum_{k=1}^d\norm{\tilde{\mathbf{g}}^*_{e_k}}_{\aH^{-1}(\C_\infty \cap B_r)} \notag \\
&\leq C r^{- \frac 12}  \left(\frac{r}{r'}\right)^{ 3 +\frac{d}{2} }  \norm{\nabla f}_{\underline{L}^{2+\delta}(\CQ{r})} .
\end{align}
The term~\eqref{eq:DecomTwoScale}-e can be estimated thanks to an integration by part and the regularity estimate~\eqref{eq:RegularityStandard1}. This yields
\begin{equation}\label{eq:DecomTwoScale-e}
\begin{split}
\norm{\text{\eqref{eq:DecomTwoScale}-e}}_{\aL^2(I_r; \aH^{-1}(\C_\infty \cap B_r))} &\leq \norm{\nabla(\Upsilon \nabla \bar{u})}_{L^\infty(I_r \times B_r)} \sum_{k=1}^d  \norm{\chi_{e_k}}_{\aL^{2}(\CQ{r})} \\
& \leq C r^{- \frac 12} \left(\frac{r}{r'}\right)^{2 + \frac{d}{2}}   \norm{\nabla f}_{\underline{L}^{2+\delta}(\CQ{r})}.
\end{split}
\end{equation}

\smallskip

\textit{Step 4: Control over the term $\norm{\nabla \cdot (\sigk (1-\Upsilon) \nabla \bar{u})}_{\aL^2(I_r; \aH^{-1}(\C_\infty \cap B_r))}$.} As was already mentioned, we cannot use a discrete integration by parts to estimate the $L^2(I_r; \aH^{-1}(\C_\infty \cap B_r))$-norm of this term. The strategy relies on the interior regularity estimate~\eqref{eq:RegularityStandardbis} which requires careful treatments since it is close to the boundary. We apply the Whitney decomposition on the ball $B_r$ stated below with a minor adaptation to triadic cubes. 
\begin{lemma}[Whitney decomposition]
There exits a family of closed triadic cubes $\{Q_j\}_{j \geq 0}$ such that
\begin{enumerate}[label=(\roman*)]
    \item $B_r = \bigcup_j Q_j$ and the cubes $Q_j$ have disjoint interiors;
    \item $\sqrt{d}\size(Q_j) \leq \dist(Q_j, \partial B_r) \leq 4\sqrt{d}\size(Q_j)$;
    \item Two neighboring cubes $Q_j$ and $Q_k$ have comparable sizes in the sense that
    $$
    \frac{1}{3} \leq \frac{\size(Q_k)}{\size(Q_j)} \leq 3;
    $$
    \item Each cube $Q_j$ has at most $C(d)$ neighbors.
\end{enumerate}
\end{lemma}
We skip the construction of this partition, refer to~\cite[Theorem 3]{stein1970singular} or ~\cite[Appendix J]{Grafakos} for the proof and to Figure~\ref{fig:figure5} for an illustration. With the help of this decomposition, we can estimate the norm $\norm{\nabla \cdot (\sigk (1-\Upsilon) \nabla \bar{u})}_{\aL^2(I_r; \aH^{-1}(\C_\infty \cap B_r))}$. We first relabel the cubes of the decomposition according to their size; we write
\begin{align*}
\{Q_j\}_{i \geq 1}:= \bigcup_{n = 0}^{\infty} \bigcup_{k=1}^{M_n}\{Q_{n,k}\}, \qquad 3^{-(n+1)}r \leq \size(Q_{n,k}) < 3^{-n}r.   
\end{align*} where $M_n$ is the number of the cubes whose size belongs to the interval $[3^{-(n+1)}r,  3^{-n}r)$. Then, we decompose the set $\supp(1 - \Upsilon)$ into two parts (see Figure~\ref{fig:figure5})
\begin{align*}
    \supp(1-\Upsilon) &= \Pi_1 \sqcup \Pi_2,\\
    \Pi_1 &:= \left\{(t,x) \in I_r \times B_r : \dist(x, \partial B_r) \leq 2 r',  -r^2 + 2(r')^2  \leq t \leq 0 \right\}, \\
    \Pi_2 &:= \left\{(t,x) \in I_r \times B_r : -r^2 \leq t \leq -r^2 + 2(r')^2 \right\}.
\end{align*}
We estimate the weak norm thanks to its definition: we let $\varphi$ be a function from $\CQ{r}$ to $\R$ which satisfies $\norm{\varphi}_{\aL^2(I_r; \aH^{1}(\C_\infty \cap B_r))} \leq 1$ and is equal to $0$ on the boundary $I_r \times \partial_{\a} \left(\C_\infty \cap B_r \right)$. We split the integral
\begin{equation*}
\int_{I_r} \int_{\C_\infty \cap B_r}  \nabla \cdot (\sigk (1-\Upsilon) \nabla \bar{u}) \varphi = \int_{(I_r \times \Zd) \cap \Pi_1}  \nabla \cdot (\sigk (1-\Upsilon) \nabla \bar{u}) \varphi + \int_{(I_r \times \Zd) \cap \Pi_2}  \nabla \cdot (\sigk (1-\Upsilon) \nabla \bar{u}) \varphi,
\end{equation*}
and treat the two terms separately. 
\begin{figure}[h!] 
    \centering
    \includegraphics[scale = 0.5]{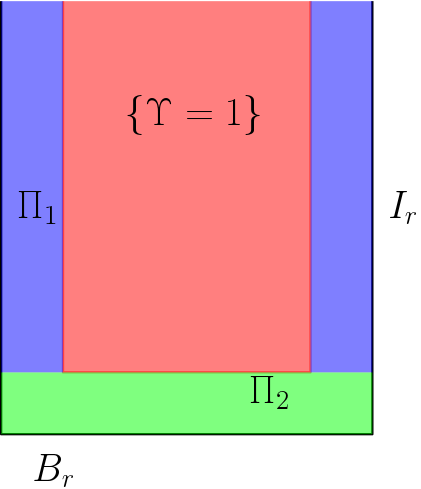}
    \hspace{2cm}
    \includegraphics[scale = 0.5]{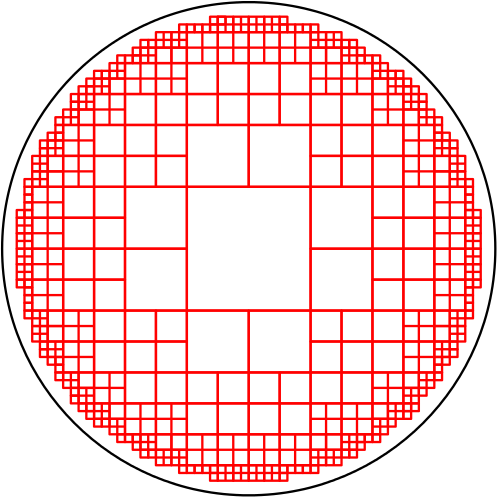}
    \caption{The figure on the left illustrates the partition of the cylinder $I_r \times B_r$, where the domain in blue stands for the set $\Pi_1$, the one in green for the set $\Pi_2$ and the one in red for $\{(t,x): \Upsilon(t,x) = 1\}$. The figure on the right illustrates a Whitney decomposition in the ball $B_r$, where we use dyadic cubes to improve the readability.\label{fig:figure5}}
    \label{fig:Whitney}
\end{figure}

\textit{Step 4.1: Control of the weak norm over $\Pi_1$.}
For the term involving the set $\Pi_1$, we use the Whitney decomposition to integrate on every cube of the partition. We first introduce the time intervals, for $m,n \in \N$, $I_{1,m} := -m + \left(-1, 0\right]$ and $n I_{1,m} := -m + \left(-n, 0\right]$, and partition the boundary layer $\Pi_1$ according to the formula
\begin{equation*}
    \Pi_1 = \bigcup_{m=0}^{\lfloor r^2 - (r')^2 \rfloor} \bigcup_{n = 0}^{\infty} \bigcup_{k=1}^{M_n} \left( I_{1,m} \times Q_{n,k} \right) \cap \Pi_1.
\end{equation*}
We remark that we can restrict our attention to the cubes whose sizes is between $1$ and $r'$ thanks to the properties of the Whitney decomposition. Indeed, the cubes of size larger than $r'$ remain outside the boundary layer $\Pi_1$, since the distance of a cube to the boundary is comparable to its size. On the other hand, the cubes of size smaller than $1$ will not contain a point in the lattice $\intr(\cv(\Zd \cap B_r))$, as these cubes are too close to the boundary, and the definition \eqref{eq:defCv} implies that all the points of the lattice in the interior $\intr(\cv(\Zd \cap B_r))$ are at distance at least $\frac{1}{\sqrt{2}}$ from the boundary. We thus have the following partition, if we denote by $n_0$ and $n_1$ the integers such that $3^{-(n_0+1)}r \leq  r' < 3^{- n_0}r$, and $3^{-(n_1 + 1)}r \leq 1 < 3^{-n_1}r$,
\begin{equation*}
\Pi_1 = \bigcup_{m=0}^{\lfloor r^2 - (r')^2 \rfloor} \bigcup_{n = n_0}^{n_1} \bigcup_{k=1}^{M_n} \left( I_{1,m} \times Q_{n,k} \right) \cap \Pi_1.
\end{equation*}
Using this partition, we can split the integral
\begin{equation}\label{eq:Pi1sum}
 \int_{(I_r \times \C_{\infty}) \cap \Pi_1}  \nabla \cdot (\sigk (1-\Upsilon) \nabla \bar{u}) \varphi 
  = \sum_{m = 0}^{\lfloor r^2 - (r')^2 \rfloor}\sum_{n=n_0}^{n_1} \sum_{k=1}^{M_n} \int_{(I_r \times \C_{\infty}) \cap (I_{1,m} \times Q_{n,k})} \nabla \cdot (\sigk (1-\Upsilon) \nabla \bar{u}) \varphi.    
\end{equation}
We fix a cylinder $I_{1,m} \times Q_{n,k}$, apply the Cauchy-Schwarz inequality and use the interior regularity estimate \eqref{eq:RegularityStandardbis} of the function $\bar{u}$ in the cylinder $I_{1,m} \times Q_{n,k}$ with the property that the distance $\dist(Q_{n,k}, \partial B_r)$ is larger than $\sqrt{d}\size(Q_{n,k})$ and the inclusion $2Q_{n,k} \subset B_r$. We obtain
\begin{align*}
\left\vert \int_{(I_r \times \C_{\infty}) \cap (I_{1,m} \times Q_{n,k})} \nabla \cdot (\sigk (1-\Upsilon) \nabla \bar{u}) \varphi \right\vert &\leq \norm{ \nabla \cdot (\sigk (1-\Upsilon) \nabla \bar{u})}_{L^2(I_{1,m} \times Q_{n,k})} \norm{\varphi}_{L^2(I_{1,m} \times (\C_{\infty} \cap Q_{n,k})} \\
& \leq C (3^{-n}r)^{-1} \norm{ \nabla \bar{u}}_{L^2(2 I_{1,m} \times 2 Q_{n,k})} \norm{\varphi}_{L^2(I_{1,m} \times (\C_{\infty} \cap Q_{n,k})}.
\end{align*}
We sum over all the cubes $\{Q_{n,k}\}_{1 \leq k \leq M_n}$ and apply the Cauchy-Schwarz inequality
\begin{align*}
\lefteqn{\sum_{k=1}^{M_n} \left\vert \int_{(I_r \times \C_{\infty}) \cap (I_{1,m} \times Q_{n,k})} \nabla \cdot (\sigk (1-\Upsilon) \nabla \bar{u}) \varphi \right\vert} \qquad & \\ &
\leq \sum_{k=1}^{M_n} C (3^{-n}r)^{-1} \norm{ \nabla \bar{u}}_{L^2(2 I_{1,m} \times 2 Q_{n,k})} \norm{\varphi}_{L^2(I_{1,m} \times (\C_{\infty} \cap Q_{n,k})} \\ 
&\leq C (3^{-n}r)^{-1} \left(\sum_{k=1}^{M_n} \norm{ \nabla \bar{u}}^2_{L^2(2 I_{1,m} \times 2 Q_{n,k})}\right)^{\frac{1}{2}} \norm{\varphi}_{L^2\left(I_{1,m} \times \left(\C_{\infty} \cap \left(\sqcup_{k=1}^{M_n} Q_{n,k}\right) \right)\right)}.
\end{align*}
We then use the following three ingredients:
\begin{itemize}
    \item Given a discrete set $A \subseteq \Zd$, the $L^2$-norm of coarsened function $[\varphi]_{\Pa}$ over $A$ is larger than the one of the function $\varphi$ over the set $\C_\infty \cap A$; \smallskip
    \item We have the inclusion $\sqcup_{k=1}^{M_n}Q_{n,k} \subset \{x \in B_r : \dist(x, \partial B_r) \leq 5 \times 3^{-n} \sqrt{d} r\}$; \smallskip
    \item We choose the vertex $z(\cdot)$ (defined in \eqref{def.coarsenu}) to be a point on the boundary $\partial B_r$ for the cubes of the partition intersecting $\partial B_r$. With this convention, the coarsened function $[\varphi]_{\Pa}$ is equal to zero on $\partial B_r$, so we can apply the Poincar\'e inequality for $[\varphi]_{\Pa}$ in the boundary layer  ${\{x \in B_r : \dist(x, \partial B_r) \leq 5\times 3^{-n} \sqrt{d} r\}}$.
    
\end{itemize}
We obtain the estimate
\begin{align*}
(3^{-n}r)^{-1} \norm{\varphi}_{L^2\left(I_{1,m} \times \left(\C_{\infty} \cap \left(\sqcup_{k=1}^{M_n} Q_{n,k}\right) \right)\right)} 
& \leq C (3^{-n}r)^{-1} \norm{[\varphi]_{\Pa} \indc_{\{x \in B_r \, : \, \dist(x, \partial B_r) \leq 5\times 3^{-n} \sqrt{d} r\}}}_{L^2(I_{1,m} \times (\Zd \cap B_r))} \\
& \leq C \norm{\nabla [\varphi]_{\Pa}  \indc_{\{x \in B_r \, : \, \dist(x, \partial B_r) \leq 5\times 3^{-n} \sqrt{d} r\}}}_{L^2(I_{1,m} \times (\Zd \cap B_r))}\\
& \leq C \norm{\nabla [\varphi]_{\Pa}}_{L^2(I_{1,m} \times (\Zd \cap B_r))}.
\end{align*}
We put these estimates back into \eqref{eq:Pi1sum} and apply once again the Cauchy-Schwarz inequality. We notice that summing over the integers between $n_0$ and $n_1$ gives an additional error term of order $\log^{\frac{1}{2}}(1+r)$,
\begin{equation}\label{eq:Pi1IPP}
\begin{split}
& \left\vert \int_{(I_r \times \C_{\infty}) \cap \Pi_1}  \nabla \cdot (\sigk (1-\Upsilon) \nabla \bar{u}) \varphi \right\vert \\
  \leq & C \sum_{m = 0}^{\lfloor r^2 - (r')^2 \rfloor}\sum_{n=n_0}^{n_1}  \left(\sum_{k=1}^{M_n} \norm{ \nabla \bar{u}}^2_{L^2(2 I_{1,m} \times 2 Q_{n,k})}\right)^{\frac{1}{2}} \norm{\nabla[\varphi]_{\Pa}}_{L^2(I_{1,m} \times (\Zd \cap B_r))} \\
  \leq & C \log^{\frac{1}{2}}(1+r) \left(\sum_{m = 0}^{\lfloor r^2 - (r')^2 \rfloor}\sum_{n=n_0}^{n_1} \sum_{k=1}^{M_n} \norm{ \nabla \bar{u}}^2_{L^2(2 I_{1,m} \times 2 Q_{n,k})}\right)^{\frac{1}{2}} \norm{\nabla[\varphi]_{\Pa}}_{L^2(I_{r} \times (\Zd \cap B_r))}  .  
\end{split}
\end{equation}
We then estimate the norm $\norm{\nabla[\varphi]_{\Pa}}_{L^2(I_{r} \times (\Zd \cap B_r))}$ thanks to the inequalities~\eqref{est.nablacoarsenL2},~\eqref{eq:Mt}, and the assumption $r > \M_{q}(\Pa)$. We obtain
\begin{align*}
\frac{1}{\vert \CQ{r} \vert^{\frac{1}{2}}}\norm{\nabla[\varphi]_{\Pa}}_{L^2(I_{r} \times (\Zd \cap B_r))} \leq Cr^{\frac{2d-1}{2q}}\norm{\nabla \varphi}_{\aL^2(I_{r} \times (\C_\infty \cap B_r))} \leq Cr^{\frac{2d-1}{2q}}.     
\end{align*}
Moreover, by the properties of the Whitney covering, the sum $\sum_{i = 0}^{\lfloor r^2 - (r')^2 \rfloor}\sum_{n=n_0}^{n_1} \sum_{k=1}^{M_n} \norm{ \nabla \bar{u}}^2_{L^2(2 I_{1,m} \times 2 Q_{n,k})}$ can be estimated by the $L^2$-norm of the function $\nabla \bar{u}$ in a boundary layer of size $6r'$ of the parabolic cylinder $I_r \times \cv(\Zd \cap B_r)$ (since every point in the ball $B_r$ belongs to at most $C(d)$ cubes of the form $2Q_j$). More specifically, we have the estimate
\begin{equation*}
\left(\frac{1}{\vert \CQ{r} \vert}\sum_{m = 0}^{\lfloor r^2 - (r')^2 \rfloor}\sum_{n=n_0}^{n_1} \sum_{k=1}^{M_n} \norm{ \nabla \bar{u}}^2_{L^2(2 I_{1,m} \times 2 Q_{n,k})}\right)^{\frac{1}{2}}  \leq C \norm{\nabla \bar{u} }_{\aL^2(I_r \times \{ x \in B_r \, : \, \dist(x, \partial B_r) \leq 6 r'\})}.
\end{equation*}
We then apply the H\"older's inequality and the global Meyers estimate~\eqref{eq:ubMeyersStrong} with the exponent $\delta' = \min\left(\delta_0, \frac{1}{2}\delta\right)$. We obtain
\begin{align*}
\left(\frac{1}{\vert \CQ{r} \vert}\sum_{m = 0}^{\lfloor r^2 - (r')^2 \rfloor}\sum_{n=n_0}^{n_1} \sum_{k=1}^{M_n} \norm{ \nabla \bar{u}}^2_{L^2(2 I_{1,m} \times 2 Q_{n,k})}\right)^{\frac{1}{2}} & \leq C \norm{\nabla \bar{u} }_{\aL^2(I_r \times \{ x \in B_r \, : \, \dist(x, \partial B_r) \leq 6 r'\})} \\
& \leq C\left(\frac{r'}{r}\right)^{\frac{\delta'}{4+2\delta'}}   \norm{\nabla f}_{\aL^{2+\delta}(\CQ{r})}.
\end{align*}
We conclude that 
\begin{multline}\label{eq:Pi1Weak}
\left\vert \frac{1}{\vert \CQ{r} \vert} \int_{(I_r \times \C_{\infty}) \cap \Pi_1}  \nabla \cdot (\sigk (1-\Upsilon) \nabla \bar{u}) \varphi \right\vert \\ \leq    C\log^{\frac{1}{2}}(1+r) r^{\frac{2d-1}{2q}} \left(\frac{r'}{r}\right)^{\frac{\delta'}{4+2\delta'}}   \norm{\nabla f}_{\aL^{2+\delta}(\CQ{r})}.     
\end{multline}

\textit{Step 4.2: Control the weak norm over $\Pi_2$.} One can repeat all the arguments above to estimate the weak norm over the set $\Pi_2$, but we should pay attention to the decomposition over the time interval $I_r$ since now the support of $\Pi_2$ is close to the time boundary (see Figure~\ref{fig:figure5}). We define the time intervals 
\begin{align*}
    \forall m \in \N, \qquad I_{2,m} := - r^2 + \left(\frac{2}{3}\right)^m(r')^2 + \left( -\frac{1}{3} \times\left(\frac{2}{3}\right)^m(r')^2, 0\right],
\end{align*}
so that they satisfy $2I_{2,m} \subset I_r$. We can then apply the same arguments as in the estimates~\eqref{eq:Pi1IPP} and~\eqref{eq:Pi1Weak} to obtain the inequality 
\begin{multline*}
\left\vert \frac{1}{\vert \CQ{r} \vert} \int_{(I_r \times \C_{\infty}) \cap \Pi_2}  \nabla \cdot (\sigk (1-\Upsilon) \nabla \bar{u}) \varphi \right\vert \\ \leq    C\log^{\frac{1}{2}}(1+r) r^{\frac{2d-1}{2q}}  \left(\frac{1}{r^2}\int_{-r^2}^{-r^2 + 2(r')^2}  \norm{\nabla \bar{u}(t, \cdot)}^2_{\aL^2(\cv(\Zd \cap B_r))} \, dt\right)^{\frac{1}{2}}.     
\end{multline*}
Then, we apply H\"older's inequality in the time variable and the estimate~\eqref{eq:ubMeyersStrong} to obtain 
\begin{align}\label{eq:Pi2Weak}
\left(\frac{1}{r^2}\int_{-r^2}^{-r^2 + 2(r')^2}  \norm{\nabla \bar{u}(t, \cdot)}^2_{\aL^2(\cv(\Zd \cap B_r))}  \, dt \right)^{\frac{1}{2}} \leq   \left(\frac{r'}{r}\right)^{\frac{2\delta'}{4+2\delta'}} \norm{\nabla f}_{\aL^{2+\delta}(\CQ{r})}.
\end{align}
This gives an estimate for the weak norm of the map $\nabla \cdot (\sigk (1-\Upsilon) \nabla \bar{u})$ over the set $\Pi_2$. Finally, we combine the estimates~\eqref{eq:Pi1Weak} and \eqref{eq:Pi2Weak} to conclude that 
\begin{equation}\label{eq:DecomTwoScale-c2}
\norm{\nabla \cdot (\sigk (1-\Upsilon) \nabla \bar{u})}_{\aL^2(I_r; \aH^{-1}(\C_\infty \cap B_r))} \leq C\log^{\frac{1}{2}}(1+r) r^{\frac{2d-1}{2q}} \left(\frac{r'}{r}\right)^{\frac{\delta'}{4+2\delta'}}   \norm{\nabla f}_{\aL^{2+\delta}(\CQ{r})}.
\end{equation}

\textit{Step 5: Choice of the parameters $q, \beta$ and conclusion.} We conclude the proof by combing the estimates \eqref{eq:DecomTwoScale-a}, \eqref{eq:DecomTwoScale-b}, \eqref{eq:DecomTwoScale-c}, \eqref{eq:DecomTwoScale-d}, \eqref{eq:DecomTwoScale-e}, \eqref{eq:DecomTwoScale-c2} and by choosing $r' = r^{1 - \beta}$ for some small exponent $\beta \in (0, 1/2]$ to obtain 
\begin{equation}\label{eq:HomoParaPerco2}
\frac{1}{r} \norm{u - \bar{u}}_{\aL^2(\CQ{r})}  \leq C \mathcal{E}(r, \beta, q)  \norm{\nabla f}_{\underline{L}^{2+\delta}(\CQ{r})},
\end{equation}
where the quantity $\mathcal{E}(r, \beta, q)$ is defined by the formula
\begin{equation} \label{def.alpha22/03}
\mathcal{E}(r, \beta, q) := r^{- \frac{1}{2} + \beta \left( 3 + \frac{d}{2} \right)} + \log^{\frac{1}{2}}(1+r)r^{\frac{2d-1}{2q}-\frac{\beta \delta'}{4+2\delta'}},
\end{equation}
where we recall that $\delta' = \min \left( \delta_0 , \frac 12 \delta \right)$ and that $\delta_0$ is the exponent given by the Meyers estimate stated in \eqref{eq:Meyes}. 

It remains to select a value for the exponents $\beta$ and $q$.
We first choose the value of the exponent $\beta$ and set $\beta := \frac{1}{12 + 2d}$ so that the first term in the right side of \eqref{def.alpha22/03} is equal to $r^{-\frac{1}{4}}$. Then we set $q := \frac{(12+2d)(2d-1)(4+2\delta')}{\delta'}$ so that $\frac{2d-1}{2q} = \frac{\beta \delta'}{2(4+2\delta')}$. With this choice, the second term in the right side of~\eqref{def.alpha22/03} is equal to $\log^{\frac{1}{2}}(1+r) r^{-\frac{\beta \delta'}{8+4\delta'}}$.

\smallskip

We obtain that Theorem~\ref{thm:HomoParaPerco} holds with the exponent $\alpha := \frac{\delta'}{(12+2d)(16+8\delta')} >0$. The proof is complete.

\end{proof}

\subsubsection{Large-scale $C^{0,1}$-regularity estimate} \label{largescaelC01}
The objective of this section is to prove the following $C^{0,1}$-large-scale regularity estimate for $\a$-caloric functions on the infinite cluster.

\begin{proposition} \label{heatkernelC01}
There exist a constant $C := C(d ,\lambda, \p ) < \infty$, an exponent $s := s(d, \lambda, \p) > 0$ such that for each point $y \in \Zd$, there exists a non-negative random variable $\M_{C^{0,1}-\mathrm{reg}}(y)$ satisfying
\begin{equation} \label{minscaleCo1reg}
\M_{C^{0,1}-\mathrm{reg}}(y) \leq \O_s (C)
\end{equation}
such that, for every $r \geq \M_{C^{0,1}-\mathrm{reg}}(y)$, and every weak solution $u \in \Hpar (\CQy{R})$ of the equation
\begin{equation*}
\partial_t u - \nabla \cdot \left(  \a \nabla u\right) = 0 ~\mbox{in}~ \CQy{R},
\end{equation*}
one has the estimate, for every radius $r \in \left[ \M_{C^{0,1}-\mathrm{reg}}(y) , R \right]$,
\begin{equation} \label{eq.heatkernelC01}
\sup_{t \in I_{r}} \left\| \nabla u \left(t , \cdot \right) \right\|_{\underline{L}^2 \left( \C_\infty \cap B_r(y) \right)} \leq \frac{C}{R} \left\| \left[ u \right]_\Pa - \left( \left[ u \right]_\Pa \right)_{ I_R \times B_R(y)} \right\|_{\underline{L}^2 \left(  I_R \times B_R(y) \right)}.
\end{equation}
\end{proposition}

\begin{remark}
The right side of the estimate involves the coarsened function $\left[ u \right]_\Pa$ and we do not try to remove the coarsening to obtain a result of the form
\begin{equation} \label{eq.heatkernelC012}
    \sup_{t \in I_{r}} \left\| \nabla u \left(t , \cdot \right) \right\|_{\underline{L}^2 \left( \C_\infty \cap B_r(y) \right)} \leq \frac{C}{R} \left\|  u  - \left(  u  \right)_{ \CQy{R}} \right\|_{\underline{L}^2 \left( \CQy{R} \right)},
\end{equation}
even though such a result would be more natural and should be provable. There are two reasons motivating this choice. First, the estimate~\eqref{eq.heatkernelC01} involving the coarsening is simpler to prove than the inequality~\eqref{eq.heatkernelC012} and this choice reduces the amount of technicalities in the proof. Second, the objective of this section is to prove the Lipschitz regularity on the heat-kernel stated in Theorem~\ref{gradbarlowintro} and the estimate~\eqref{eq.heatkernelC01} is sufficient in this regard. 
\end{remark}

This proposition proves that there exists a large random scale above which one has a good control on the gradient of $\a$-caloric functions. Such result belongs to the theory of large-scale regularity which is an important aspect of stochastic homogenization. The result presented above is a percolation version of a known result in the uniformly elliptic setting (see~\cite[Theorem 8.7]{armstrong2017quantitative}) and can be considered a first step toward the establishment of a general large-scale regularity theory for the parabolic problem on the infinite percolation cluster.

We do not establish such a general theory here but we believe that it should follow from similar arguments: in the elliptic setting a general large-scale regularity theory was established in~\cite{AD2} and the generalization to the parabolic setting should be achievable. The reason justifying this choice is that our objective is to prove an estimate on the gradient of the Green's function (Theorem~\ref{gradbarlowintro}) and we do not need the full strength of the large-scale regularity theory to prove this result.

The main idea of the proof is that, thanks to Theorem~\ref{thm:HomoParaPerco}, an $\a$-caloric function is well-approximated by a $\sigk$-caloric function. It is then possible to transfer the regularity known for $\sigk$-caloric functions to $\a$-caloric functions following the classical ideas of the regularity theory. Such result can only hold when the $\a$-caloric function  is well-approximated by a $\sigk$-caloric function which, according to Theorem~\ref{thm:HomoParaPerco}, only holds on large scales.

This strategy has been carried out in~\cite{armstrong2017quantitative} and is summarized in the following lemma, for which we refer to~\cite[Lemma 8.9]{armstrong2017quantitative}.

\begin{lemma}[Lemma 8.9 of~\cite{armstrong2017quantitative}] \label{lemma8.9}
Fix an exponent $\beta \in (0,1]$, $k \geq 1$ and $X \geq 1$. Let $R \geq 4X$ and $v \in L^2 \left( I_R \times B_R\right)$ have the property that, for every $r \in \left[X , \frac 14 R\right]$, there exists a function $w \in \Hpar (I_r \times B_r)$ which is a weak solution of
\begin{equation} \label{sigkharmfunct.lem3.6}
\partial_t w - \frac \sigk2  \Delta w = 0 ~\mbox{in}~ I_r \times B_r,
\end{equation}
satisfying
\begin{equation*}
\left\| v - w \right\|_{\underline{L}^2\left( I_{r/2} \times B_{r/2} \right)} \leq K r^{-\beta} \left\| v - \left( v \right)_{I_{4r} \times B_{4r}} \right\|_{\underline{L}^2\left( I_{4r} \times B_{4r} \right)}.
\end{equation*}
Then there exists a constant $C := C(\beta, K , d , \lambda) < \infty$ such that for every radius $r \in [X , R]$,
\begin{equation*}
\frac1r \left\|  v - \left( v \right)_{I_r \times B_r}\right\|_{\underline{L}^2 \left( I_r \times B_r \right)} \leq \frac CR \left\|  v - \left( v \right)_{I_R \times B_R}\right\|_{\underline{L}^2 \left(I_R \times B_R \right)}.
\end{equation*}
\end{lemma}
To prove Proposition~\ref{heatkernelC01}, we apply the previous lemma and combine it with Theorem~\ref{thm:HomoParaPerco}. One has to face the following difficulty: we want to apply the previous result in the setting of percolation where the functions are only defined on the infinite cluster and not on $\Rd$ as in the statement of Lemma~\ref{lemma8.9}.

To overcome this issue, the idea is to use the partition $\Pa$ to extend the function $u$, using the definition of the coarsened map $\left[ u \right]_\Pa$ stated in~\eqref{def.coarsenu}. 
The strategy of the proof is then the following:
\begin{enumerate}[label=(\roman*)]
\item Proving that the function $\left[ u \right]_\Pa$ is a good approximation of the function $u$. In particular we wish to prove that if the map $u$ is well-approximated by a $\sigk$-caloric function, then the map $\left[ u \right]_\Pa$ is also well-approximated by a $\sigk$-caloric function.
\item Apply Lemma~\ref{lemma8.9} to the function $\left[ u \right]_\Pa$ to obtain a large-scale $C^{0,1}$-regularity estimate for this map.
\item Transfer the result from the function $\left[ u \right]_\Pa$ to the function $u$.
\end{enumerate}
The details are carried out in the following proof.

\begin{proof}
Using the translation invariance of the model, we can assume without loss of generality that $y=0$. We also let $\delta_0 $ be the exponent which appears in the Meyers estimate stated in Proposition~\ref{p.Meyers} and consider the minimal scale $\M_{\mathrm{Meyers}}(0)$ given by Proposition~\ref{p.Meyers}. We consider the minimal scale $\M_{\mathrm{hom}, \delta_0}(0)$ given by Theorem~\ref{thm:HomoParaPerco} and we let $\alpha$ be the exponent which appears in the estimate~\eqref{eq:HomoParaPerco}.  We set $q := \max\left(\frac{d}{\alpha}, 4d\right)$ and let $\M_{q}(\Pa)$ be a minimal scale provided by Proposition~\ref{p.partitions}. In particular, one has, for each radius $r \geq \M_{q}(\Pa)$,
\begin{equation} \label{estimateminimalscaleC01}
r^{-d} \sum_{x \in \Zd \cap B_r} \size \left( \cu_{\Pa} (x) \right)^{q} \leq C~\mbox{and}~ \sup_{x \in \Zd \cap  B_{r}} \size \left( \cu_{\Pa} (x) \right) \leq r^{\frac{1}{q}}.
\end{equation}
The reasons justifying the choice of the exponents $q$ will be become clear later in the proofs. By Proposition~\ref{p.partitions}, one knows that the minimal scale $\M_{q}(\Pa)$ satisfies the stochastic integrability estimate
\begin{equation*}
\M_{q}(\Pa) \leq \O_s \left( C \right).
\end{equation*}
We then let $\M_{C^{0,1}-\mathrm{reg}}(0)$ be the minimal scale defined by the formula
\begin{equation*}
\M_{C^{0,1}-\mathrm{reg}}(0) := \max \left( \M_{q}(\Pa) , \M_{\mathrm{Meyers}}(0), \M_{\mathrm{hom}, \delta_0}(0) \right).
\end{equation*}
In the rest of the proof, we assume that the radii $r$ and $R$ are always larger than this minimal scale. We also note that, under the assumption $r \geq \M_{C^{0,1}-\mathrm{reg}}(0) $, we can compare the volume of the ball~$B_r$ and the cardinality of~$\C_\infty \cap B_r$, and we have the estimate
\begin{equation*}
c r^{d} \leq  \left|  \C_\infty \cap B_r \right| \leq C r^d.
\end{equation*}
This is a consequence of the assumption $\M_{C^{0,1}-\mathrm{reg}}(0) \geq \M_{\mathrm{hom}, \delta_0}(0)$ and the estimate~\eqref{compvolumecluster}.

We apply Theorem~\ref{thm:HomoParaPerco} to the function $u$ on the parabolic cylinder $\CQ{r}$, with the boundary condition $f = u$ and with the exponent $\delta_0$; this proves that there exists a function $\bar u \in \Hpar (I_r \times  \cv(\Zd \cap B_r))$, which is a solution of the equation~\eqref{sigkharmfunct.lem3.6}, such that 
\begin{equation} \label{estimate3.8bis}
\frac{1}{r} \norm{u - \bar{u}}_{\aL^2(\CQ{r})} \leq C r^{-\alpha} \norm{ \nabla u}_{\aL^{2 + \delta_0}(\CQ{r})}.
\end{equation}
We split the proof into 4 steps. In the first two steps, we prove that we can apply Lemma~\ref{lemma8.9} with the coarsened function $\left[ u \right]_{\Pa}$, we then post-process the result in Steps 3 and 4 and deduce Proposition~\ref{heatkernelC01}.

\medskip

\textit{Step 1.} In this step, we post-process the result of Theorem~\ref{thm:HomoParaPerco}: in the statement of the estimate~\eqref{estimate3.8bis}, the right-hand side is expressed with an $L^{2 + \delta_0}$-norm, for some small strictly positive exponent $\delta_0$. The goal of this step is to remove this additional assumption. To this end, we use the assumption $r \geq \M_{C^{0,1}-\mathrm{reg}}(0) \geq \M_{\mathrm{Meyers}}(0)$, which implies
\begin{equation*}
\norm{ \nabla u}_{\underline{L}^{2 + \delta_0}(\CQ{r})} \leq C \left\| \nabla u \right\|_{\underline{L}^2 \left( \CQ{2r}\right)}.
\end{equation*}
We then apply the parabolic Caccioppoli inequality, which is stated in Proposition~\ref{p.caccioppoli}, and reads
\begin{equation*}
    \left\| \nabla u \right\|_{\underline{L}^2 \left( \CQ{2r}\right)} \leq \frac Cr \left\| u - \left( u \right)_{\CQ{4r}} \right\|_{\underline{L}^2 \left( \CQ{4r}\right)}.
\end{equation*}
Combining the two previous displays with the inequality~\eqref{estimate3.8bis} shows
\begin{equation} \label{est.step1prop3.5}
     \norm{u - \bar{u}}_{\aL^2(\CQ{r})} \leq C r^{-\alpha} \norm{ u - \left( u \right)_{\CQ{4r}}}_{\aL^{2}(\CQ{4r})},
\end{equation}
and Step 1 is complete.

\medskip

\textit{Step 2.} The goal of this step is to prove that the $L^2$-norm of the difference $\left[ u \right]_{\Pa} - \bar u$ on the continuous parabolic cylinder $I_{r/2} \times B_{r/2}$ is small: we prove that there exists an exponent $\beta > 0$ such that
\begin{equation*}
\left\|  \left[ u \right]_\Pa  - \bar u \right\|_{\underline{L}^2 \left( I_{r/2} \times B_{r/2} \right)} \leq C r^{- \beta} \left\|  \left[ u\right]_\Pa  - \left( \left[u\right]_\Pa \right)_{I_{4r} \times B_{4r}} \right\|_{\underline{L}^2 \left( I_{4r} \times B_{4r} \right)}.
\end{equation*}
The proof of this inequality relies on the estimate~\eqref{est.step1prop3.5}, which establishes that $\bar u$ is a good approximation of $u$ on the infinite cluster, together with the following parabolic regularity result: since $\bar u$ is $\sigk$-caloric on the parabolic cylinder $ I_{r} \times B_{r}$, one has the estimate
\begin{equation} \label{reg.estwint1931}
\norm{\nabla \bar{u}}_{L^{\infty}(I_{r/2} \times B_{r/2})} \leq C r^{-1}\norm{\bar u - \left( \bar u \right)_{I_r \times  \cv(\Zd \cap B_r)}}_{\aL^2(I_r \times  \cv(\Zd \cap B_r))}.
\end{equation}
We then consider consider a (continuous) triadic cube $\cu$ of the partition $\Pa$ such that $\cu \cap B_{r/2} \neq \emptyset$. By definition of the coarsening stated in~\eqref{def.coarsenupar}, one sees that, for each time $t \in I_{r/2}$,
\begin{align*}
    \left\| \left[ u \right]_{\Pa}(t , \cdot) - \bar u(t , \cdot) \right\|_{L^\infty \left( \cu \right)} & \leq \left\| \left[ u \right]_\Pa(t , \cdot) - \left[ \bar u \right]_\Pa(t , \cdot) \right\|_{L^\infty \left( \cu \right)} + \left\| \left[ \bar u \right]_{\Pa}(t , \cdot) - \bar u(t , \cdot)  \right\|_{L^\infty \left( \C_\infty \cap \cu \right)} \\
    & = \left\| \left[ u - \bar u \right]_\Pa(t , \cdot) \right\|_{L^\infty \left( \cu \right)} + \left\| \left[ \bar u \right]_{\Pa}(t , \cdot) - \bar u(t , \cdot)  \right\|_{L^\infty \left( \C_\infty \cap \cu \right)}.
\end{align*}
We then note that, by definition of the coarsening stated in~\eqref{def.coarsenupar}, the $L^\infty$-norm of the function $ \left[ u - \bar u \right]_\Pa$ is smaller than the $L^\infty$-norm of the function $ u - \bar u$. Combining this observation with the estimate~\eqref{eq:estcoarseminusu}, we obtain 
\begin{equation*}
\left\| \left[ u \right]_{\Pa}(t , \cdot) - \bar u(t , \cdot) \right\|_{L^\infty \left( \cu \right)}  \leq \left\| u (t , \cdot) - \bar u(t , \cdot) \right\|_{L^\infty \left( \C_\infty \cap \cu \right)} + \size \left( \cu \right) \left\| \nabla \bar u (t , \cdot) \right\|_{L^\infty(\hat \cu)},
\end{equation*}
where the set $\hat \cu$ stands for the union of the cube $\cu$ and all its neighbors in the partition~$\Pa$. We use the $L^\infty-L^2$ estimate, valid in the discrete setting,
\begin{equation*}
\left\| u (t , \cdot) - \bar u(t , \cdot) \right\|_{L^\infty \left( \C_\infty \cap \cu \right)} \leq \left\| u (t,\cdot) - \bar u (t, \cdot) \right\|_{L^2\left( \C_\infty \cap \cu \right)} ,
\end{equation*}
together with the regularity estimate~\eqref{reg.estwint1931} and the estimate~\eqref{estimateminimalscaleC01} on the sizes of the cubes of the partition $\Pa$ to deduce
\begin{multline*}
\left\|  \left[ u\right]_\Pa  - \bar u \right\|_{L^2 \left( I_{r/2} \times \cu \right)} \leq \size \left( \cu \right)^{d/2} \left\|  u  - \bar u \right\|_{L^2 \left( I_{r/2} \times \left( \C_\infty \cap \cu \right) \right)} \\
+  C \size \left( \cu \right)^{1 + d/2}  r^{-1} \norm{\bar u - \left( \bar u \right)_{I_r \times  \cv(\Zd \cap B_r)}}_{\aL^2(I_r \times  \cv(\Zd \cap B_r))}.    
\end{multline*}
We then use the bounds~\eqref{estimateminimalscaleC01} to estimate the size of the cube $\cu$ and sum over all the cubes of the partition $\Pa$ which intersect the ball $B_{r/2}$, and use the estimate $1 + \frac d2 \leq d$,
\begin{equation*}
\left\|  \left[ u \right]_\Pa  - \bar u  \right\|_{\underline{L}^2 \left( I_{r/2} \times B_{r/2} \right)} \leq C r^{\frac \alpha 2} \left\|  u  - \bar u  \right\|_{\underline{L}^2 \left( \CQ{r} \right)} + C r^{-1 +  \alpha }  \norm{\bar u - \left( \bar u \right)_{I_r \times \cv(\Zd \cap B_r)}}_{\aL^2(I_r \times \cv(\Zd \cap B_r))}.
\end{equation*}
An estimate on the first term on the right side is provided by the estimate~\eqref{est.step1prop3.5}. For the second term on the right-hand side, we apply the Poincar\'e inequality (\cite[Corollary 3.4]{ABM17}), use the estimate~\eqref{eq:ubMeyersStrong} for $f=u$, and apply Proposition~\ref{p.caccioppoli} (the parabolic Caccioppoli inequality) and Proposition~\ref{p.Meyers} (the interior Meyers estimate)
\begin{multline*}
 \norm{\bar u - \left( \bar u \right)_{I_r \times \cv(\Zd \cap B_r)}}_{\aL^2(I_r \times \cv(\Zd \cap B_r))} \leq r \norm{ \nabla \bar u}_{\aL^2(I_r \times \cv(\Zd \cap B_r))} \leq  r \norm{ \nabla u}_{\aL^{2+\delta}(I_r \times (\C_\infty \cap B_r))} \\ \leq Cr \norm{ \nabla u}_{\aL^{2}(I_{2r} \times (\C_\infty \cap B_{2r}))} \leq C \norm{ u  - \left( u \right)_{\CQ{4r} }}_{\aL^{2}(I_{4r} \times (\C_\infty \cap B_{4r}))}.
\end{multline*} 
Thus, we combine the two estimates and obtain
\begin{multline} \label{eq:2154}
\left\|  \left[ u \right]_\Pa  - \bar u  \right\|_{\underline{L}^2 \left( I_{r/2} \times B_{r/2} \right)} \leq r^{\frac \alpha 2} r^{-\alpha}  \left\|  u  - \left( u \right)_{\CQ{4r}} \right\|_{\underline{L}^2 \left( \CQ{4r}\right)} \\+ C r^{-1 + \alpha }   \left\|  u  - \left( u \right)_{\CQ{4r} } \right\|_{\underline{L}^2 \left( \CQ{4r}\right)} .
\end{multline}
We then set the value $\beta :=  \frac \alpha2$. Since, by the identity~\eqref{def.alpha22/03}, the value of the exponent $\alpha$ is smaller than $\frac 12$, we have the inequality $1-\alpha \geq \frac{\alpha}2$. This implies
\begin{equation*}
\left\|  \left[ u \right]_\Pa  - \bar u  \right\|_{\underline{L}^2 \left( I_{r/2} \times B_{r/2} \right)} \leq C r^{- \beta} \left\|  u  - \left( u \right)_{\CQ{4r}} \right\|_{\underline{L}^2 \left( \CQ{4r} \right)}.
\end{equation*}
By definition of the coarsened function $\left[u  \right]_{\Pa}$, we also have
\begin{equation}\label{eq:AverageTrcik}
\begin{split}
\left\|  u  - \left( u \right)_{\CQ{4r}} \right\|_{\underline{L}^2 \left( \CQ{4r} \right)} & \leq \left\|   u - \left( \left[ u \right]_\Pa \right)_{ I_{4r} \times B_{4r}} \right\|_{\underline{L}^2 \left( \CQ{4r} \right)} \\ 
& \leq C \left\|  \left[ u\right]_{\Pa}  - \left( \left[ u \right]_\Pa \right)_{ I_{4r} \times B_{4r}} \right\|_{\underline{L}^2 \left(  I_{4r} \times B_{4r} \right)}.    
\end{split}    
\end{equation}
The proof of Step 2 is complete.

\medskip

\textit{Step 3.} In the two previous steps, we proved that the coarsened function $\left[ u \right]_{\Pa}$ satisfies the assumption of Lemma~\ref{lemma8.9}, with the choice $X = \M_{C^{0,1}-\mathrm{reg}}(0)$. We consequently apply the lemma and obtain that there exists a constant $C := C \left( d , \p , \lambda \right) < \infty$ such that, for every pair of radii $r,R$ satisfying $R \geq r \geq \M_{C^{0,1}-\mathrm{reg}}(0)$,
\begin{equation} \label{C01uQa}
\frac1r \left\|  \left[ u\right]_{\Pa} - \left( \left[ u\right]_{\Pa} \right)_{I_r \times B_r}\right\|_{\underline{L}^2 \left( I_r \times B_r \right)} \leq \frac CR \left\| \left[ u\right]_{\Pa} - \left( \left[ u\right]_{\Pa}\right)_{I_R \times B_R}\right\|_{\underline{L}^2 \left( I_R \times B_R \right)}.
\end{equation} 
Then, we apply once again the estimate \eqref{eq:AverageTrcik} for the left-hand side of \eqref{C01uQa} in $I_r \times  (\C_\infty \cap B_r)$ and we obtain   
\begin{equation} \label{similarto1048}
\frac1r \left\|  u - \left( u \right)_{ I_r \times  (\C_\infty \cap B_r)}\right\|_{\underline{L}^2 \left(  \CQ{r} \right)} \leq \frac CR \left\| \left[ u \right]_\Pa - \left( \left[ u \right]_\Pa \right)_{ I_R \times B_R}\right\|_{\underline{L}^2 \left( I_R \times B_R \right)}.
\end{equation}

\medskip

\textit{Step 4.} In this final step, we upgrade the large-scale $C^{0,1}$-regularity estimate into the estimate~\eqref{eq.heatkernelC01}. The strategy is to use an $L^\infty_t L^2_x$ regularity estimate which is valid for the $\a$-caloric functions since the environment $\a$ is assumed to be time independent.  This result is stated in Lemma~\ref{timeslicelemma} and we apply it to the function $u$ to obtain, for each $r \geq 1$,
\begin{equation*}
\sup_{t \in I_{r/2}} \left\| \nabla u \left(t , \cdot \right) \right\|_{\underline{L}^2 \left( \C_\infty \cap B_{r/2} \right)} \leq \frac{C}{r} \left\| u - \left( u \right)_{\CQ{r}} \right\|_{\underline{L}^2 \left(   I_r \times (\C_\infty  \cap B_r ) \right)}.
\end{equation*}
Combining this result with~\eqref{similarto1048} completes the proof of Proposition~\ref{heatkernelC01} with the radius $\frac r2$ instead of $r$; this is a minor difference which can be fixed by standard arguments.
\end{proof}

\subsubsection{Decay of the gradient of the heat kernel} \label{Greendecaygrad}
The objective of this section is to post-process the regularity theory established in Proposition~\ref{heatkernelC01} and to apply it to the heat kernel. Together with Theorem~\ref{barlow}, we deduce Theorem~\ref{gradbarlowintro}.

\begin{proof}[Proof of Theorem~\ref{gradbarlowintro}]
We fix a time $t \in (0,\infty)$, two points $x,y \in \Zd$ and work on the event $\{y \in \C_\infty \}$. We let $\M_{C^{0,1}-\mathrm{reg}}(x)$ be the minimal scale provided by Proposition~\ref{heatkernelC01} and, for $z \in \Zd$, we let $\T_{\mathrm{NA}}(z)$ be the minimal time provided by Theorem~\ref{barlow}. We first define the minimal time $\mathcal{T}_{\mathrm{NA}}'(x)$ by the formula
\begin{equation} \label{def.Mdens}
\mathcal{T}_{\mathrm{NA}}'(x) := \sup \left\{ t \in [1,\infty)~:~ \exists z \in B_{t}(x)~\mbox{such that} ~ \mathcal{T}_{\mathrm{NA}}(z) \geq t\right\},
\end{equation}
so that for every time $t$ larger than this minimal time, every point $z \in \C_\infty \cap B_{t}(x)$, and every point $z' \in \C_\infty$, one has the estimate
\begin{equation*}
 p(t , z , z') \leq \Phi_C \left( t , z - z' \right).
\end{equation*}
By using the symmetry of the heat kernel, we also have, for each time $t \geq \mathcal{T}_{\mathrm{NA}}'(x)$, for each point $z \in \C_\infty \cap B_{t}(x)$, and each point $z' \in \C_\infty$,
\begin{equation} \label{boundbarlowprime}
     p(t , z' , z) \leq \Phi_C \left( t , z - z' \right).
\end{equation}
Additionally, we claim that this minimal time satisfies the stochastic integrability estimate 
\begin{equation} \label{est.intTNA'}
 \mathcal{T}_{\mathrm{NA}}'(x) \leq \O_s \left( C \right).
\end{equation}
The proof of the estimate~\eqref{est.intTNA'} relies on an application of Lemma~\ref{lemma1.5} by choosing
$$
X_n := \sup_{z \in \Zd \cap B_{3^n}(x)} 3^{-n} \T_{\mathrm{NA}}(z),
$$
and we refer to the computation~\eqref{eq:TL2709} for the details of the argument. We then define the minimal scale
\begin{equation} \label{def.Mreg2680}
    \M_{\mathrm{reg}}(x) := \max \left( \M_{C^{0,1}-\mathrm{reg}}(x) , \sqrt{\T_{\mathrm{NA}}'(x)} \right).
\end{equation}
Using the definition of the $\O_s$ notation, the stochastic integrability estimates~\eqref{minscaleCo1reg} and~\eqref{est.intTNA'}, one has, by reducing the size of the exponent $s$ if necessary, $$ \M_{\mathrm{reg}}(x) \leq \O_s \left( C \right).$$
In particular the tail of the random variable $\M_{\mathrm{reg}}(x)$ satisfies the inequality~\eqref{eq:stochintMell}.
We define, for $\tau \in [-t , \infty)$ and $z \in \C_\infty$,
\begin{equation*}
    u(\tau , z) := p(t+ 1 + \tau , z , y).
\end{equation*}
We let $R := \frac{\sqrt{t}}{2}$ and note that the function $u$ is solution of the parabolic equation on the cylinder $\CQx{R}$. 
Applying Proposition~\ref{heatkernelC01} with the values $r, R$ and using the assumption $R \geq r \geq \M_{\mathrm{reg}}(x)$, we obtain
\begin{equation} \label{almoset1850.06}
    \sup_{\tau \in I_{r}} \left\| \nabla u \left(\tau, \cdot \right) \right\|_{\underline{L}^2 \left( \C_\infty \cap B_r(x) \right)} \leq \frac{C}{t^{1/2}} \left\| \left[ u \right]_\Pa - \left( \left[ u \right]_\Pa \right)_{I_R \times B_R(x)} \right\|_{\underline{L}^2 \left( I_R \times B_R(x) \right)}.
\end{equation}
We then note that the assumption $R \geq \M_{\mathrm{reg}}(x)$ implies $\frac t4 \geq \T_{\mathrm{NA}}'(x)$ and allows to apply the estimate~\eqref{boundbarlowprime} to bound the right side of the previous display. This shows
\begin{equation*}
     \left\| \left[ u \right]_\Pa - \left( \left[ u \right]_\Pa \right)_{I_R \times B_R(x)} \right\|_{\underline{L}^2 \left(  I_R \times B_R(x)\right)} \leq \Phi_C(t , y-x).
\end{equation*}
Combining the previous display with~\eqref{almoset1850.06}, considering the specific value $\tau = -1$, and increasing the value of the constant $C$ shows
\begin{equation*}
    \left\| \nabla p \left(t , \cdot, y \right) \right\|_{\underline{L}^2 \left( \C_\infty \cap B_{r}(x) \right)} \leq t^{-1/2} \Phi_C(t , y-x).
\end{equation*}
The proof of Theorem~\ref{gradbarlowintro} is complete.

\end{proof}

\section{Quantitative homogenization of the heat kernel} \label{section4}

In this section, we establish the theorem of quantitative homogenization of the parabolic Green's function, Theorem~\ref{mainthm}. From now on, we fix a point $y \in \Zd$ and only work on the event $\{ y \in \C_\infty \}$. The proof of this result relies on a two-scale expansion which takes the following form:
\begin{equation} \label{sec42scexp.1605}
    h(t , x , y) := \theta(\p)^{-1} \left( \bar p \left( t , x - y \right) + \sum_{k= 1}^d  \Dr{k} \bar{p}(t,x - y) \chi_{e_k}(x) \right),
\end{equation}
where the correctors $\chi_{e_k}$ are normalized around the point $y$, following the procedure described after Proposition~\ref{prop.sublin.corr}.

As is common for two-scale expansions, the proof relies on two important ingredients:
\begin{itemize}
    \item The sublinearity of the corrector, stated in Proposition~\ref{prop.sublin.corr};
    \item The sublinearity of the flux stated in Proposition~\ref{prop.sublin.flux}.
\end{itemize}
The analysis also requires the estimates on the parabolic Green's function and its gradient provided by Theorem~\ref{barlow} and Theorem~\ref{gradbarlowintro}.
We now present a sketch of the proof of Theorem~\ref{mainthm}. The result can be rewritten by using the $\Phi_C$ notation: for each exponent $\delta > 0 $ and each $x \in \C_\infty$,
\begin{equation} \label{mainresult2229}
\vert p(t, x, y) - \theta(\p)^{-1}\bar{p}(t, x - y) \vert \leq t^{-\frac{1}{2} + \delta} \Phi_{C}(t, x-y), \qquad \sqrt{t} \geq \T_{\mathrm{par}, \delta}(y).
\end{equation}
To prove this result, we first prove a weighted $L^2$-estimate (see \eqref{def.psiC} for the definition of $\Psi_C$)
\begin{equation} \label{mainresult2230}
\left\| \left( p(t, \cdot , y) - \theta(\p)^{-1}\bar{p}(t, \cdot - y  ) \right) \exp\left(\Psi_C(t, \vert \cdot - y\vert)\right) \right\|_{L^2(\C_{\infty})} \leq C t^{-\frac d4 - \frac 12 + \delta},
\end{equation}
and deduce the estimate~\eqref{mainresult2229} thanks to the semigroup property, this is proved in Section~\ref{section4.3}. Proving the estimate~\eqref{mainresult2230} is the core of the proof. To this end, we need to introduce a mesoscopic time $1 \ll \tau \ll t$ and a number of intermediate functions which are listed below:
\begin{itemize}
    \item The two-scale expansion $h$ defined in~\eqref{sec42scexp.1605}; 
    \item The function $q:= q(\cdot, \cdot, \tau , y)$ introduced in~\eqref{thefunctionq};
    \item The function $v:= v(\cdot, \cdot, \tau , y)$ introduced in~\eqref{def.v.para16.36};
    \item The function $w$ defined by the formula  $w := h - v - q$. 
\end{itemize}
The idea is to use these functions to split the difference
\begin{align*}
    p(t, x,y) - \theta(\p)^{-1}\bar{p}(t, x - y) & = \left( p(t, x, y) - q(t, x, \tau, y) \right)
 - v(t,x,\tau,y) \\ & \quad
 - w(t, x, \tau, y) + \left( h(t,x,y) - \theta(\p)^{-1} \bar{p}(t, x - y)) \right),
\end{align*}
and then to prove that the $L^2$-norm of each of the terms is smaller than $t^{-\frac d4 - \frac 12 + \delta}$. More specifically, we organize the proof as follows:
\begin{enumerate}[label=(\roman*)]
    \item in Lemma~\ref{l.lemma4.2}, we prove that the term corresponding to the difference $(p - q)$ is small;
    \item in Lemma~\ref{lem:v}, we prove that the term corresponding to the function $v$ is small;
    \item in Proposition~\ref{lem:w}, we prove that the term corresponding to the function $w$ is small;
    \item the term $\left( h(t,x,y) - \theta(\p)^{-1} \bar{p}(t, x - y)) \right)$ is proved to be small by using the sublinearity of the corrector, the proof is straightforward and not stated in a specific lemma.
\end{enumerate}
The rest of this section is organized as follows. Section~\ref{section4.1} is devoted to the proof of Lemmas~\ref{l.lemma4.2} and~\ref{lem:v}. Section~\ref{section4.2} is devoted to the proof of Proposition~\ref{lem:w} and is the core of the analysis: we make use of the regularization Lemmas proved in Section~\ref{section4.1} as well as the various results recorded in the previous sections to perform the two-scale expansion. In Section~\ref{section4.3}, we post-process the results and prove the quantitative convergence of the heat kernel, Theorem~\ref{mainthm}.

\subsection{Two regularization steps} \label{section4.1}
We now introduce the function $q$. For a fixed initial time $\tau > 0$ and a vertex $y \in \C_\infty$, we let $(t ,x) \mapsto q(t , x, \tau , y)$ be the solution of the parabolic problem
\begin{equation} \label{thefunctionq}
    \left\{ \begin{array}{lcl}
    \partial_t q - \nabla \cdot \left(  \a \nabla q \right) = 0 &\mbox{in} ~(\tau , \infty) \times \C_\infty, \\
    q(\tau , \cdot, \tau, y) = \theta \left( \p\right)^{-1}\bar p(\tau , \cdot - y) &\mbox{on}~\C_\infty.
    \end{array} \right.
\end{equation}
A reason justifying this construction is that the initial condition of the heat kernel $p$, which is a Dirac at $y$, is too singular and one cannot perform the two-scale expansion due to this lack of regularity. The idea is thus to replace the Dirac by a smoother function, the function $\bar p(\tau , \cdot)$, and to exploit its more favorable regularity properties to perform the two-scale expansion (see Section~\ref{section4.2}). For this strategy to work, one needs to choose the value of the time $\tau$ to be both:
\begin{itemize}
    \item Large enough so that so that the function $\bar p(\tau , \cdot)$ is regular enough: in particular we want $\tau \gg 1$;
    \item Small enough so that the function $q$ is a good approximation of the heat kernel $p$: we want $\tau \ll t$. 
\end{itemize}
The choice of $\tau$ will be $\tau := t^{1 - \kappa}$, for some small exponent $\kappa$ whose value is decided at the end of the proof.

The following lemma states that the function $q$ is a good approximation of the heat kernel $p$, when the coefficient $\tau$ is chosen such that $1 \ll \tau \ll t$. In the following lemma, given an exponent $\alpha >0$, we use the notation $\mathcal{T}_{\mathrm{dense}, \alpha}(y)$ to denote the minimal time introduced in Proposition~\ref{prop:DensityContrentrationGaussian}, above which the mass of the homogenized heat kernel $\bar p$ is almost equal to the density of the infinite cluster $\C_\infty$ up to an error of order $t^{-\frac 12 + \alpha}$.
    
\begin{lemma} \label{l.lemma4.2}
For each exponent $\alpha > 0$ and each vertex $y \in \Zd$, we let $ \T_{\mathrm{approx}, \alpha}(y)$ be the minimal time defined by the formula
\begin{equation*}
    \T_{\mathrm{approx},\alpha}(y) := \max \left( \M_{\mathrm{reg}}(y)^2, \T_{\mathrm{dense}, \alpha}(y) \right).
\end{equation*}
This random variable satisfies the stochastic integrability estimate
\begin{equation*}
    \T_{\mathrm{approx}, \alpha}(y) \leq \O_s \left(C \right),
\end{equation*}
and the following property: there exists a positive constant $C:= C(d , \p , \lambda, \alpha) < \infty$ such that, on the event $\{y \in \C_\infty\}$, for every pair of times $t, \tau \in (0 , \infty)$ such that $t \geq 3 \tau$ and $\tau \geq \T_{\mathrm{approx}, \alpha}(y)$, and for every $x \in \C_\infty$, one has
\begin{equation} \label{e.mainestLem4.2}
    \left|q(t , x , \tau , y) - p \left( t , x , y \right) \right| \leq \left( \left( \frac \tau t \right)^{\frac{1}{2}}+ \tau^{-\frac 12 + \alpha} \right) \Phi_C(t, x-y ).
\end{equation}
\end{lemma}

\begin{proof}
Before starting the proof, we note that the assumptions of the lemma imply the following results:
\begin{itemize}
    \item The two inequalities $t \geq 3\tau$ and $\tau \geq \mathcal{T}_{\mathrm{approx}, \alpha}(y)$ imply the estimates $t \geq \mathcal{T}_{\mathrm{approx}, \alpha}(y)$ and $t - \tau \geq \mathcal{T}_{\mathrm{approx}, \alpha}(y)$;
    \item By definition, the minimal time $\mathcal{T}_{\mathrm{approx}, \alpha}(y)$ is larger than the minimal times $\T_{\mathrm{dense}, \alpha}(y)$, $\mathcal{T}_{NA}'(y)$ and the square of the minimal scales $\M_{\mathrm{Poinc}}(y)$ and $\M_{\mathrm{reg}}(y)$. We can thus apply the corresponding results in the proof.
\end{itemize}

\textit{Step 1: Set up.} We fix a vertex $ y \in \Zd$ and work on the event $\{ y \in \C_\infty \}$. We first record the following estimate: under the assumption $ \T_{\mathrm{dense}, \alpha}(y) $, for each radius $r \geq \sqrt \tau$, we can compare the volume of the ball~$B_r(y)$ and the cardinality of the set~$\C_\infty \cap B_r(y)$, and we have the estimate
\begin{equation*}
c r^{d} \leq  \left|  \C_\infty \cap B_r(y) \right| \leq C r^d.
\end{equation*}
We consider a point $x$ in the infinite cluster $\C_\infty$ and two times $t,\tau > 0$ such that $t \geq \tau \geq \T_{\mathrm{approx}, \alpha}(y)$. By Duhamel's principle, one has
\begin{equation} \label{Duhamelapprox}
    q(t,x , \tau , y) - p(t , x , y) = \int_{\C_\infty} \left( \theta(\p)^{-1}\bar p \left( \tau, z-y \right) - p(\tau , z , y) \right) p(t-\tau , x , z) \, dz.
\end{equation}
Using the inequality $\tau  \geq \T_{\mathrm{dense}, \alpha}(y)$ and Proposition~\ref{prop:DensityContrentrationGaussian}, we have the inequality
\begin{equation*}
   \left| \theta(\p)^{-1} \int_{\C_\infty} \bar p \left( \tau, z-y \right) \, dz - 1 \right| \leq C \tau^{-\frac 12+ \alpha}.
\end{equation*}
 Since the mass of the transition kernel $p(\tau , \cdot , y)$ on the infinite cluster is equal to $1$, we deduce that
\begin{equation*}
        \left| \int_{\C_\infty} \left(\theta(\p)^{-1} \bar p \left( \tau, z-y \right) - p(\tau , z , y) \right) \, dz \right| \leq C \tau^{-\frac 12 + \alpha}.
\end{equation*}
We can thus subtract a constant term equal to $\left( p(\tau , x , \cdot) \right)_{\C_\infty \cap B_{\sqrt{\tau}}(y)}$ in the right side of~\eqref{Duhamelapprox} up to a small cost of order $\tau^{-\frac 12 + \alpha}$,
\begin{align*}
    \lefteqn{\left| q(t,x , \tau , y) - p(t , x , y) \right| } \qquad & \\ & \leq \left| \int_{\C_\infty} \left(\theta(\p)^{-1} \bar p \left( \tau, z-y \right) - p(\tau , z , y) \right) \left( p(t-\tau , x , z) - \left( p(t-\tau , x , \cdot) \right)_{\C_\infty \cap B_{\sqrt{\tau}}(y)} \right) \, dz \right| \\ & \quad + C \tau^{-\frac 12 + \alpha}  \left| \left( p(t-\tau , x , \cdot) \right)_{\C_\infty \cap B_{\sqrt{\tau}}(y)} \right|.
\end{align*}
Using the inequality $t - \tau \geq \T_{\mathrm{NA}}'(y)$ and \eqref{def.Mdens}, we apply Theorem~\ref{barlow} to obtain
\begin{multline} \label{estqp14.01}
    \left| q(t,x , \tau , y) - p(t , x , y) \right| \\ \leq \int_{\C_\infty} \Phi_C (\tau , z-y) \left|  p(t -\tau , x , z) - \left( p(t-\tau , x , \cdot) \right)_{\C_\infty \cap B_{\sqrt{\tau}}(y)} \right| \, dz + \tau^{-\frac 12 + \alpha}  \Phi_C (t , x-y).
\end{multline}
We then treat the first term on the right side of \eqref{estqp14.01}. The strategy is to split the integral into scales: for each integer $n \geq 1$, we let $A_n$ be the dyadic annulus $A_n := \left\{ z \in \Zd ~:~ 2^{n} \sqrt{\tau} \leq |z-y| < 2^{n+1} \sqrt{\tau} \right\}$ and compute
\begin{align} \label{splittingscales.tech.1956}
\lefteqn{\int_{\C_\infty} \Phi_C (\tau , z-y) \left|  p(t -\tau , x , z) - \left( p(t-\tau , x , \cdot) \right)_{\C_\infty \cap B_{\sqrt{\tau}}(y)} \right| \, dz  \qquad} & \\ & = \int_{\C_\infty \cap B_{\sqrt{\tau}}(y)} \Phi_C (\tau , z-y) \left|  p(t -\tau , x , z) - \left( p(t-\tau , x , \cdot) \right)_{\C_\infty \cap B_{\sqrt{\tau}}(y)} \right| \, dz \notag \\ & \quad 
+ \sum_{n = 0}^\infty \int_{\C_\infty \cap A_n } \Phi_C (\tau , z-y) \left|  p(t -\tau , x , z) - \left( p(t-\tau , x , \cdot) \right)_{\C_\infty \cap B_{\sqrt{\tau}}(y)} \right| \, dz \notag.
\end{align}

\textit{Step 2: Multiscale analysis in the ball $B_{\sqrt{\tau}}(y)$.} 
The term pertaining to small scales $B_{\sqrt{\tau}}(y)$ can be estimated thanks to the estimate $ \Phi_C (\tau , z-y) \leq C \tau^{-d/2}$ and the Poincar\'e inequality. This latter inequality can be applied since we assumed $\tau \geq \M_{\mathrm{Poinc}}(y)^2$. This gives
\begin{align} \label{Poinca9:34}
\lefteqn{\int_{\C_\infty \cap B_{\sqrt{\tau}}(y)} \Phi_C (\tau , z-y) \left|  p(t -\tau , x , z) - \left( p(t-\tau , x , \cdot) \right)_{\C_\infty \cap B_{\sqrt{\tau}}(y)} \right| \, dz } \qquad & \\ & \hspace{10mm}\leq C  \left\|  p(t -\tau , x , \cdot ) - \left( p(t-\tau , x , \cdot) \right)_{\C_\infty \cap B_{\sqrt{\tau}}(y)} \right\|_{\underline{L}^2 \left( \C_\infty \cap B_{\sqrt{\tau}}(y) \right)} \notag\\
        & \hspace{10mm} \leq C \sqrt{\tau} \left\|  \nabla_y p(t -\tau , x , \cdot) \right\|_{\underline{L}^2 \left( \C_\infty \cap B_{\sqrt{\tau}}(y) \right)}, \notag
\end{align}
where the notation $\nabla_y$ means that the gradient is on the second spatial variable.
We now estimate the term on the right side thanks to Theorem~\ref{gradbarlowintro}, or more precisely Remark~\ref{remark1.3}, which can be applied since we assumed $t -\tau \geq \M_{\mathrm{reg}}(y)^2$. We have
\begin{align*}
    \left\|  \nabla p(t -\tau , x , \cdot) \right\|_{\underline{L}^2 \left( \C_\infty \cap B_{\sqrt{\tau}}(y) \right)} &\leq C (t-\tau)^{-1/2} \Phi_C \left( t-\tau , x - y \right) \\
           & \leq  C t^{-1/2} \Phi_C \left( t , x - y\right),
\end{align*}
where to go from the first line to the second one we used that $t-\tau \geq \frac 23 t$ and increased the value of the constant $C$.

A combination of the two previous displays shows
\begin{equation} \label{e.13.46estsmallscale}
    \int_{\C_\infty \cap B_{\sqrt{\tau}}(y)} \Phi_C (\tau , z-y) \left|  p(t -\tau , x , z) - \left( p(t-\tau , x , \cdot) \right)_{\C_\infty \cap B_{\sqrt{\tau}}(y)} \right| \, dz \leq  C \left(\frac{\tau}{t}\right)^{1/2} \Phi_C \left( t , x -y \right).
\end{equation}
This completes the estimate of the term corresponding to the small scales in~\eqref{splittingscales.tech.1956}.

\smallskip

\textit{Step 3: Multiscale analysis in the annuli $A_n$.} To estimate the terms corresponding to the dyadic annuli, we fix some integer $n \in \N$ and study the integral in the region $A_n$; thanks to the triangle inequality, we insert a constant term equal to $\left( p(t-\tau , x , \cdot) \right)_{\C_\infty \cap B_{2^{n+1}\sqrt{\tau}}(y)}$ in the integral,
\begin{align} \label{10:57splitest}
   \lefteqn{ \int_{\C_\infty \cap A_n } \Phi_C (\tau , z-y) \left|  p(t -\tau , x , z) - \left( p(t-\tau , x , \cdot) \right)_{\C_\infty \cap B_{\sqrt{\tau}}(y)} \right| \, dz } \qquad \notag & \\ &
                    \leq  \underbrace{\int_{\C_\infty \cap A_n } \Phi_C (\tau , z-y) \left|  p(t -\tau , x , z) - \left( p(t-\tau , x , \cdot) \right)_{\C_\infty \cap B_{2^{n+1}\sqrt{\tau}}(y)} \right| \, dz}_{\text{\eqref{10:57splitest}}-a} \\ & 
                    \quad + \underbrace{\frac{\left( \int_{\C_\infty \cap A_n } \Phi_C (\tau , z-y) \, dz \right)}{\left|\C_\infty \cap B_{\sqrt{\tau}}(y) \right|} \int_{\C_\infty \cap B_{\sqrt{\tau}}(y)} \left|  p(t -\tau , x , z) - \left( p(t-\tau , x , \cdot) \right)_{\C_\infty \cap B_{2^{n+1}\sqrt{\tau}}(y)} \right| \, dz}_{\text{\eqref{10:57splitest}}-b}. \notag
\end{align}
\textit{Step 3.1: Estimate for the term \eqref{10:57splitest}-a in the annuli $A_n$.} 
We estimate \eqref{10:57splitest}-a and distinguish three types of scales:
\begin{enumerate}[label=(\roman*)]
    \item The small scales which are defined as the annuli $A_n$ such that $2^{n+2} \sqrt{\tau} \leq \sqrt{t}$;
    \item The intermediate scales which are defined as the annuli $A_n$ such that $t \geq 2^{n+2} \sqrt{\tau} > \sqrt{t}$;
    \item The large scales which are defined as the annuli $A_n$ such that $2^{n+2} \sqrt{\tau} > t$.
\end{enumerate}
The following estimate for the function $\Phi_C$ is easy to check by its definition \eqref{def.phiC} and is used several times in the proof: for any constant $C' > C$, there exists a constant $c>0$ depending only on the values of $C$ and $C'$ such that
\begin{equation*}
    \left( \sup_{\C_\infty \cap A_n}  \Phi_C (\tau , \cdot - y)  \right) \leq e^{-c 2^n \sqrt{\tau}} \left( \inf_{\C_\infty \cap A_n}  \Phi_{C'} (\tau , \cdot - y)  \right).
\end{equation*}
This implies that for any positive integer $k \in \N$, there exists a constant $C_k$ depending only on the integer $k$ and the constant $C$ such that
\begin{equation}\label{eq:PhiCTrick}
\left( \sup_{\C_\infty \cap A_n}  \Phi_C (\tau , \cdot - y)  \right) \leq \left(2^n \sqrt{\tau}\right)^{-k}\left( \inf_{\C_\infty \cap A_n}  \Phi_{C_k} (\tau , \cdot - y)  \right).         
\end{equation}

\textit{Step 3.1.1: Estimate for \eqref{10:57splitest}-a in the small scales $2^{n+2} \sqrt{\tau} \leq \sqrt{t}$.} We first focus on the small scales and apply the Poincar\'e inequality. With a computation similar to~\eqref{Poinca9:34}, one can estimate the first term in the right side of~\eqref{10:57splitest}
\begin{align*}
    \lefteqn{ \int_{\C_\infty \cap A_n } \Phi_C (\tau , z-y) \left|  p(t -\tau , x , z) - \left( p(t-\tau , x , \cdot) \right)_{\C_\infty \cap B_{2^{n+1}\sqrt{\tau}}} \right| \, dz} \qquad & \\ & \leq \left( \sup_{\C_\infty \cap A_n}  \Phi_C (\tau , \cdot - y)  \right) \int_{\C_\infty \cap B_{2^{n+1}\sqrt{\tau}}(y) } \left|  p(t -\tau , x , z) - \left( p(t-\tau , x , \cdot) \right)_{\C_\infty \cap B_{2^{n+1}\sqrt{\tau}}(y)} \right| \, dz \\ &
            \leq \left( \sup_{\C_\infty \cap A_n}  \Phi_C (\tau , \cdot - y)  \right) C (2^{n+1} \sqrt{\tau})^{\frac d2 + 1} \left\|  \nabla_y p(t -\tau , x , \cdot ) \right\|_{L^2 \left( \C_\infty \cap B_{ 2^{n+1}\sqrt{\tau}}(y) \} \right)}.
\end{align*}
Using the assumption $t - \tau \geq \M_{\mathrm{reg}}^2(y)$, we apply Theorem~\ref{gradbarlowintro}. This shows
\begin{equation*}
    \left\|  \nabla_y p(t -\tau , x , \cdot) \right\|_{L^2 \left( \C_\infty \cap B_{2^{n+1}\sqrt{\tau}}(y) \right)} \leq C t^{-1/2} \left(2^{n+1}\sqrt{\tau}\right)^\frac d2 \Phi_C(t ,x- y).
\end{equation*}
Using the explicit formula for the function $\Phi_C$ stated in~\eqref{def.phiC}, one has the estimate
\begin{equation*}
    \left( \sup_{\C_\infty \cap A_n}  \Phi_C (\tau , \cdot - y)  \right) (2^{n+1} \sqrt{\tau})^{d} \leq C 2^{-2 n}.
\end{equation*}
Combining the three previous displays shows, for each integer $n \in \N$ such that $2^{n+2} \sqrt{\tau} \leq \sqrt{t}$,
\begin{equation} \label{sum.scale13.36}
    \int_{\C_\infty \cap A_n } \Phi_C (\tau , z-y) \left|  p(t -\tau , x , z) - \left( p(t-\tau , x , \cdot) \right)_{\C_\infty \cap B_{2^{n+1}\sqrt{\tau}}} \right| \, dz \leq C 2^{-n} \left( \frac{\tau}{t} \right)^{\frac12} \Phi_C(t , x-y).
\end{equation}

\textit{Step 3.1.2: Estimate for \eqref{10:57splitest}-a in the intermediate scales $t \geq 2^{n+2} \sqrt{\tau} > \sqrt{t}$.} We now treat the case of the intermediate scales. In this case, we have $2^{n+2}  \geq \sqrt{\frac{t}{\tau}}$, thus by the estimate \eqref{eq:PhiCTrick} for $k=2$, one has
\begin{equation} \label{eq:fri1449}
    \left( \sup_{\C_\infty \cap A_n}  \Phi_C (\tau , \cdot - y)  \right) \leq  2^{- n} \left( \frac{\tau}{t} \right)^{\frac 12}  \left( \inf_{\C_\infty \cap A_n}  \Phi_{ C_2} (\tau , \cdot - y)  \right).
\end{equation}
Using the assumptions $t - \tau \geq \T_{\mathrm{NA}}'(y)$, we can apply Theorem~\ref{barlow}, the previous estimate~\eqref{eq:fri1449} and the convolution property~\eqref{eq:convprop} for the map $\Phi_C$. We obtain
\begin{align} \label{eq:genest2042}
    & \lefteqn{\left( \sup_{\C_\infty \cap A_n}  \Phi_C (\tau , \cdot - y)  \right) \int_{\C_\infty \cap B_{2^{n+1}\sqrt{\tau}}(y) } \left|  p(t -\tau , x , z) - \left( p(t-\tau , x , \cdot) \right)_{\C_\infty \cap B_{2^{n+1}\sqrt{\tau}}(y)} \right| \, dz} \\ \leq  & 2^{- n} \left( \frac{\tau}{t} \right)^\frac 12  \left( \inf_{\C_\infty \cap A_n}  \Phi_{C_2} (\tau , \cdot - y)  \right)  \int_{\C_\infty \cap B_{2^{n+1}\sqrt{\tau}}(y) } \Phi_C \left( t - \tau , x - z \right) \, dz \notag \\ 
  \leq & 2^{- n} \left( \frac{\tau}{t} \right)^\frac 12     \int_{\C_\infty \cap B_{2^{n+1}\sqrt{\tau}}(y) }  \Phi_{C_2} (\tau , y -z) \Phi_C \left( t - \tau , x - z \right) \, dz \notag \\
  \leq & 2^{- n} \left( \frac{\tau}{t} \right)^\frac 12 \Phi_C \left(t, x-y  \right), \notag
\end{align}
by increasing the value of the constant $C$ and using \eqref{eq:convprop} in the last line.

\smallskip

\textit{Step 3.1.3: Estimate for the term \eqref{10:57splitest}-a in the large scales $2^{n+2} \sqrt{\tau} > t$.} The computation is similar to the one performed in~\eqref{eq:genest2042} up to two differences listed below:
\begin{enumerate}[label=(\roman*)]
    \item For these scales, we cannot apply the Gaussian bounds on the heat kernel given by Theorem~\ref{barlow}. Instead, we thus apply the Carne-Varopoulos bound which is stated in Proposition~\ref{proposition2.11} and can be rewritten with the notation $\Phi_C$: for each $x , z \in \C_\infty$,
    \begin{equation*}
        p(t,x,z) \leq t^{\frac d2} \Phi_C \left(t , x-z\right);
    \end{equation*}
    \item We use the inequality $2^{n+2} \sqrt{\tau} > t$ and the estimate \eqref{eq:PhiCTrick} for $k = \frac{d}{2} + 2$ to obtain the bound in the annulus $A_n$
    \begin{equation*}
    \left( \sup_{\C_\infty \cap A_n}  \Phi_C (\tau , \cdot - y)  \right) \leq  2^{- n} t^{-\frac d2} \left( \frac{\tau}{t} \right)^{\frac 12}  \left( \inf_{\C_\infty \cap A_n}  \Phi_{ C'} (\tau , \cdot - y)  \right),
    \end{equation*}
    for some constant $C' > C$.
\end{enumerate}
We can then perform the computation~\eqref{eq:genest2042} and obtain the estimate
\begin{equation} \label{eq:Mun1840}
    \int_{\C_\infty \cap A_n } \Phi_C (\tau , z-y) \left|  p(t -\tau , x , z) - \left( p(t-\tau , x , \cdot) \right)_{B_{2^{n+1}\sqrt{\tau}}(y)} \right| \, dz \leq C 2^{-n} \left( \frac{\tau}{t} \right)^{\frac12} \Phi_C(t , x-y).
\end{equation}
Combining the estimates~\eqref{sum.scale13.36},~\eqref{eq:genest2042} and~\eqref{eq:Mun1840}, we have obtained, for each integer $n \in \N,$
\begin{equation} \label{eq:tra2117}
\int_{\C_\infty \cap A_n } \Phi_C (\tau , z-y) \left|  p(t -\tau , x , z) - \left( p(t-\tau , x , \cdot) \right)_{B_{2^{n+1}\sqrt{\tau}}(y)} \right| \, dz \leq C 2^{-n} \left( \frac{\tau}{t} \right)^{\frac12} \Phi_C(t , x-y).
\end{equation}

\textit{Step 3.2: Estimate for \eqref{10:57splitest}-b in the annuli $A_n$.} The second term \eqref{10:57splitest}-b can be estimated thanks to a similar strategy: we first apply the inequality~\eqref{eq:PhiCTrick} with $k=d$
\begin{align*}
 \frac{\left( \int_{\C_\infty \cap A_n } \Phi_C (\tau , z-y) \, dz \right)}{\left|\C_\infty \cap B_{\sqrt{\tau}}(y) \right|} & \leq \frac{\left|\C_\infty \cap A_n \right| }{\left|\C_\infty \cap B_{\sqrt{\tau}}(y) \right|} \sup_{\C_\infty \cap A_n} \Phi_C (\tau , \cdot - y) \\
& \leq 2^{dn} \sup_{\C_\infty \cap A_n} \Phi_C (\tau , \cdot - y) \\
& \leq  \inf_{\C_\infty \cap A_n} \Phi_{C'} (\tau , \cdot - y),
\end{align*}
for some constant $C' > C$. Using this estimate, we deduce
\begin{align*}
\lefteqn{\frac{\left( \int_{\C_\infty \cap A_n } \Phi_C (\tau , z - y) \, dz \right)}{\left|\C_\infty \cap B_{\sqrt{\tau}}(y) \right|} \int_{\C_\infty \cap B_{\sqrt{\tau}}(y)} \left|  p(t -\tau , x , z) - \left( p(t-\tau , x , \cdot) \right)_{\C_\infty \cap B_{2^{n+1}\sqrt{\tau}}(y)} \right| \, dz} \qquad & \\ &  \leq  \inf_{\C_\infty \cap A_n} \Phi_{C'} (\tau , \cdot - y)  \int_{\C_\infty \cap  B_{2^{n+1}\sqrt{\tau}}(y) } \left|  p(t -\tau , x , z) - \left( p(t-\tau , x , \cdot) \right)_{\C_\infty \cap B_{2^{n+1}\sqrt{\tau}}(y)} \right| \, dz.
\end{align*}
We can then apply the same proof as for the first term in the right side of~\eqref{10:57splitest}. This proves the inequality
\begin{multline*}
\frac{\left( \int_{\C_\infty \cap A_n } \Phi_C (\tau , z - y) \, dz \right)}{\left|\C_\infty \cap B_{\sqrt{\tau}}(y) \right|} \int_{\C_\infty \cap B_{\sqrt{\tau}}(y)} \left|  p(t -\tau , x , z) - \left( p(t-\tau , x , \cdot) \right)_{\C_\infty \cap B_{2^{n+1}\sqrt{\tau}}(y)} \right| \, dz \\ \leq  C 2^{-n} \left( \frac{\tau}{t} \right)^{\frac12} \Phi_C(t , x-y).
\end{multline*}
Combining the previous inequality with the estimate~\eqref{eq:tra2117}, we obtain, for each integer $n \in \N$,
\begin{equation*}
 \int_{\C_\infty \cap A_n } \Phi_C (\tau , z-y) \left|  p(t -\tau , x , z) - \left( p(t-\tau , x , \cdot) \right)_{\C_\infty \cap B_{\sqrt{\tau}}(y)} \right| \, dz \leq C 2^{-n} \left( \frac{\tau}{t} \right)^{\frac12} \Phi_C(t , x-y).
\end{equation*}
Combining this inequality with~\eqref{estqp14.01}, \eqref{splittingscales.tech.1956}, \eqref{e.13.46estsmallscale}, and summing over the integer $n \in \N$ completes the proof of Lemma~\ref{l.lemma4.2}.
\end{proof}

We now introduce a second intermediate function useful in the proof of Theorem~\ref{mainthm}, the function $v$ which is defined as follows. For some fixed $(\tau,y) \in (0 , \infty) \times \C_\infty$, we let $(t,x) \mapsto v(t,x, \tau ,y)$ be the solution of the parabolic equation
\begin{equation} \label{def.v.para16.36}
    \left\{ \begin{aligned}
    & \partial_t v (\cdot,\cdot, \tau , y) - \nabla \cdot \a \nabla v (\cdot,\cdot, \tau , y) = 0 &~ \mbox{in}~(\tau , \infty) \times \C_\infty,   \\
     & v (\tau , \cdot , \tau , y) = h(\tau , \cdot , y) - \theta\left( \p \right)^{-1} \bar p (\tau , \cdot - y) & ~\mbox{in}~ \C_\infty.
    \end{aligned} \right.
\end{equation}
To define this function we run the parabolic equation starting from time $\tau$ and until time $t$ with the initial condition given by the difference between the two-scale expansion $h$ defined in~\eqref{sec42scexp.1605} and the homogenized heat kernel $\theta\left( \p \right)^{-1} \bar p (\tau , \cdot - y)$. By the sublinearity of the corrector stated in Proposition~\ref{prop.sublin.corr}, we expect the function $ h(\tau , \cdot , y) - \theta\left( \p \right)^{-1} \bar p (\tau , \cdot - y)$ to be small. The following proposition states that the solution of the parabolic equation with this initial condition remains small (in the sense of the inequality~\eqref{e:lem:subv}).

\begin{lemma}\label{lem:v}
For any exponent $\alpha > 0$, there exists a positive constant $C:= C(d , \lambda, \p, \alpha) < \infty$ such that for each pair of times $t , \tau \in (0 , \infty)$ satisfying $t \geq 3\tau$, $(t- \tau) \geq \T_{\mathrm{NA}}'(y)$, and $\sqrt{\tau} \geq  \M_{\mathrm{corr}, \alpha}(y)$, the following estimate holds
\begin{equation} \label{e:lem:subv}
    \left|v (t , x , \tau , y) \right| \leq C \tau^{- \frac 12 + \frac \alpha2} \Phi_C(t , x-y).
\end{equation}
\end{lemma}

\begin{proof}
The proof relies on two main ingredients: the quantitative sublinearity of the corrector and an explicit formula for the function $v$ in terms of the heat kernel $p$.

First, by the definition~\eqref{def.v.para16.36}, we have the formula
\begin{align*}
    v(t , x , \tau , y) & = \int_{\C_\infty} \left( h(\tau , z , y) - \theta(\p)^{-1}\bar p (\tau , z- y) \right) p(t-\tau , x , z ) \, dz \\
                    & = \int_{\C_\infty} \theta(\p)^{-1}\left( \sum_{k= 1}^d  \Dr{k} \bar{p}(\tau,z - y) \chi_{e_k}(z) \right) p(t-\tau , x , z ) \, dz.
\end{align*}
We then apply the four following estimates:
\begin{enumerate}[label=(\roman*)]
\item The sublinearity of the corrector: under the assumption $\sqrt{\tau} \geq \M_{\mathrm{corr}, \alpha}(y)$ and the normalization convention chosen for the corrector in the definition of the two-scale expansion $h$ stated in~\eqref{sec42scexp.1605}, one has, for each $k \in \{ 1 , \ldots , d \}$,
\begin{equation*}
    \left|\chi_{e_k}(z) \right| \leq \left\{ \begin{aligned}
    &\tau^{\frac \alpha2} &~\mbox{if}~|z -y| \leq \tau^\frac12, \\
    &|z - y|^{\alpha} &~\mbox{if}~|z-y| \geq \tau^\frac12;
    \end{aligned} \right.
\end{equation*}
\item The Gaussian bounds on the transition kernel $p(t-\tau , x , \cdot )$, valid under the assumptions $\tau \geq \T_{\mathrm{NA}}'(y)$ and $t \geq 3\tau$: for each $z \in \C_\infty \cap B_{t - \tau} (y)$,
\begin{equation*}
p(t-\tau , x , z ) \leq \Phi_C\left( t-\tau, x-z \right) \leq \Phi_{C'}\left( t, x-z \right),
\end{equation*}
where the second inequality follows from the inequality $t - \tau \geq \frac 23 t$ (by increasing the value of the constant $C$);
\item The bound on the transition kernel $p \left( t-\tau , x , z\right)$: for any point $z \in \C_\infty \setminus B_{t - \tau} (y)$,
\begin{equation*}
    p(t-\tau , x , z ) \leq (t-\tau)^{d/2} \Phi_C\left( t-\tau , z - x \right),
\end{equation*}
which is a consequence of the definition of the map $\Phi_C$ stated in~\eqref{def.phiC}, Proposition~\ref{proposition2.11} and the assumption $t \geq 3 \tau$ (by increasing the value of the constant $C$);
\item The estimate on the homogenized heat kernel $\bar p$, which follows from standard results from the regularity theory,
\begin{equation*}
   \left| \Dr{k} \bar{p}(\tau,z - y) \right| \leq C \tau^{-\frac 12} \Phi_C(\tau, z - y).
\end{equation*}
\end{enumerate}
We obtain the inequality
\begin{equation}\label{eq:vBound}
\begin{split}
    \left| v(t , x , \tau  , y) \right| & \leq C \tau^{-\frac12 + \frac\alpha2} \int_{\C_\infty \cap B_{\sqrt{\tau}}\left(y\right)}  \Phi_C\left( \tau, z-y \right) \Phi_C\left( t-\tau, x-z \right) \, dz \\ & \quad +  C \tau^{- \frac 12}\int_{\C_\infty \cap  \left( B_{t-\tau}(y) \setminus B_{\sqrt{\tau}}\left(y\right)\right)} |z - y |^{\alpha} \Phi_C\left( \tau, z-y \right) \Phi_C\left( t-\tau, x-z \right) \, dz \\
            & \quad + C \tau^{-\frac 12} (t-\tau)^{d/2} \int_{\C_\infty \setminus B_{t-\tau}(y)} |z - y|^{\alpha} \Phi_C\left( \tau, z-y \right) \Phi_C\left( t-\tau, x-z \right) \, dz.    
\end{split}    
\end{equation}
We then estimate the three terms in the right side of~\eqref{eq:vBound} separately.
For the first term, we use the inequality~\eqref{eq:convprop} and obtain, for some constant $C' > C$,
\begin{align} \label{eq:vBound1}
    \tau^{-\frac12 + \frac\alpha2} \int_{\C_\infty \cap B_{\sqrt{\tau}}\left(y\right)}  \Phi_C\left( \tau, z-y \right) \Phi_C\left( t-\tau, x-z \right) \, dz & \leq \tau^{-\frac12 + \frac\alpha2} \int_{\Zd} \Phi_C\left( \tau, z-y \right) \Phi_C\left( t-\tau, x-z \right) \, dz \notag \\
    & \leq \tau^{-\frac12 + \frac\alpha2} \Phi_{C'}\left( t, x-y \right). 
\end{align}
To estimate the second term, we use the estimate~\eqref{eq:PhiCTrick} and deduce that
\begin{equation} \label{eq:vBound2}
 |z - y |^{\alpha} \Phi_C\left( \tau, z-y \right) \leq  \tau^{\frac{\alpha}{2}} \Phi_{C'}\left( \tau, z-y \right),
\end{equation}
for some constant $C' > C$.
To estimate the third term in the right side of~\eqref{eq:vBound}, we use the estimate~\eqref{eq:PhiCTrick} again with the value $k = d/2 + \alpha$ (the estimate applies even though $k$ is not an integer). We obtain that for some constant $C' > C$ and for any point $z \in \C_\infty \setminus B_{t-\tau}(y)$,
\begin{equation} \label{eq:vBound3}
    (t-\tau)^{d/2}|z - y|^{\alpha} \Phi_C\left( \tau, z-y \right) \leq |z - y|^{d/2 + \alpha} \Phi_C\left( \tau, z-y \right) \leq \Phi_{C'}\left( \tau, z-y \right).
\end{equation}
Finally combining the identity~\eqref{eq:vBound} and the estimates~\eqref{eq:vBound1},~\eqref{eq:vBound2},~\eqref{eq:vBound3}, we obtain
\begin{equation*}
    \left| v(t , x , \tau , y) \right| \leq C \tau^{- \frac12 + \frac \alpha2} \Phi_{C'} \left( t , x - y \right).
\end{equation*}
The proof of Lemma~\ref{lem:v} is complete.
\end{proof}

\subsection{The two-scale expansion} \label{section4.2} 
The main objective of this section is to prove that the weighted $L^2$ norm of the function $w$ is small in the sense of~\eqref{mainresultprop4.19}. Before starting the proof, we recall the notation for the function $\Psi_C$ introduced in~\eqref{def.psiC}. We also recall the notation convention for discrete and continuous derivatives:
\begin{itemize}
    \item In the proof of Proposition~\ref{lem:w}, the functions $\bar p ,$ $\psi$ and $\eta$ are defined on $\Rd$ and valued in $\R$, for these functions, we use the symbols $\nabla$ and $\Delta$ to denote respectively the continuous gradient and the continuous Laplacian. To refer to the discrete derivatives, we use the notations $\D, \D^*, \D^2, \D^3$ etc.
    \item All the other functions are defined on the discrete lattice $\Zd$ or on the infinite cluster $\C_\infty$, for these functions, we use the notation $\nabla$ to denote the discrete gradient defined on the edges and the notation $\D$ for the discrete derivative defined on the vertices, following the conventions of Section~\ref{Discanalandfunspa}.
    \end{itemize}
    
\begin{proposition}\label{lem:w}
For every exponent $\alpha >0$ , there exists a positive constant $C:= C(d , \p , \lambda, \alpha) < \infty$ such that for every point $y \in \Zd$ and every time $t \in (0 , \infty)$ such that $\sqrt{t} \geq \max \left( \M_{\mathrm{corr}, \alpha}(y), \M_{\mathrm{flux}, \alpha}(y)\right)$, one has, on the event $\{y \in \C_\infty\},$
\begin{equation}  \label{mainresultprop4.19}
    \left\Vert w(t, \cdot , \tau , y) \exp(\Psi_C(t, \vert \cdot - y\vert))\right\Vert_{L^2(\C_{\infty})} \leq C \left(\frac{t}{\tau}\right)^{\frac{1}{2}} \tau^{-\frac{d}{4} - \frac 12 + \frac{\alpha}{2}}.
\end{equation}
\end{proposition}

\begin{proof}
The key is to develop a differential inequality for the function $w$. The proof is decomposed into five steps and is organized as follows. In Step 1, we use the explicit formula for $w$ and apply the parabolic operator $\partial_t - \nabla \cdot \a \nabla $ to the map $w$ to obtain the formulas~\eqref{eq:defw} and~\eqref{eq:defw2}. In Step 2, we test the equation obtained in~\eqref{eq:defw} with the function $\psi w$, where $\psi$ is a map which is either equal to the constant $1$ or equal to the function $x \mapsto \exp(\Psi_C(t, \vert x - y\vert)$. In the three remaining steps, we treat the different terms obtained and complete the proof of the estimate~\eqref{mainresultprop4.19}.

\medskip

\textit{Step 1 : Establishing the equation for $w$.} We claim that the function $w$ satisfies the equation 
\begin{equation} \label{eq:defw}
    \left\{ \begin{aligned}
    \partial_t w (\cdot,\cdot, \tau , y) - \nabla \cdot \a \nabla w (\cdot,\cdot, \tau , y) &= f(\cdot,\cdot,  y) + \D^* \cdot F(\cdot,\cdot,  y) + \xi(\cdot,\cdot,  y) &~ \mbox{in}~(\tau , \infty) \times \C_\infty,   \\
     w (\tau , \cdot ,  y) &= 0 & ~\mbox{in}~ \C_\infty,
    \end{aligned} \right.
\end{equation}
where the three functions $f: (0, \infty) \times \C_\infty \times \C_\infty \to \R$,   $F: (0, \infty) \times \C_\infty \times \C_\infty \to \Rd$ and $\xi : (0, \infty) \times \C_\infty \times \C_\infty \to \R$ are defined by the formulas, for each $(t ,y) \in (0, \infty) \times \C_\infty$,
\begin{equation} \label{eq:defw2}
    \left\{ \begin{aligned}
    f(t,\cdot,y) &= \dsigk \left(\Delta \bar{p}(t, \cdot-y) - (-\D^* \cdot \D \bar{p}(t, \cdot-y))\right) + \sum_{k=1}^d \left(\partial_t \Dr{k} \bar{p}(t, \cdot-y) \right)\chi_{e_k}(\cdot), \\
    [F]_i(t, \cdot, y) &= \sum_{k=1}^d [\a\D \Dr{k} \bar{p}(t, \cdot-y)]_i T_{e_i}(\chi_{e_k})(\cdot),  \qquad \forall i \in \{ 1 , \ldots, d \}, \\
    \xi(t, \cdot, y) & = \sum_{k=1}^d \D^* \Dr{k} \bar{p}(t, \cdot-y) \cdot \tilde{\mathbf{g}}^*_{e_k}(\cdot),
    \end{aligned} \right.
\end{equation}
where  $\tilde{\mathbf{g}}^*_{e_k}$ is a translated version of the flux $\tilde{\mathbf{g}}_{e_k}$ defined by the formula, for each $x \in \C_\infty$,
\begin{equation*}
    \tilde{\mathbf{g}}^*_{e_k}(x) := \begin{pmatrix}
 T_{- e_1} \left[\a \left( \D \chi_{e_k} + e_k \right) - \dsigk e_k \right]_1 \\[3mm]
\vdots \\[3mm]
 T_{- e_d} \left[\a \left( \D \chi_{e_k} + e_k \right) - \dsigk e_k \right]_d  \\
\end{pmatrix},
\end{equation*}
and we recall the notation $\left[\a \left( \D \chi_{e_k} + e_k \right) - \dsigk e_k \right]_i$ introduced in Section~\ref{section1.6.1} to denote the $i$th-component of the vector ${\a \left( \D \chi_{e_k} + e_k \right) - \dsigk e_k}$.
In Appendix~\ref{appendixb}, it is proved that the translated flux $ \tilde{\mathbf{g}}^*_{e_k}$ has similar properties as the centered flux~$\tilde{\mathbf{g}}_{e_k}$. In particular, it is proved in Remark~\ref{rmk:Flux} that, for every radius $r \geq \M_{\mathrm{flux}, \alpha}(y)$,
\begin{equation*}
    \norm{\tilde{\mathbf{g}}^*_{e_k}}_{\aH^{-1}(\C_\infty \cap B_r(y))} \leq C r^{\alpha}.
\end{equation*}
The proof follows from a direct calculation. Since one has the equality $(\partial_t - \nabla \cdot \a \nabla) (v + q )(\cdot,\cdot, s , y) = 0$, it suffices to focus on the term $h$ in the definition of $w$, the details are left to the reader.

\medskip

\textit{Step 2 : A differential inequality.} In this step and the rest of the proof, we let $\psi$ be the function from $(0,\infty) \times \Rd$ to $\R$ which is either:
\begin{itemize}
    \item[(i)] The constant function equal to $1$; 
    \item[(ii)] The function $\exp(\Psi_C(t, \vert \cdot - y\vert))$, for some large constant $C > 0$.
\end{itemize}
We note that in both cases, the function $\psi$ is smooth and satisfies the following property
\begin{equation}\label{eq:hypoPsi}
\exists \tilde{C} > 1, \quad \forall |h| \leq 1, \quad |\nabla \psi(\cdot + h)| \leq \tilde{C} |\nabla \psi|.
\end{equation}
This estimate allows to replace the discrete derivative $\D \psi$, by the continuous derivative $\nabla \psi$ by only paying a constant. The strategy of the proof is then to test the equation~\eqref{eq:defw} with the function~$w\psi^2$ to derive a differential inequality. We first write
\begin{equation} \label{testinwpsi.2129}
    \underbrace{\int_{\C_\infty} \left(\left(\partial_t w - \nabla \cdot \a \nabla  w \right) \psi^2 w \right)}_{\textrm{LHS}} = \underbrace{\int_{\C_\infty}f w\psi^2  + \int_{\C_\infty} F \cdot \D(w\psi^2) + \int_{\C_\infty}\xi w\psi^2 }_{\mathrm{RHS}}.
\end{equation}
We then estimate the terms on the left and right-hand sides separately. Before starting the computation, we mention that in the following paragraphs all the derivatives and integrations are acting on the first spatial variable. For the left-hand side of~\eqref{testinwpsi.2129}, we have 
\begin{align*}
\text{LHS} & = \int_{\C_\infty} \left(\left(\partial_t w - \nabla \cdot \a \nabla  w \right) \psi^2 w \right)\\
&=  \int_{\C_\infty}  \frac{1}{2} \partial_t(\psi^2 w^2) - \int_{\C_\infty} (\partial_t \psi) \psi w^2 + \int_{\C_\infty} \nabla (\psi^2 w) \cdot \a \nabla w. 
\end{align*}
Using \eqref{eq:hypoPsi} and Young's inequality, we have, for any $a \in (0,\infty)$,
\begin{align*}
\int_{\C_\infty}\nabla (\psi^2 w) \cdot \a \nabla w  & \geq \lambda \int_{\C_\infty} \vert  \nabla w\vert^2 \psi^2 - a \int_{\C_\infty} \vert \nabla w \vert^2 \psi^2 - \frac{\tilde{C}^2}{a}  \int_{\C_\infty} \vert \nabla \psi \vert^2 w^2,
\end{align*} 
where, following the notation convention recalled at the beginning of this section, the notation $\nabla \psi$ denotes the continous gradient for functions $\psi$, while the notation $\nabla (\psi w)$ refers to the discrete gradient on the infinite cluster since the map $w$ is only defined on $\C_\infty$. By choosing the value $a = \frac\lambda2$, the previous display can be rewritten
\begin{equation}\label{eq:diffLHS}
\text{LHS} \geq \int_{\C_\infty}  \frac{1}{2} \partial_t(\psi^2 w^2) + \frac\lambda2 \int_{\C_\infty} \vert \nabla w \vert^2 \psi^2  - \int_{\C_\infty} (\partial_t \psi) \psi w^2 - \frac{2\tilde{C}^2}{\lambda} \int_{\C_\infty} \vert \nabla \psi \vert^2 w^2.
\end{equation}
We now focus on terms on the right-hand side of the equality~\eqref{testinwpsi.2129}. For the first two terms, we use Young's inequality and obtain
\begin{align*}
\int_{\C_\infty}f w \psi^2  + \int_{\C_\infty} F \cdot \D(w\psi^2) & \leq \int_{\C_\infty} t f^2 \psi^2 + \frac{1}{4t} \int_{\C_\infty} w^2 \psi^2  + \frac{\tilde{C}^2}{\lambda} \int_{\C_\infty} \vert F \vert^2 \psi^2 + \frac{\lambda }{4} \int_{\C_\infty} \vert \nabla w \vert^2 \psi^2 \\
& \qquad + \frac\lambda2 \int_{\C_\infty} \vert F \vert^2 \psi^2 + \frac{2\tilde{C}^2}{\lambda} \int_{\C_\infty} \vert \nabla \psi \vert^2 w^2.
\end{align*}
A combination of the two previous displays and the identity~\eqref{testinwpsi.2129} shows
\begin{align} \label{eq:diffLHSRHSbis}
    \int_{\C_\infty} \left( \frac{1}{2} \partial_t(\psi^2 w^2) + \frac{\lambda}{4}  \vert \nabla w \vert^2 \psi^2 \right) \leq  \int_{\C_\infty} w^2 \left( (\partial_t \psi) \psi  + \frac{1}{4t}\psi^2  + \frac{4 \tilde{C}^2}{\lambda} \vert \nabla \psi \vert^2 \right) \\ \quad + 2 \int_{\C_\infty} t f^2 \psi^2 +  \left( \frac{\tilde{C}^2}{\lambda} + \frac\lambda2 \right) \int_{\C_\infty} \vert F \vert^2 \psi^2 + \int_{\C_\infty}\xi w\psi^2. \notag
\end{align}
The value of the constants in the second line of the estimate~\eqref{eq:diffLHSRHSbis} does not need to be tracked in the proof, we thus rewrite it in the following form
\begin{align} \label{eq:diffLHSRHS}
    \int_{\C_\infty} \left( \frac{1}{2} \partial_t(\psi^2 w^2) + \frac{\lambda}{4}  \vert \nabla w \vert^2 \psi^2 \right) & \leq  \int_{\C_\infty} w^2 \left( (\partial_t \psi) \psi  + \frac{1}{4t}\psi^2  + \frac{4 \tilde{C}^2}{\lambda} \vert \nabla \psi \vert^2 \right) \\ & \quad + C \left(  \int_{\C_\infty} t f^2 \psi^2 +  \vert F \vert^2 \psi^2 + \xi w\psi^2 \right). \notag
\end{align}
To complete the proof, we need to prove that the quantities on the second line of the previous display are small:
\begin{itemize}
    \item One needs to prove that the term $ \int_{\C_\infty} \xi(t, \cdot, y) (\psi^2(t, \cdot, y) w(t,\cdot, s,y))$ is small, this is proved in Step 3;
    \item One needs to prove that the term $\int_{\C_\infty} \left( t f^2(t, \cdot, y) + \vert F \vert^2(t,\cdot,y) \right)\psi^2(t, \cdot, y)$ is small, this is proved in Step 4.
\end{itemize}

\smallskip

\textit{Step 3 : Estimate of the term $ \int_{\C_\infty} \xi(t, \cdot, y) (\psi^2(t, \cdot, y) w(t,\cdot, s,y)) $.}
The term $\xi$ involves the centered flux $\tilde{\mathbf{g}}^*_{e_k}$, to prove that this integral is small, the strategy is to use the weak norm estimate on this function stated in Proposition~\ref{prop.sublin.flux} and a multiscale argument. Specifically, the goal of this step is to prove the inequality
\begin{equation} \label{mainest.Step314125}
    \int_{\C_\infty} \xi \psi^2 w  \leq C  t^{-\frac d4 - 1 +\frac\alpha2} \norm{\nabla(w \psi)}_{L^2(\C_\infty)} + C t^{ - \frac{d}{4} - \frac{3}{2} + \frac{\alpha}{2}} \norm{w \psi}_{L^2(\C_\infty)}.
\end{equation}
As in Lemma~\ref{l.lemma4.2}, we need to split the space into scales and we define the dyadic annuli: for each integer $m \geq 1$, we let~$A_m$ be the annulus $A_m := \left\{ z \in \Zd ~:~ 2^{m-1} \sqrt{t} \leq |z-y| < 3 \cdot 2^{m} \sqrt{t} \right\}$, we also let $A_0 := B_{\sqrt{t}}(y)$.
We then split the proof into two steps:
\begin{enumerate}[label=(\roman*)]
    \item We first prove the estimate
\begin{equation}\label{eq:WeakNormMultiscale}
\int_{\C_\infty} \xi \psi^2 w  \leq C  \Xi_1 \norm{\nabla(w \psi)}_{L^2(\C_\infty)} + C \Xi_2 \norm{w \psi}_{L^2(\C_\infty)},
\end{equation}
where the two quantities $\Xi_1, \Xi_2$ are defined by the formulas 
\begin{equation}\label{eq:defXi}
\begin{split}
\Xi_1 &:= \left( \sum_{k=1}^d \sum_{m=0}^{\infty} \left(2^m \sqrt{t}\right)^d \norm{\tilde{\mathbf{g}}^*_{e_k}}^2_{\aH^{-1}(\C_\infty \cap A_m)} \norm{\D^2 \bar{p}(t,\cdot - y)\psi}^2_{L^{\infty}(A_m)} \right)^\frac12, \\
\Xi_2 &:= \left( \sum_{k=1}^d \sum_{m=0}^{\infty} \left(2^m \sqrt{t}\right)^d \norm{\tilde{\mathbf{g}}^*_{e_k}}^2_{\aH^{-1}(\C_\infty \cap A_m)} \right.\\
 & \left. \qquad \times \left(\norm{\D^3 \bar{p}(t,\cdot - y)\psi}^2_{L^{\infty}(A_m)} + \norm{\D^2 \bar{p}(t,\cdot - y) \D \psi}^2_{L^{\infty}(A_m)} +(3^m \sqrt{t})^{-2}\norm{\D^2 \bar{p}(t,\cdot - y)\psi}^2_{L^{\infty}(A_m)}\right) \right)^\frac12;
\end{split}
\end{equation}
\item We then prove the estimates
\begin{equation} \label{e.Xi1Xi2.estsmall}
\Xi_1 \leq C t^{-\frac{d}{4} - 1 + \frac{\alpha}{2}} ~\mbox{and}~ \Xi_2 \leq C t^{-\frac{d}{4} - \frac{3}{2} + \frac{\alpha}2}.
\end{equation}
\end{enumerate}
The estimate~\eqref{mainest.Step314125} is a consequence of the inequalities~\eqref{eq:WeakNormMultiscale} and~\eqref{e.Xi1Xi2.estsmall}.

We now focus on the proof of the inequality~\eqref{eq:WeakNormMultiscale}. The strategy is to use a multiscale analysis. We let $\eta$ be a smooth cutoff function from $\Rd$ to $\R$ satisfying the properties 
\begin{equation}\label{eq:defeta}
0 \leq \eta \leq 1, \qquad \vert \nabla \eta \vert \leq 1, \qquad \supp(\eta) \subset B_3(y), \qquad \eta \equiv 1 \text{ in } B_1(y). 
\end{equation}
For an integer $m$, we define the rescaled version $\eta_m$ of $\eta$ according to the formula $\eta_m := \eta\left(\frac{\cdot \, - y}{2^m \sqrt{t}}+y\right)$, we also set the convention $\eta_{-1,y} \equiv 0$. This function satisfies the property: for each $m \in \N$,
\begin{equation*}
\vert \nabla \eta_{m} \vert \leq \left(2^m\sqrt{t}\right)^{-1}, \quad \supp(\eta_{m}) \subset B_{2^{m+1}\sqrt{t}}(y), \quad \eta_{m} \equiv 1 \text{ in } B_{2^{m}\sqrt{t}}(y), \quad \supp(\eta_{m} - \eta_{m-1}) \subset A_m.
\end{equation*}
We also note that the family of functions $\eta_m$ can be used as a partition of unity and we have
\begin{equation*}
    1 = \sum_{m=0}^{\infty} \left( \eta_{m} - \eta_{m-1}\right).
\end{equation*}
With this property, we compute
\begin{align} \label{est.xipsisquarew.2342}
\int_{\C_\infty} \xi \psi^2 w & = \sum_{m=0}^{\infty} \int_{\C_\infty} \left(\eta_{m} - \eta_{m-1}\right)\xi \psi^2 w  \\
&= \sum_{k=1}^d \sum_{m=0}^{\infty} \int_{\C_\infty \cap B_{2^{m+1} \sqrt{t}}(y)} \tilde{\mathbf{g}}^*_{e_k} \cdot \D^* \Dr{k} \bar{p}(t, \cdot-y) \left(\eta_{m} - \eta_{m-1}\right) \psi^2 w \notag\\
&\leq \sum_{k=1}^d \sum_{m=0}^{\infty} (2^m \sqrt{t})^d \norm{\tilde{\mathbf{g}}^*_{e_k}}_{\aH^{-1}(\C_\infty \cap A_m)}  \norm{\D^* \Dr{k} \bar{p}(t, \cdot-y) \left(\eta_{m} - \eta_{m-1}\right) \psi^2 w}_{\aH^{1}(\C_\infty \cap A_m)}. \notag 
\end{align}
Then we calculate the $\underline{H}^1$-norm of the term $\D^* \Dr{k} \bar{p}(t, \cdot-y) \left(\eta_{m} - \eta_{m-1}\right) \psi^2 w$. We use the fact that the function $ \left(\eta_{m} - \eta_{m-1}\right)$ is supported in the annulus $A_m$ and write
\begin{equation*}
\norm{\D^* \Dr{k} \bar{p}(t, \cdot-y) \left(\eta_{m} - \eta_{m-1}\right) \psi^2 w}_{\aH^{1}(\C_\infty \cap B_{2^{m+1}\sqrt{t}}(y))} \leq C \left( I_1 + I_2\right),
\end{equation*}
where the two terms $I_1$ and $I_2$ are defined by the formulas
\begin{align*}
I_1 & :=   \norm{\D^2 \bar{p}(t,\cdot - y)\psi}_{L^{\infty}(A_m)} \norm{ \nabla (w \psi)}_{\underline{L}^2(\C_\infty \cap A_m)}  \\
I_2 & :=  \left( \norm{\D^3 \bar{p}(t,\cdot - y)\psi}_{L^{\infty}(A_m)} + \norm{\D^2 \bar{p}(t,\cdot - y) \D \psi}_{L^{\infty}(A_m)} \right. \\
& \qquad + \left. (3^m \sqrt{t})^{-1}\norm{\D^2 \bar{p}(t,\cdot - y)\psi}_{L^{\infty}(A_m)} \right) \norm{w \psi}_{\underline{L}^2(\C_\infty \cap A_m)}. 
\end{align*}
We put these equations back into the right-hand side of the estimate~\eqref{est.xipsisquarew.2342}. This gives
\begin{equation*}
    \int_{\C_\infty} \xi \psi^2 w \leq C \sum_{k=1}^d \sum_{m=0}^{\infty} (2^m \sqrt{t})^{\frac{d}{2}} \norm{\tilde{\mathbf{g}}^*_{e_k}}_{\aH^{-1}(\C_\infty \cap A_m)} \left( I_1 + I_2 \right).
\end{equation*}
We then estimate the two terms on the right side by applying the Cauchy-Schwarz inequality. For the term involving the quantity $I_1$, we obtain
\begin{align*}
\lefteqn{\sum_{k=1}^d \sum_{m=0}^{\infty} (2^m \sqrt{t})^{\frac{d}{2}} \norm{\tilde{\mathbf{g}}^*_{e_k}}_{\aH^{-1}(\C_\infty \cap A_m)} I_1}\\
\leq & \left(\sum_{k=1}^d \sum_{m=0}^{\infty} (2^m \sqrt{t})^{d} \norm{\tilde{\mathbf{g}}^*_{e_k}}^2_{\aH^{-1}(\C_\infty \cap B_{2^{m+1}\sqrt{t}}(y))} \norm{\D^2 \bar{p}(t,\cdot - y)\psi}^2_{L^{\infty}(A_m)} \right)^{\frac{1}{2}} \\
& \qquad \times \left(\sum_{k=1}^d \sum_{m=0}^{\infty} \norm{ \nabla (w \psi)}^2_{L^2(\C_\infty \cap A_m)}\right)^{\frac{1}{2}} \\
\leq & C \Xi_1  \norm{\nabla (w \psi)}_{L^2(\C_\infty)},
\end{align*}
where to go from the second to the third line, we used the definition of $\Xi_1$ given in~\eqref{eq:defXi} and the inequality $\sum_{m=1}^{\infty} \indc_{A_m(y)} \leq 4$. The same argument works for the terms involving the quantities~$I_2$ and $\Xi_2$, this concludes the proof of the estimate~\eqref{eq:WeakNormMultiscale}.

We now prove an estimate on the terms $\Xi_1, \Xi_2$; precisely we prove the inequality~\eqref{e.Xi1Xi2.estsmall} which is recalled below
\begin{equation}\label{eq:QuantXi}
\Xi_1 \leq C t^{-\frac{d}{4} - 1 + \frac{\alpha}{2}} ~\mbox{and}~ \Xi_2 \leq C t^{-\frac{d}{4} - \frac{3}{2} + \frac{\alpha}2}.
\end{equation}
The proof comes from a direct calculation of the quantities $\Xi_1$ and $\Xi_2$.
We recall that the function $\psi$ is chosen to be either the constant function equal to $1$, or the function $\exp(\Psi_C(t, \vert \cdot - y\vert))$, for some large constant $C$. 

We first focus on the estimate of the term $\Xi_1$; if the constant $C$ in the definition of $\psi$ is chosen large enough, for instance larger than $8\sigk$, then one has the estimate 
$$
\norm{\D^2 \bar{p}(t,\cdot - y)\psi}_{L^{\infty}(A_m(y))} \leq C t^{-\frac{d}{2}-1} \exp\left(-\frac{2^{2m}}{8 \sigk}\right).
$$
Thanks to the assumption $\sqrt{t} > \M_{\mathrm{flux}, \alpha}(y)$, we have the estimate $$\norm{\tilde{\mathbf{g}}^*_{e_k}}_{\aH^{-1}(\C_\infty \cap A_m)}  \leq \norm{\tilde{\mathbf{g}}^*_{e_k}}_{\aH^{-1}(\C_\infty \cap B_{2^{m+1}\sqrt{t}}(y))} \leq C \left(2^m\sqrt{t}\right)^{\alpha}.$$ Combining these two bounds with the definition of $\Xi_1$ given in~\eqref{eq:defXi}, we obtain
\begin{equation*}
\left( \Xi_1 \right)^2 \leq C\sum_{k=1}^d \sum_{m=0}^{\infty} \left(2^m\sqrt{t}\right)^{d + 2\alpha} t^{-d-2} \exp\left(-\frac{2^{2m}}{4 \sigk}\right) \leq C t^{-\frac{d}{2}-2+\alpha}.
\end{equation*}
The term $\Xi_2$ can be estimated thanks to a similar strategy and the details are left to the reader. The proof of the estimate~\eqref{eq:QuantXi}, and thus of the inequality~\eqref{mainest.Step314125} is complete.

\medskip

\textit{Step 4 : Quantification of the term $\int_{\C_\infty} \left( t f^2(t, \cdot, y) + \vert F \vert^2(t,\cdot,y) \right)\psi^2(t, \cdot, y) $.} The goal of this step is to prove the inequality
\begin{equation} \label{eq:QuantfF}
    \int_{\C_\infty} \left( t f^2(t, \cdot, y) + \vert F \vert^2(t,\cdot,y) \right)\psi^2(t, \cdot, y)  \leq C t^{-\frac{d}{2} - 2 + \alpha}.
\end{equation}
As was the case in the previous step, the function $\psi$ is either the constant function equal to $1$ or the function $\exp(\Psi_C(t, \vert \cdot - y\vert))$. In the latter case, we assume that the constant $C$ is at least larger than~$8 \sigk$.

We first consider the term involving the function $f$. From the definition of this function given in~\eqref{eq:defw2}, we see that it is the sum of two terms. The first one is the difference of the discrete and the continuous Laplacian of the heat kernel $\bar p$; it can be estimated as follows
$$
\left \vert \left( \Delta \bar{p}(t, \cdot-y) - (-\D^* \cdot \D \bar{p}(t, \cdot-y)) \right) \psi \right\vert \leq C t^{-\frac{d}{2} - \frac{3}{2}} \exp\left( - \frac{\vert \cdot - y \vert^2}{4 \sigk t}\right).
$$
The second term is the quantity 
$\sum_{k=1}^d \partial_t \Dr{k} \bar{p}(t, \cdot-y) \chi_{e_k}(\cdot) $. To estimate it, we split the space into different scales using the functions $\eta_m$ introduced in Step 3. This gives 
\begin{align*}
& \lefteqn{\int_{\C_\infty}  t\sum_{k=1}^d \left(\partial_t \left(\Dr{k} \bar{p}(t, \cdot-y) \chi_{e_k} \right) \psi\right)^2} \\ 
 =  & \sum_{k=1}^d \sum_{m=0}^{\infty} \int_{\C_\infty} t(\eta_{m} - \eta_{m-1})\left( \partial_t\Dr{k} \bar{p}(t, \cdot-y) \psi \right)^2  \chi_{e_k}^2 \\
\leq &\sum_{k=1}^d \sum_{m=0}^{\infty} t \left(2^{m}\sqrt{t}\right)^d \norm{\chi_{e_k}}^2_{\aL^2(\C_\infty \cap B_{2^{m+1}\sqrt{t}}(y))} \norm{\left( \partial_t\Dr{k} \bar{p}(t, \cdot-y) \psi \right)}^2_{L^{\infty}(A_m)}.
\end{align*}
We then use the assumption $\sqrt{t} > \M_{\mathrm{corr}, \alpha}(y)$, which implies $\norm{\chi_{e_k}}_{\aL^{2}(\C_\infty \cap B_{2^{m+1}\sqrt{t}}(y))} \leq C \left(2^m\sqrt{t}\right)^{\alpha}$, and the estimate
$$
\norm{\left( \partial_t\Dr{k} \bar{p}(t, \cdot-y) \psi \right)}_{L^{\infty}(A_m)} \leq C t^{-\frac{d}{2}-\frac{3}{2}} \exp\left(-\frac{2^{2m}}{8 \sigk}\right),
$$ 
to obtain the inequality
\begin{align*}
\int_{\C_\infty}  t\sum_{k=1}^d \left(\partial_t \left(\Dr{k} \bar{p}(t, \cdot-y) \chi_{e_k} \right) \psi\right)^2 & \leq C\sum_{k=1}^d \sum_{m=0}^{\infty} t \left(2^m\sqrt{t}\right)^{d +2\alpha} t^{-d-3} \exp\left(-\frac{2^{2m}}{4 \sigk}\right) \\ & \leq C t^{-\frac{d}{2}-2+\alpha}.
\end{align*}
The estimate for the term involving the function $F$ is similar and we skip its proof.

\medskip

\textit{Step 5 : The conclusion.}
We collect the results established in the previous steps and complete the proof of Proposition~\ref{lem:w}. We first consider the inequality~\eqref{eq:diffLHSRHS} in the case $\psi = 1$. This gives
\begin{equation*}
    \int_{\C_\infty} \left( \frac{1}{2} \partial_t w^2 + \frac{\lambda}{4}  \vert \nabla w \vert^2 \right) \leq  \frac{1}{4t} \int_{\C_\infty} w^2  + C \left( \int_{\C_\infty} t f^2 + \vert F \vert^2 +\xi w \right),
\end{equation*}
since in this case the constant $\tilde C$ introduced in~\eqref{eq:hypoPsi} is equal to $1$.
Applying the main results~\eqref{mainest.Step314125} of Step 3 and~\eqref{eq:QuantfF} of Step 4, we deduce 
\begin{equation*}
\int_{\C_\infty}  \left(\partial_t  w^2 + \frac{\lambda}{2}  \vert \nabla w \vert^2 \right)  \leq  \frac{1}{2t}  \int_{\C_\infty} w^2 + C t^{-\frac{d}{2} - 2 +\alpha} + C  t^{-\frac{d}{4} - 1 +\frac{\alpha}2} \left( \norm{\nabla w }_{L^2(\C_\infty)} + t^{-\frac12}\norm{w}_{L^2(\C_\infty)} \right).
\end{equation*}
By Young's inequality, the previous display can be simplified
\begin{equation*}
\int_{\C_\infty}  \left( \partial_t  w^2 + \frac{\lambda}{4}  \vert \nabla w \vert^2 \right)  \leq  \frac{1}{t}  \int_{\C_\infty} w^2 + C t^{-\frac{d}{2} - 2 +\alpha},
\end{equation*}
which implies
\begin{equation*}
\partial_t \int_{\C_\infty}   w^2  \leq  \frac{1}{t}  \int_{\C_\infty}    w^2 + C t^{-\frac{d}{2} - 2 +\alpha}.
\end{equation*}
By integrating over the time interval $[\tau,t]$, we obtain that there exists a constant $C := C(  d, \p, \lambda) < \infty$ such that 
\begin{equation}\label{eq:wL2}
\int_{\C_\infty}  w^2(t, \cdot, s, y) \leq C\left(\frac{t}{\tau}\right) \tau^{-\frac{d}{2}-1+\alpha}. 
\end{equation}
We now consider the inequality~\eqref{eq:diffLHSRHS} in the case $\psi = \exp(\Psi_C(t, \cdot - y))$. This gives
\begin{multline*}
    \int_{\C_\infty} \left( \frac{1}{2} \partial_t(\psi^2 w^2) + \frac{\lambda }{4}  \vert \nabla w \vert^2 \psi^2 \right) \leq  \int_{\C_\infty} w^2 \left( (\partial_t \psi) \psi  + \frac{1}{8t}\psi^2  + \frac{4\tilde{C}^2}{\lambda} \vert \nabla \psi \vert^2 \right) \\ + C  t^{-\frac{d}{2} - 2 +\alpha} +  C  t^{-\frac{d}{4} - 1 +\frac{\alpha}2} \left( \norm{\nabla \left( w\psi\right) }_{L^2(\C_\infty)} + t^{-\frac12}\norm{w\psi}_{L^2(\C_\infty)} \right).
\end{multline*}
Applying Young's inequality, the previous display can be simplified and we obtain
\begin{equation*}
    \int_{\C_\infty} \left( \frac{1}{2} \partial_t(\psi^2 w^2) + \frac{\lambda}{8}  \vert \nabla w \vert^2 \psi^2 \right) \leq  \int_{\C_\infty} w^2 \left( (\partial_t \psi) \psi  + \frac{1}{4t}\psi^2  + \frac{8\tilde C}{\lambda} \vert \nabla \psi \vert^2 \right) + C  t^{-\frac{d}{2} - 2 +\alpha} .
\end{equation*}
We then note that if the constant $C$ in the definition of $\psi = \exp(\Psi_C(\cdot, \cdot - y))$ is chosen large enough, then we have
\begin{equation} \label{essential.ineq}
    (\partial_t \psi) \psi  + \frac{1}{4t}\psi^2  + \frac{8 \tilde C}{\lambda} \vert \nabla \psi \vert^2 \leq \frac{C}{t}.
\end{equation}
A combination of the two previous displays shows the differential inequality
\begin{equation*}
   \partial_t \int_{\C_\infty} \frac{1}{2} \psi^2 w^2 \leq  \frac{C}{t} \int_{\C_\infty} w^2 + C  t^{-\frac{d}{2} - 2 +\alpha}.
\end{equation*}
We then apply~\eqref{eq:wL2} to obtain
\begin{equation*}
    \partial_t \int_{\C_\infty} \frac{1}{2} \psi^2 w^2 \leq C t^{-1}\left(\frac{t}{\tau}\right) \tau^{-\frac{d}{2}-1+\alpha} + C  t^{-\frac{d}{2} - 2 + \alpha}.
\end{equation*}
Integrating with respect to the time $t$ and recalling that $w\left(\tau , \cdot, \tau , y \right) = 0$, we obtain
\begin{equation*}
\int_{\C_\infty} (w \psi)^2  \leq C \left(\frac{t}{\tau}\right) \tau^{-\frac{d}{2} - 1+\alpha},
\end{equation*}
for some constant $C := C( d , \p , \lambda) < \infty$.
This completes the proof of Proposition~\ref{lem:w}.
\end{proof}

\begin{remark}
The reason why we choose the function $\psi = \exp(\Psi_C(\cdot, \cdot - y))$, and why the main result~\eqref{mainresultprop4.19} of Proposition~\ref{lem:w} is stated with this function can be explained by the inequalities~\eqref{eq:hypoPsi} and~\eqref{essential.ineq}. Indeed, the function $(t,x) \mapsto \exp(\Psi_C(t, x - y))$ is the one which has the fastest growth as $x$ tends to infinity such that the inequalities~\eqref{eq:hypoPsi} and~\eqref{essential.ineq} are satisfied. In particular there is an important difference between the discrete setting and the continuous setting: in the latter, one does not need the inequality~\eqref{eq:hypoPsi} to hold which allows to choose the function $\psi(t,x) := \exp \left( \frac{\left|x - y \right|^2}{Ct} \right)$, and to obtain the result with this function (see~\cite[Lemma 8.22]{armstrong2017quantitative}). This observation is consistent with the asymptotic behavior of the discrete heat kernel on the percolation cluster or on $\Zd$ which is described by Proposition~\ref{proposition2.11}.
\end{remark}

We have now collected all the necessary results to prove Theorem~\ref{mainthm}. The following section is devoted to its proof.

\subsection{Proof of Theorem 1} \label{section4.3}
By translation invariance of the model, it is sufficient to prove the result when $y = 0 \in \C_\infty$. We fix an exponent $\delta >0$; the objective is to apply the results of Proposition~\ref{lem:w}, Lemma~\ref{l.lemma4.2} and Lemma~\ref{lem:v} with the mesoscopic time $\tau = t^{1-\kappa}$ and with following values of exponents
\begin{equation*}
    \alpha := \frac{\delta}2 \hspace{5mm} \mbox{and} \hspace{5mm} \kappa := \frac{\delta}{d+2},
\end{equation*}
For later use, we note that with these specific choices of exponents, the following estimates hold
\begin{equation} \label{eq:usest1635}
    (1 - \kappa) \left( \frac 12 - \frac \alpha2 \right) > \frac12 - \delta \hspace{5mm} \mbox{and} \hspace{5mm} - \kappa + (1- \kappa) \left( \frac d4 + \frac 1 2 - \frac \alpha2 \right) > \frac d4 + \frac 12 - \delta.
\end{equation}
The proof relies on an induction argument and we give a setup of the proof. We first define the sequence $\kappa_n$ of real numbers inductively by the formula
\begin{equation} \label{def:kappan}
    \kappa_ 0 = \frac\kappa2 \hspace{5mm} \mbox{and} \hspace{5mm} \kappa_{n+1} := \min \left( (1 - \kappa) \kappa_n + \frac\kappa2, \frac 12 - \delta \right)
\end{equation}
This sequence is increasing and is ultimately constant equal to the value $\frac 12 - \delta$. We let $N$ be the integer
\begin{equation*}
    N := \inf \left\{ n \in \N ~:~ \kappa_n = \frac 12 - \delta \right\},
\end{equation*}
and we note that this integer only depends on the parameters $d, \p, \lambda$ and $\delta$. For each point $z \in \Zd$, we define the random time $\mathcal{T}^{0}_{\mathrm{par}}(z)$ according to the formula
\begin{equation*}
\mathcal{T}^{0}_{\mathrm{par}}(z) := 4 \max \left( \T_{\mathrm{approx}, \alpha}(z)^{\frac{1}{1 - \kappa}}, \M_{\mathrm{corr}, \alpha}(z)^{\frac{2}{1-\kappa}}, \M_{\mathrm{flux}, \alpha}(z)^2  \right)
\end{equation*}
so that for any time $t \geq\mathcal{T}^{0}_{\mathrm{par}}(z)$, all the results of Sections~\ref{section4.1} and~\ref{section4.2} are valid with the value $\tau := t^{1 - \kappa}$.
We then upgrade the random variable $\mathcal{T}^{0}_{\mathrm{par}}(z)$ and define
\begin{equation} \label{eq:TV1244}
\mathcal{T}^{1}_{\mathrm{par}} := \sup \left\{ t \in [1,\infty)~:~ \exists z \in \C_\infty~\mbox{such that} ~ |z| \leq (N+1) t^{\frac{1}{(1 - \kappa)^N}} ~\mbox{and}~ \mathcal{T}^{0}_{\mathrm{par}}(z) \geq t\right\},
\end{equation}
so that for any time $t \geq \mathcal{T}^{1}_{\mathrm{par}}$, and any point $z \in \C_\infty$ satisfying $|z| \leq (N+1) t^{\frac{1}{(1 - \kappa)^N}}$, one has the estimate $t \geq \mathcal{T}^{0}_{\mathrm{par}}(z)$; this implies that all the results of Sections~\ref{section4.1} and~\ref{section4.2} are valid with the value $\tau := t^{1 - \kappa}$ for the heat kernel started from the point $z$. This construction is identical to the used to define the minimal time $\mathcal{T}_{\mathrm{NA}}'(x)$ in~\eqref{def.Mdens}. As it was the case for the random variable $\mathcal{T}_{\mathrm{NA}}'(x)$, an application of Lemma~\ref{lemma1.5} shows the stochastic integrability estimate
\begin{equation*}
  \mathcal{T}^{1}_{\mathrm{par}} \leq \O_s \left( C \right).  
\end{equation*}
For each integer $n \in \{ 0 , \ldots, N \}$, we let $H_n$ be the following statement.

\smallskip

\textit{Statement $H_n$.} There exists a constant $C(d , \lambda, n) <\infty$ such that for each time $t \geq \left( \mathcal{T}^1_{\mathrm{par}} \right)^{\frac{1}{\left(1-\kappa\right)^n}}$, each point $x \in \C_\infty$, and each point $z \in \C_\infty$ satisfying $|z| \leq  (N-n) t^{\frac{1}{(1 - \kappa)^{N-n}}}$, one has the estimate
    \begin{equation} \label{eq:inducthomog}
 \left| p (t , x , z) - \theta(\p)^{-1} \bar p (t , x - z) \right| \leq C t^{-\kappa_n } \Phi_C \left( t , x-z \right).    
\end{equation}
We prove by induction that the statement $H_n$ holds for each integer $n \in \{ 0 , \ldots, N \}$.

\medskip
\textit{The base case.} We prove that $H_0$ holds and first prove the $L^2$-estimate: for each time $t \geq \mathcal{T}^1_{\mathrm{par}}$, and each point $z \in \C_\infty$ satisfying $|z| \leq (N+1) t^{\frac{1}{(1 - \kappa)^{N}}}$,
\begin{equation} \label{eq:est.basecaseL2}
    \Vert \left( p(t, \cdot , z) - \theta(\p)^{-1}\bar{p}(t,  \cdot - z) \right) \exp\left( \Psi_C(t, \vert \cdot - z\vert)\right)\Vert_{L^2(\C_{\infty})} \leq C  t^{-\frac d4 -\frac\kappa2}.
\end{equation}
We recall the definitions of the functions $h$, $q$ and $v$ stated in~\eqref{sec42scexp.1605},~\eqref{thefunctionq},~\eqref{def.v.para16.36} respectively as well as the definition of $w$ given by the formula $w := h - v - q$. We write
\begin{align} \label{eq:decomppbarp}
    p(t, x , z) - \theta(\p)^{-1}\bar{p}(t,  x - z) & = \left( p(t , x , z) - q \left( t , x , \tau , z\right) \right) - v(t , x , \tau , z) + w (t,x, \tau , z)  \\ & \quad + \left( h(t , x , z) - \theta \left( \p \right)^{-1} \bar p \left(t , x , z \right) \right). \notag
\end{align}
To prove the estimate~\eqref{eq:est.basecaseL2}, we split the $L^2$-norm according to the decomposition~\eqref{eq:decomppbarp} and estimate each terms thanks to the results established in Sections~\ref{section4.1} and~\ref{section4.2}:
\begin{itemize}
    \item The term $\left( p(t , x , z) - q \left( t , x , \tau , z\right) \right)$ is estimated thanks to Lemma~\ref{l.lemma4.2}, this term accounts for an error of order 
    $$\left( \left( \frac \tau t \right)^{\frac{1}{2}}+ \tau^{-\frac 12 + \alpha} \right) t^{-\frac d4} =  \left( t^{-\frac \kappa2 } + t^{(1-\kappa) \left( - \frac12 + \alpha \right)} \right) t^{-\frac d4}\leq t^{- \frac d4 -\frac \kappa2 };$$
    \item The term $w$ is estimated thanks to Proposition~\ref{lem:w}, this term accounts for an error of order $$\left( \frac t\tau \right)^{\frac{1}{2}} \tau^{-\frac d4 - \frac12 + \frac \alpha2}  \leq t^{- \frac d4 - \frac 12 + \delta },$$
    where we used the estimate~\eqref{eq:usest1635};
    \item The term $v(t , x , \tau , z) $ is estimated thanks to Lemma~\ref{lem:v}, this term accounts for an error of order $$t^{-\frac d4} \tau^{-\frac 12 + \frac \alpha2 } = t^{ - \frac d4 + (1 - \kappa)\left( -\frac 12 + \frac \alpha2 \right)} \leq t^{-\frac d4 -\frac12 + \delta},$$ where we used the estimate~\eqref{eq:usest1635};
    \item The term $h(t , x , z) - \theta \left( \p \right)^{-1} \bar p \left(t , x , z \right)$ can be estimated as follows. By the definition of $h$ given in~\eqref{sec42scexp.1605}, we have
    \begin{equation*}
        h(t , x , z) - \theta \left( \p \right)^{-1} \bar p \left(t , x ,z \right) = \sum_{k= 1}^d  \Dr{k} \bar{p}(t,x - z) \chi_{e_k}(x).
    \end{equation*}
    The term can then be estimated by using the sublinearity of the corrector stated in Proposition~\ref{prop.sublin.corr} and the assumption $\sqrt{t} \geq \M_{\mathrm{corr}, \alpha}(z)$. The proof is similar to the one of Lemma~\ref{lem:v} and the details are left to the reader. It accounts for an error of order $t^{-\frac d4 -\frac12 + \delta}$.
    \end{itemize}
There remains to obtain the pointwise estimate~\eqref{eq:inducthomog} in the case $n = 0$ from the $L^2$-estimate~\eqref{eq:est.basecaseL2}. To this end, we fix a point $z \in \C_\infty$ such that $|z| \leq Nt^{\frac{1}{(1 - \kappa)^N}}$. We may without loss of generality restrict our attention to the points $x \in \C_\infty$ such that $|x| \leq (N+1)t^{\frac{1}{(1 - \kappa)^N}}$, otherwise we necessarily have $|x - z| \geq t$ and the inequality~\eqref{eq:inducthomog} is satisfied by Proposition~\ref{proposition2.11} and Remark~\ref{remark2.12}. We then use the semigroup property on the heat kernels $p$ and $\bar p$: for each $x , z \in \C_\infty$, one has
\begin{align}\label{eq:L2toPoint}
\lefteqn{p(t, x ,z) - \theta(\p)^{-1}\bar{p}(t,  x - z ) } \qquad & \\ & \qquad = \underbrace{\int_{\C_{\infty}} p\left(\frac{t}{2}, x, y\right)p\left(\frac{t}{2}, y, z\right) - \theta(p)^{-2}\bar{p}\left(\frac{t}{2}, x-y\right)\bar{p}\left(\frac{t}{2}, y-z\right) \, dy }_{\text{\eqref{eq:L2toPoint}-a}} \notag \\ &
 \qquad \quad + \theta(p)^{-1}\underbrace{\left(\theta(p)^{-1}\int_{\C_\infty} \bar{p}\left(\frac{t}{2}, x- y\right)\bar{p}\left(\frac{t}{2}, y- z\right) \, dy - \bar{p}(t, x-z)\right) }_{\text{\eqref{eq:L2toPoint}-b}}. \notag
\end{align}
We first treat the part \eqref{eq:L2toPoint}-a using the following $L^2$-estimate
\begin{equation*}
    \vert \text{\eqref{eq:L2toPoint}-a} \vert \leq \text{\eqref{eq:L2toPoint}-a1} + \text{\eqref{eq:L2toPoint}-a2},
\end{equation*}
where the two terms \text{\eqref{eq:L2toPoint}-a1} and \text{\eqref{eq:L2toPoint}-a2} are defined by the formulas
\begin{align*}
 \text{\eqref{eq:L2toPoint}-a1} &= \norm{\left( p\left(\frac{t}{2}, x , \cdot \right) - \theta(p)^{-1}\bar{p}\left(\frac{t}{2},  x - \cdot\right) \right) \exp\left( \Psi_C\left(\frac{t}{2}, \vert x - \cdot \vert\right)\right)}_{L^2(\C_{\infty})}\\
 & \qquad \times \norm{p\left(\frac{t}{2}, \cdot, z\right) \exp\left(-\Psi_C\left(\frac{t}{2}, \vert x - \cdot \vert\right)\right)}_{L^2(\C_\infty)} 
 \end{align*}
 and
 \begin{align*}
\text{\eqref{eq:L2toPoint}-a2} &= \norm{\left(  p\left(\frac{t}{2}, \cdot, z\right) -\theta(\p)^{-1} \bar{p}\left(\frac{t}{2}, \cdot- z\right) \right) \exp\left(\Psi_C\left(\frac{t}{2}, \vert \cdot - z\vert\right)\right)}_{L^2(\C_\infty)} \\
& \qquad \times \norm{\theta(p)^{-1}\bar{p}\left(\frac{t}{2}, x-\cdot\right)\exp\left(- \Psi_C\left(\frac{t}{2}, \vert \cdot - z\vert\right)\right)}_{L^2(\C_\infty)}.
\end{align*}
The term~\eqref{eq:L2toPoint}-a1 can be estimated by using the three following ingredients:
\begin{itemize}
\item The symmetry of the heat kernel $p$;
\item The $L^2$-estimate~\eqref{eq:est.basecaseL2} applied with the point $z = x$ which is valid under the assumption $|x| \leq (N+1)t^{\frac{1}{(1 - \kappa)^N}}$; 
\item The upper bound stated in Theorem~\ref{barlow}, which can be applied since we assumed $t \geq \T_{\mathrm{par}}^1 \geq 2 \T_{\mathrm{NA}}'(z)$, and reads, by increasing the value of the constant $C$ in the right side if necessary,
\begin{equation*}
	\norm{p\left(\frac{t}{2}, \cdot, z\right) \exp\left(-\Psi_C\left(\frac{t}{2}, \vert x - \cdot \vert\right)\right)}_{L^2(\C_\infty)}  \leq t^{\frac d4} \Phi_C(t, x- z).
\end{equation*}
\end{itemize}
These arguments imply the estimate
\begin{equation*}
 \text{\eqref{eq:L2toPoint}-a1} \leq t^{- \frac\kappa2}  \Phi_C(t, x- y).
\end{equation*}
The term \eqref{eq:L2toPoint}-a2 can be treated similarly and we omit the details. There remains to estimate the term~\eqref{eq:L2toPoint}-b. We note that by an application of the parallelogram law, i.e., the identity $\vert x-y\vert^2 + \vert y - z\vert^2 = 2 \left(\left\vert \frac{x-z}{2}\right\vert^2 + \left\vert y - \frac{x+z}{2}\right\vert^2\right)$, the function $\bar p$ satisfies the following property: for each $t \geq 0$, and each $x,y,z \in \Rd$,
$$
\bar{p}\left(\frac{t}{2}, x-y\right) \bar{p}\left(\frac{t}{2},y-z\right) = \bar{p}\left(t, x-z\right) \bar{p}\left(\frac{t}{4},y - \frac{x+z}{2}\right).
$$
By combining this identity with Proposition~\ref{prop:DensityContrentrationGaussian}, we obtain
\begin{align*}
\vert\text{\eqref{eq:L2toPoint}-b}\vert = \bar{p}(t, x-z) \left\vert\int_{\C_\infty} \theta(p)^{-1}\bar{p}\left(\frac{t}{4},y - \frac{x+z}{2}\right) \, dy - 1\right\vert \leq Ct^{- \frac 12 + \delta}\bar{p}(t, x-z).
\end{align*}
This finishes the proof of the base case.

\medskip

\textit{The iteration step.} We prove that, for each integer $n \in \N$, the statement $H_{n-1}$ implies the statement $H_{n}$. The strategy follows the one of the base case and we first prove the $L^2$-estimate, under the assumption that the statement $H_{n-1}$ is valid: for each time $t \geq \left( \mathcal{T}^1_{\mathrm{par}}\right)^{\frac{1}{\left(1-\kappa\right)^{n}}}$, each point $x \in \C_\infty$, and each point $z \in \C_\infty$ satisfying $|z| \leq  (N+1-n) t^{\frac{1}{(1 - \kappa)^{N-n}}}$,
    \begin{equation} \label{eq:inducstepL2ineqn}
    \Vert \left( p(t, \cdot , z) - \theta(\p)^{-1}\bar{p}(t,  \cdot - z) \right) \exp\left( \Psi_C(t, \vert \cdot - z\vert)\right)\Vert_{L^2(\C_{\infty})} \leq C  t^{-\frac d4 - \kappa_{n}}.
\end{equation}
We use the decomposition~\eqref{eq:decomppbarp} with the same value for the mesoscopic time $\tau = t^{1-\kappa}$. The error introduced by the terms $w$, $v$ and $h - \theta^{-1}(\p) \bar p$ are of order $t^{-\frac d4 - \frac 12 + \delta }$ which is smaller than the value $t^{-\frac d4 - \kappa_{n}}$ we want to prove in this step. The limiting factor comes from the term $\left( p(t , x , z) - q \left( t , x , \tau , z \right) \right)$ which is estimated in Lemma~\ref{l.lemma4.2} and gives an error of order $t^{-\frac{d}{4} - \frac  \kappa2 }$. The objective of the induction step is to improve this error by using the statement $H_{n-1}$. 

Under the assumption $t \geq \left( \mathcal{T}^1_{\mathrm{par}}\right)^{\frac{1}{\left(1-\kappa\right)^n}}$, we have $\tau = t^{1- \kappa} \geq \left( \mathcal{T}^1_{\mathrm{par}}\right)^{\frac{1}{\left(1-\kappa\right)^{n-1}}}$. We can thus apply the induction hypothesis $H_{n-1}$ with time $\tau$. This gives the inequality, for each point $x \in \C_\infty$, and each point $z \in \C_\infty$ satisfying $|z| \leq  (N+1-n) \tau^{\frac{1}{(1 - \kappa)^{N+1-n}}} = (N+1-n) t^{\frac{1}{(1 - \kappa)^{N-n}}}$,
\begin{equation} \label{eq:homogind1013}
     \left| p (\tau , x , z) - \theta(\p)^{-1} \bar p (\tau , x - z) \right| \leq C \tau^{- \frac{d}{4} -\kappa_{n-1} } \Phi_C \left( \tau , x-z \right).  
\end{equation}
This estimate can be used to improve the result of Lemma~\ref{l.lemma4.2} according to the following procedure. We go back to the proof of Lemma~\ref{l.lemma4.2} and in the inequality~\eqref{estqp14.01}, instead of using the Nash-Aronson estimate stated in Theorem~\ref{barlow}, we use the homogenization estimate~\eqref{eq:homogind1013}. We then proceed with the proof and do not make any other modification. This implies the following improved version of Lemma~\ref{l.lemma4.2}
\begin{equation*}
    \left|q(t , x , \tau , z) - p \left( t , x , z \right) \right| \leq \left( \tau^{-\kappa_{n-1}} \left( \frac \tau t \right)^{\frac{1}{2}}+ \tau^{-\frac 12 + \alpha} \right) \Phi_C(t, x-z ).
\end{equation*}
Once equipped with this estimate, we can prove the $L^2$-estimate~\eqref{eq:inducstepL2ineqn}. The proof is the same as the one presented in the base case, we only use the estimate~\eqref{eq:inducstepL2ineqn} instead of Lemma~\ref{l.lemma4.2}. We obtain, for any point $z \in \C_\infty$ satisfying $|z| \leq (N+1-n) \tau^{\frac{1}{(1 - \kappa)^{N-n}}}$,
\begin{equation*}
    \Vert \left( p(t, \cdot , z) - \theta(\p)^{-1}\bar{p}(t,  \cdot - z) \right) \exp\left( \Psi_C(t, \vert \cdot - z\vert)\right)\Vert_{L^2(\C_{\infty})} \leq t^{-\frac d4} \tau^{- \kappa_{n-1}} \left( \frac{\tau}{t} \right)^\frac 12 + t^{-\frac d4 - \frac 12 +\delta}.
\end{equation*}
We then use the equality $\tau = t^{1-\kappa}$ and the inductive definition of the sequence $\kappa_n$ stated in~\eqref{def:kappan} to deduce the estimate
\begin{equation*}
    \Vert \left( p(t, \cdot , z) - \theta(\p)^{-1}\bar{p}(t,  \cdot - z) \right) \exp\left( \Psi_C(t, \vert \cdot - z\vert)\right)\Vert_{L^2(\C_{\infty})} \leq C  t^{-\frac d4 - \kappa_{n}}.
\end{equation*}
The proof of the pointwise estimate~\eqref{eq:inducthomog} is identical to the proof written for the base case and we omit the details. This completes the proof of the induction step.

We now complete the proof of Theorem~\ref{mainthm}. We then define the minimal time $\mathcal{T}_{\mathrm{par}, \delta}(0) := \left( \mathcal{T}^{1}_{\mathrm{par}} \right)^{\frac{1}{(1-\kappa)^N}}$. Since the statement $H_N$ holds, we have the estimate, for each time $t \geq \mathcal{T}_{\mathrm{par}, \delta}(0)$,
\begin{equation*}
 \left| p (t , x , 0) - \theta(\p)^{-1} \bar p (t , x) \right| \leq C t^{-\frac 12 + \delta } \Phi_C \left( t , x \right).    
\end{equation*}
The proof of Theorem~\ref{mainthm} is complete in the case $y = 0$ is complete. The proof in the general case is obtained by using the stationarity of the model.

\section{Quantitative homogenization of the elliptic Green's function} \label{section5}

The objective of this section is to present a theorem of quantitative homogenization for the elliptic Green's function on the infinite cluster, i.e., to establish Theorem~\ref{mainthmell}. This result is a consequence of the quantitative homogenization theorem for the parabolic Green's function, Theorem~\ref{mainthm}, established in the previous section: in dimension $d \geq 3$, it can be essentially obtained by integrating the heat kernel over time since one has the identity, for each $x , y \in \C_\infty$,
\begin{equation} \label{ellandaprgreen}
    g \left( x , y \right) = \int_{0}^{\infty} p (t , x , y) \, dt.
\end{equation}

The case of the dimension $2$ is more specific and requires some additional attention. In this setting the heat kernel is not integrable as the time $t$ tends to infinity. This difficulty is related to the recurrence of the random walk on $\Z^2$ or to the unbounded behavior of the Green's function in dimension $2$. To remedy this, we use a corrected version of the formula~\eqref{ellandaprgreen}: for each $x , y \in \C_\infty$, one has
\begin{equation*}
    g(x,y) = \int_{0}^{\infty} \left(p(t , x , y) - p(t , y,y) \right) \, dt,
\end{equation*}
where $g$ is the unique elliptic Green's function on the infinite cluster under the environment $\a$ such that $g(y,y) = 0$.

\begin{proof}[Proof of Theorem~\ref{mainthmell}]
We first treat the case of the dimension $d \geq 3$. By the stationarity of the model, we prove the result in the case $y = 0$. To simplify the notation we write $g(x)$ instead of $g(x , 0)$. We let $\T_{\mathrm{par},\delta/2}(0)$ be the minimal time provided by Theorem~\ref{mainthm} with exponent $\delta/2$ and define the minimal scale $\M_{\mathrm{ell},\delta}(0)$ according to the formula
\begin{equation*}
\M_{\mathrm{ell}, \delta}(0) :=\T_{\mathrm{par},\delta/2}(0).
\end{equation*}
It is on purpose that we do not respect the parabolic scaling, we need to have $\M_{\mathrm{ell}, \delta}(0) \gg \sqrt{\T_{\mathrm{par},\delta/2}(0)}$.
As was mentioned in the introduction of this section, in dimension $d \geq 3$, we use the explicit formula~\eqref{ellandaprgreen} and note that Duhamel's principle implies the identity
\begin{equation*}
    \bar g \left( x \right) = \theta(\p)^{-1}  \int_{0}^{\infty} \bar p (t , x) \, dt.
\end{equation*}
We obtain
\begin{equation*}
    \left| g(x) - \bar g (x) \right| \leq \int_{0}^{\infty} \left|p(t , x , 0) - \theta(p)^{-1} \bar p(t , x ) \right| \, dt.
\end{equation*}
We then split the integral at time $\vert x \vert$,
\begin{align}\label{splitMpar1551}
    \left| g(x) - \bar g (x) \right| & \leq \int_{0}^{\vert x\vert} \left|p(t , x , 0) - \theta(p)^{-1}\bar p(t , x) \right| \, dt \\ & \quad + \int_{\vert x\vert}^\infty \left|p(t , x , 0) - \theta(p)^{-1}\bar p(t , x) \right| \, dt,  \notag  
\end{align} 
and estimate the two terms on the right side separately. The second term is the simplest one, we apply the quantitative estimate~\eqref{mainthmmainest} provided by Theorem~\ref{mainthm} and use the assumption $\vert x\vert \geq \T_{\mathrm{par},\delta/2}(0)$. This shows
\begin{align} \label{ellgreenest.larget}
    \int_{\vert x\vert}^\infty \left|p(t , x , 0) - \theta(p)^{-1}\bar p(t , x ) \right| \, dt & \leq  C \int_{\vert x\vert}^\infty t^{- \frac{d}{2}- \frac 12 +  \frac{\delta}{2}} \exp\left(- \frac{\vert x\vert^2}{Ct}\right) \, dt \\
                            & \leq  C \int_{0}^\infty t^{- \frac{d}{2}- \frac 12 +  \frac{\delta}{2}} \exp\left(- \frac{\vert x\vert^2}{Ct}\right) \, dt \notag \\
                            & \leq C |x |^{- 1 + \delta} |x|^{2 - d}. \notag
\end{align}
 To treat the first term in the right side of~\eqref{splitMpar1551}, we use the first estimate~\eqref{est.heatkernelgraph} of Proposition~\ref{proposition2.11}, which is recalled below, for each $t \in (0, \infty)$, $x \in \C_\infty$ such that $|x| \geq t$,
\begin{equation*}
    p(t , x , 0) \leq  C \exp \left( - C^{-1} |x | \left( 1 + \ln \frac{|x|}{t} \right) \right).
\end{equation*}
The same estimate is also valid for the function $\bar p$. Therefore, the term $\ln \left( |x| / t \right)$ is positive on the interval $\left(0, |x|\right]$ and one has the estimate
\begin{align*}
    \int_{0}^{\vert x\vert} \left|p(t , x , 0) - \theta(\p)^{-1} \bar p(t , x) \right| \, dt & \leq C \int_{0}^{\vert x \vert} \exp \left( - C^{-1} |x | \left( 1 + \ln \frac{|x|}{t} \right) \right) \, dt \\
                                    & \leq C \int_{0}^{\vert x\vert} \exp \left( - C^{-1} |x| \right) \, dt \\
                                    & \leq C \vert x\vert \exp \left( - C^{-1} |x| \right).
\end{align*}
By increasing the value of the constant $C$, one has
\begin{equation*}
    \int_{0}^{\vert x\vert} \left|p(t , x ,0) - \theta(\p)^{-1} \bar p(t , x ) \right| \, dt \leq C \exp \left( - C^{-1} |x| \right).
\end{equation*}
Combining the previous estimate with~\eqref{ellgreenest.larget}, we deduce
\begin{align*}
    \left| g(x) - \bar g (x ) \right| & \leq C |x |^{- 1 + \delta} |x|^{2 - d} + C \exp \left( - C^{-1} |x| \right) \\
                                        & \leq C |x|^{- 1 + \delta} |x|^{2 - d}.
\end{align*}
This completes the proof of the estimate~\eqref{eq:TV11522603} in dimension larger than $3$.

We now focus on the case of the dimension $2$. The strategy is similar, but some additional attention is needed due to the fact that the integral~\eqref{ellandaprgreen} is ill-defined in dimension $2$. We define the elliptic Green's function $g$ on the infinite cluster by the formula
\begin{equation} \label{ellandaprgreen=22}
 g(x) = \int_{0}^{\infty} \left(p(t , x , 0) - p(t ,0,0) \right) \, dt.
\end{equation}
For the homogenized Green's function, we cannot use the formula~\eqref{ellandaprgreen=22} by replacing the transition kernel $p$ by the homogenized heat kernel $\bar p$; indeed the integral
\begin{equation}
    \int_{0}^{\infty} \left(\bar p(t , x ) - \bar p(t ,0) \right) \, dt,
\end{equation}
 is ill-defined as soon as $x \neq y$ since the term $\bar p(t ,0)$ is of order $t^{-1}$ around $0$. To overcome this issue, we introduce the notation $ \left( \bar p (t , \cdot ) \right)_{B_1} := \fint_{B_1} \bar p (t , z ) \, dz$ and note that, for each $x \in \Rd$, the integral
 \begin{equation*}
     \int_{0}^\infty \left( \bar p (t , x) - \left( \bar p (t , \cdot) \right)_{B_1}\right) \, dt
\end{equation*}
is well-defined. Additionally, the function $$x \mapsto \theta(p)^{-1} \left(\int_{0}^\infty \bar p (t , x) - \left( \bar p (t , \cdot) \right)_{B_1} \, dt\right)$$ is equal to $\bar{g}$ up to a constant (see \cite[Chapter 1.8]{DR1984} for detailed discussions). We denote this constant by $K_1$, i.e., we write, for any $x \in \Rd \setminus \{0 \}$,
\begin{equation}\label{def.K1} 
K_1 :=  \theta(p)^{-1} \left(\int_{0}^\infty \bar p (t , x) - \left( \bar p (t , \cdot) \right)_{B_1} \, dt\right) - \bar g(x).
\end{equation}
We note that the value $K_1$ depends only on the diffusivity $\sigk$. Using these two integrals, we have 
\begin{equation*}
\begin{split}
g(x) - \bar g(x)  &=  \int_{0}^{\infty} \left(p(t , x , 0) - p(t ,0,0) \right) - \theta(p)^{-1}\left(\bar p (t , x) - \left( \bar p (t , \cdot) \right)_{B_1(y)}\right) \, dt  + K_1\\
&= \int_{0}^{\infty} \left(p(t , x , 0) - \theta(p)^{-1} \bar p (t , x) \right) \, dt + K_1 - K_2(0), 
\end{split}
\end{equation*}
where $K_2$ is defined by the formula
\begin{equation} \label{def.K2}
    K_2(0) := \int_{0}^{\infty}  p(t , 0 ,0) - \theta(p)^{-1} \left(\bar p (t , \cdot) \right)_{B_1}  \, dt.
\end{equation}
We now prove that this integral is well-defined, and that the constant $K_2$ satisfies the stochastic integrability estimate
\begin{equation*}
    \left| K_2(0) \right| \leq \O_s(C).
\end{equation*}
The proof relies on Theorem~\ref{mainthm} and on the estimates on the discrete heat kernel $p(t, 0 ,0) \leq 1$ and $\left( \bar p (t , \cdot) \right)_{B_1} \leq 1$ for all times $t$. We compute
\begin{align*}
    \left| K_2(0) \right| & \leq \int_{0}^{\infty} \left| p(t , 0 ,0) -  \theta(p)^{-1} \left( \bar p (t , \cdot) \right)_{B_1} \right| \, dt \\ & 
                                    \leq  \int_{0}^{\T_{\mathrm{par}, \delta}(0)} \left| p(t , 0 ,0) -  \theta(p)^{-1} \left( \bar p (t , \cdot)  \right)_{B_1}  \right| \, dt + \int_{\T_{\mathrm{par}, \delta}(0)}^\infty \left| p(t , 0 ,0) -  \theta(p)^{-1} \bar p (t , 0) \right| \, dt  \\ & \quad +  \theta(p)^{-1} \int_{\T_{\mathrm{par}, \delta}(0)}^\infty \left| \bar p(t , 0) - \left( \bar p (t , \cdot ) \right)_{B_1 } \right| \, dt   \\
                                            & \leq \int_{0}^{\T_{\mathrm{par}, \delta}(0)} C \, dt + \int_{\T_{\mathrm{par}, \delta}(0)}^\infty C t^{-\frac{3}{2} + \delta} \, dt + \int_{\T_{\mathrm{par}, \delta}(0)}^\infty  C t^{-\frac{3}{2}} \, dt \\
                                            & \leq C \T_{\mathrm{par}, \delta}(0) + C.
\end{align*}
This implies the estimate $\left|K_2(0)\right| \leq \O_s(C)$. We define $K(0) := K_1 - K_2(0)$, and by the previous computation, it satisfies the stochastic integrability estimate $\left| K(0) \right| \leq \O_s (C)$. 

To complete the proof Theorem~\ref{mainthmell} in dimension $2$, it is thus sufficient to control the term $\int_{0}^{\infty} \left(p(t , x , 0) - \theta(p)^{-1} \bar p (t , x \right) \, dt$; the argument is the same than in dimension larger than $3$ and the details are omitted.
\end{proof}

\appendix

\section{A concentration inequality for the density of the infinite cluster} \label{appendixB}

In this appendix, we study the density of the infinite cluster in a cube $\cu$, which is defined as the random variable $\frac{\left| \C_\infty \cap \cu \right|}{|\cu|}.$ As the size of the cube tends to infinity, an application of the ergodic theorem shows that this random variable converges, almost surely and in $L^1$, to the value $\theta (\p)$. The objective of the following proposition is to provide a quantitative version of this result.

\begin{proposition}\label{prop:DensityContrentration}
There exists a positive constant $C(d,\p) < \infty$ such that for any triadic cube $\cu \in \T$ of size $3^m$, one has an estimate
\begin{equation}\label{eq:DensityContrentration}
   \left| \frac{\left| \C_\infty \cap \cu \right|}{|\cu|} - \theta(\p) \right| \leq \O_{\frac{2(d-1)}{3d^2+2d-1}} \left( C 3^{-\frac{dm}{2}} \right).
\end{equation}
As a corollary, we obtain that, for any exponent $\alpha > 0$, there exist a positive constant $C(d, \p, \alpha) < \infty$, an exponent $s(d, \p, \alpha) > 0$, and a minimal scale $\M_{\mathrm{dense}, \alpha} \leq \O_s(C)$ such that for every $3^m \geq \M_{\mathrm{dense}, \delta}(y)$, 
\begin{equation}\label{eq:DensityContrentrationM}
   \left| \frac{\left| \C_\infty \cap \cu_m \right|}{|\cu_m|} - \theta(\p) \right| \leq 3^{-\left(\frac{d}{2}-\alpha\right)m}.
\end{equation}
\end{proposition}

\begin{remark}
The stochastic integrability exponent $\frac{2(d-1)}{3d^2+2d-1}$ in the estimate~\eqref{eq:DensityContrentration} is  suboptimal and we do not try to reach optimality. The spatial scaling is the one of the central limit theorem and is optimal. We note that a result of large deviation for the concentration of the density of the infinite cluster can be found in the article~\cite[Theorem 1.2]{P} of Pisztora: for any $\epsilon > 0$ and $\p > \p_c(d)$, there exist two constants $C_1(\p,d,\epsilon) < \infty, C_2(\p,d,\epsilon) < \infty$ such that for any cube $\cu$ of size $3^m$,
$$
\P\left( \left| \frac{\left| \C_\infty \cap \cu \right|}{|\cu|} - \theta(\p) \right| > \epsilon \right) \leq  C_1 \exp\left(-C_2 3^{(d-1)m}\right).
$$
However, this estimate cannot be used in the setting considered in this article since the dependence of the constants $C_1$ and $C_2$ in the variable $\epsilon$ is not explicit.
\end{remark}

We prove Proposition~\ref{eq:DensityContrentration} with an exponential version of the Efron-Stein inequality. A proof of this result can be found in \cite[Proposition 2.2]{armstrong2017optimal}. In the context of supercritical percolation, this inequality was used in~\cite[Proposition 2.18, Proposition 3.3]{dario2018optimal} to study the corrector and in~\cite[Proposition 3.2]{gu2019efficient} to study the flux. It is stated in the following proposition and we recall the notations introduced in Section~\ref{section1.6.1}: we denote by $(\Omega, \F, \P)$ the probability space and by $\F(\Bd \backslash \{e\})$ denotes the sigma algebra generated by the collection of random variables $\{\a(e')\}_{e' \in \Bd \backslash \{e\}}$.
\begin{proposition}[Exponential Efron-Stein inequality, Proposition 2.2 of \cite{armstrong2017optimal}]\label{prop:SpectralInequality} Fix an exponent $\beta \in (0,2)$ and let $X$ be a random variable defined on the probability space $(\Omega, \F, \P)$. We define the random variables
\begin{align}
X_e := \E\left[X|\F(\Bd \backslash \{e\})\right], &\qquad \mathbb{V}[X] := \sum_{e \in \Bd} (X - X_e)^2.
\end{align}
There exists a positive constant $C:=C(d, \beta) < \infty$ such that 
\begin{equation}\label{eq:EfronStein}
\E\left[\exp\left(|X - \E[X]|^{\beta}\right)\right] \leq C \E \left[\exp\left( (C \mathbb{V}[X])^{\frac{\beta}{2 - \beta}}\right) \right]^{\frac{2-\beta}{2}}.
\end{equation}
\end{proposition} 
We define $X := \frac{\vert \C_{\infty}  \cap \cu \vert}{\vert \cu \vert} - \theta(\p)$. To prove Proposition~\ref{prop:DensityContrentration}, it suffices to prove the two inequalities 
\begin{align}\label{eq:SpectralCondition}
\E[X] \leq C_1(\p, d)3^{-\frac{dm}{2}} \quad \text{and}\quad  \mathbb{V}[X] \leq \O_{s'}\left(C_2(d, \p)3^{-dm}\right),   
\end{align}
and to use the estimate~\eqref{eq:EfronStein} to deduce that $X \leq \O_s(C)$ with the exponent $s = \frac{2s'}{1 + s'}$. These two inequalities are natural since they mean that the bias and variance of the random variable satisfy the desired upper bounds. Since the random variable $X = \frac{1}{\vert \cu \vert}\sum_{x\in \cu} \left(\indc_{\{x \in \C_\infty\}} - \theta(\p)\right)$ is centered, we can focus on the term $\mathbb{V}[X]$. 

To estimate this term, we consider an independent copy of the environment $\a$ which we denote by $\tilde{\a}$ (and enlarge the underlying probability space to achieve this if necessary). Given a bond $e \in \Bd$, we define $\{\a^e(e')\}_{e' \in \Bd}$ ``the environment obtained by resampling the conductance at the bond $e$" by the formula
\begin{equation*}
    \a^e(e') = \left\{ \begin{array}{lcl} 
    \a(e') & ~\mbox{if}~e' \neq e, \\
    \tilde{\a}(e') &  ~\mbox{if}~ e' = e.
    \end{array} \right.
\end{equation*}
We denote by $X^e$ the random variable obtained by resampling the bond $e$, i.e., $X^e = X \left( \a^e \right)$. We also denote by $\C_\infty^e$ the infinite cluster under the environment $\a^e$.
We have the following implication
\begin{align}\label{eq:SpectralCondition2}
\sum_{e \in \Bd}(X^e - X)^2 \leq \O_{s'}(C3^{-dm}) \Longrightarrow \mathbb{V}[X] \leq \O_{s'}(C3^{-dm}),
\end{align}
whose proof can be found in \cite[Lemma 3.1]{gu2019efficient}. We note that since the two environments $\{\a(e')\}_{e' \in \Bd}$ and $\{\a^e(e')\}_{e' \in \Bd}$ are only different on one bond, the following statement holds $\P$-almost surely
\begin{equation*}
    \C^e_{\infty} \subseteq \C_\infty ~\mbox{or}~\C_\infty \subseteq  \C^e_{\infty}.
\end{equation*}
We have the following identity
$$
\vert X^e - X \vert = \frac{1}{\vert \cu \vert} \left\vert (\C^e_{\infty} \bigtriangleup \C_{\infty}) \cap \cu \right\vert,
$$
where $\C^e_{\infty} \bigtriangleup \C_{\infty} := (\C^e_{\infty} \setminus \C_{\infty}) \cup (\C_{\infty} \setminus \C^e_{\infty})$ denotes the symmetric difference between the two clusters $\C_{\infty}$ and $\C^e_{\infty}$. This suggests to study the properties of this quantity and we prove the following lemma.

\begin{lemma}\label{lem:ClusterSize}
The following estimates hold:
\begin{enumerate}
    \item There exists a positive constant $C(d, \p) < \infty$ such that
    \begin{align}\label{eq:ClusterSizeShort}
        \forall e \in \Bd, \qquad \left\vert \C^e_{\infty} \bigtriangleup \C_{\infty} \right\vert \leq \O_{\frac{d-1}{d}}(C).
    \end{align}
    \item There exists a positive constant $C(d, \p) < \infty$ such that 
        \begin{align}\label{eq:ClusterSizeLong}
        \forall e \in \Bd \setminus \Bd(3\cu), \qquad \left\vert (\C^e_{\infty} \bigtriangleup \C_{\infty}) \cap \cu \right\vert^2 \leq \O_{\frac{d-1}{(3d+1)d}}\left(\frac{C}{\dist(e, \cu)^{d+1}}\right),
    \end{align}
    where we recall the notation $3\cu$ introduced in~\eqref{eq:nonstnt}.
    As a corollary, we have that 
    \begin{align}\label{eq:ClusterSizeLongSum}
        \sum_{e \in \Bd \setminus \Bd(3\cu)} \left\vert (\C^e_{\infty} \bigtriangleup \C_{\infty}) \cap \cu \right\vert^2 \leq \O_{\frac{d-1}{(3d+1)d}}(C).
    \end{align}
\end{enumerate}
\end{lemma}

We first show how to obtain Proposition~\ref{prop:DensityContrentration} from Lemma~\ref{lem:ClusterSize} and then prove Lemma~\ref{lem:ClusterSize}.

\begin{proof}[Proof of Proposition~\ref{prop:DensityContrentration}]
The result is a consequence of the estimate~\eqref{eq:OSum} and Lemma~\ref{lem:ClusterSize}. We have
\begin{align*}
\sum_{e \in \Bd}(X^e - X)^2 &= \frac{1}{\vert \cu \vert^2}\sum_{e \in \Bd}  \left( \vert \C^e_{\infty}  \cap \cu \vert -  \vert \C_{\infty} \cap \cu \vert \right)^2 \\
&= 3^{-2dm}\underbrace{\sum_{e \in \Bd(3\cu)}  \left\vert (\C^e_{\infty} \bigtriangleup \C_{\infty}) \cap \cu \right\vert^2}_{\leq 3^{d(m+1)} \times \O_{\frac{d-1}{2d}}(C)}  + 3^{-2dm}\underbrace{\sum_{e \in \Bd \setminus \Bd(3\cu)} \left\vert (\C^e_{\infty} \bigtriangleup \C_{\infty}) \cap \cu \right\vert^2}_{\leq \O_{\frac{d-1}{(3d+1)d}}(C)}\\
& \leq \O_{\frac{d-1}{(3d+1)d}}\left(C 3^{-dm}\right).
\end{align*}
We then complete the proof by applying the implication~\eqref{eq:SpectralCondition2} and Proposition~\ref{prop:SpectralInequality} with the exponent $s = \frac{2(d-1)}{3d^2+2d-1}$.
\end{proof}

We now prove Lemma~\ref{lem:ClusterSize}. The argument relies on the upper and lower bounds on the tail of the distribution of the finite clusters in supercritical percolation. The result is stated below, was proved by Kesten and Zhang in~\cite{kesten1990} for the upper bound and by Aizenman, Delyon and Souillard in~\cite{aizenman1980} for the lower bound. We also refer to the monograph \cite[Section 8.6]{grimmett1999} for related discussions.
\begin{theorem}[Sub-exponential decay of cluster size distribution \cite{kesten1990, aizenman1980}] \label{th:THM6}
For any supercritical probability ${ \p \in (\pc(d) , 1]}$, there exist positive constants $0 < c_1(d, \p), c_2(d, \p) < \infty$ such that, if we denote by $\C(0)$ the cluster containing $0$ and let $n$ be a strictly positive integer, then we have the estimate
\begin{equation}\label{eq:ClusterSizeBound}
\forall n \in \mathbb{N}^+, \qquad \exp\left(-c_1 n^{\frac{d-1}{d}}\right) \leq \P\left[\vert \C(0) \vert = n\right] \leq \exp\left(-c_2 n^{\frac{d-1}{d}}\right).    
\end{equation}
\end{theorem}
\begin{remark}
With the notation $\O_s$, one can reformulate the upper bound as $\vert \C(0) \vert \leq \O_{\frac{d-1}{d}}(C)$.
\end{remark}

\begin{remark}
The estimate~\eqref{eq:ClusterSizeBound} implies the inequality $\P\left[n \leq \vert \C(0) \vert < \infty\right] \leq \exp\left(-c_3 n^{\frac{d-1}{d}}\right)$.
\end{remark}
\begin{proof}[Proof of Lemma~\ref{lem:ClusterSize}] 
We first prove the inequality~\eqref{eq:ClusterSizeShort}. We use the definition of $\O_s$ notation and prove the estimate $$\E\left[\exp\left(\left(\frac{\left\vert \C^e_{\infty} \bigtriangleup \C_{\infty} \right\vert}{C}\right)^{\frac{d-1}{d}}\right)\right] \leq 2,$$ for some constant $C(d, \p) < \infty$. By symmetry, it suffices to consider the case when $\{\a^e(e) > 0, \a(e) = 0\}$ and we have the identity
    \begin{align*}
        \E\left[\exp\left(\left(\frac{\left\vert \C^e_{\infty} \bigtriangleup \C_{\infty} \right\vert}{C}\right)^{\frac{d-1}{d}}\right)\right] = 1 + 2\E\left[\exp\left(\left(\frac{\left\vert \C^e_{\infty} \bigtriangleup \C_{\infty} \right\vert}{C}\right)^{\frac{d-1}{d}} \right)\indc_{\{\a^e(e) > 0, \a(e) = 0 \}}\right].
    \end{align*}
    We then notice that, under the condition $\{\a^e(e) > 0, \a(e) = 0\}$, we have the equality ${\C^e_{\infty} \bigtriangleup \C_{\infty} = \C^e_{\infty} \setminus \C_{\infty}}$. We then distinguish two cases:
    \begin{itemize}
    \item Either there exists a finite cluster connected to the bond $e$ in the environment $\{\a(e')\}_{e'\in\Bd}$. In that case, we denote this cluster by $\C(e)$ and we have the identity
    \begin{equation*}
        \C(e) = \C^e_{\infty} \bigtriangleup \C_{\infty};
    \end{equation*}
    \item Or both ends of the bond $e$ are connected to the infinite cluster $\C_\infty$ under the environment $\{\a(e')\}_{e'\in\Bd}$. In that case, we have the equality $\C^e_{\infty} \bigtriangleup \C_{\infty} = \emptyset$.
    \end{itemize}
    We then use Theorem~\ref{th:THM6} to estimate the volume of the cluster $\C(e)$ and we obtain
    \begin{equation*}
     \E\left[\exp\left(\left(\frac{\left\vert \C^e_{\infty} \bigtriangleup \C_{\infty} \right\vert}{C}\right)^{\frac{d-1}{d}}\right) \indc_{\{\a^e(e) > 0, \a(e) = 0 \}}\right] \leq \sum_{n = 0}^{\infty} \exp\left(\frac{n^{\frac{d-1}{d}}}{C^{\frac{d-1}{d}}}\right)  \exp\left(-c_1 n^{\frac{d-1}{d}}\right).
    \end{equation*}
    Then, we can choose a constant $C$ depending on the parameters $d$ and $\p$ such that $$\E\left[\exp\left(\frac{\left\vert \C^e_{\infty} \bigtriangleup \C_{\infty} \right\vert}{C}\right)^{\frac{d-1}{d}}\right] \leq 2.$$ This implies that $\left\vert \C^e_{\infty} \bigtriangleup \C_{\infty} \right\vert \leq \O_{\frac{d-1}{d}}(C)$.
    
\medskip

\begin{figure}[h!]
    \centering
    \includegraphics[scale = 0.7]{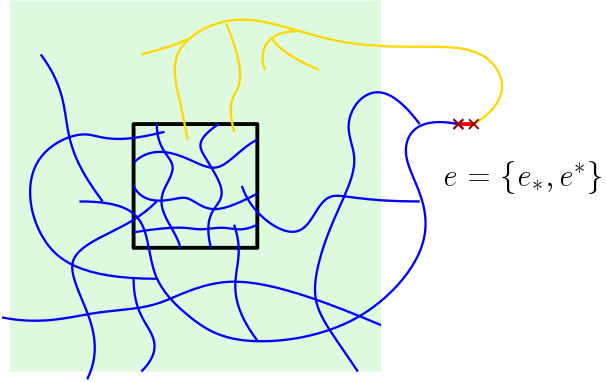}
    \caption{The figure illustrates the situation when the set $ (\C^e_{\infty} \bigtriangleup \C_{\infty}) \cap \cu$ is nonempty under the condition $\a^e(e) > 0$, $\a(e) =0$. The blue cluster is the infinite cluster $\C_\infty$ under the environment $\a$, and the yellow cluster is the finite cluster connecting the bond $e$ to the cube $\cu$. The green square represents the cube $3\cu$. The probability of the event depicted in the picture becomes exponentially small when the sizes of the cubes are large.}
    \label{fig:ClusterPerturbation}
\end{figure}

We now prove the estimate~\eqref{eq:ClusterSizeLong}. It relies on the following observation: when the bond $e$ is far away from the cube $\cu$, the set $(\C^e_{\infty} \bigtriangleup \C_{\infty}) \cap \cu$ is non-empty with exponentially small probability. More precisely, if we denote by $l = \dist(e, \cu)$, then we have the estimate
\begin{align}\label{eq:ObservationSize}
    \P\left[ (\C^e_{\infty} \bigtriangleup \C_{\infty}) \cap \cu \neq \emptyset \right] &= 2 \P\left[\C^e_{\infty} \bigtriangleup \C_{\infty} \neq \emptyset ~\mbox{and}~\C(e) \cap \cu \neq \emptyset ~\mbox{and}~ \a^e(e) > 0, \a(e) = 0 \right] \\ &\leq 2\P[ \C^e_{\infty} \bigtriangleup \C_{\infty} \neq \emptyset ~\mbox{and}~ \vert \C(e) \vert > l ~\mbox{and}~ \a^e(e) > 0, \a(e) = 0 ] \notag\\& 
    \leq 2\exp\left(-c_3 l^{\frac{d-1}{d}}\right). \notag
\end{align}
We also note that, since the bond $e$ lies outside the cube $3 \cu$, we have the estimate $l \geq 3^m$. This implies the almost sure inequalities
\begin{equation*}
    \left\vert (\C^e_{\infty} \bigtriangleup \C_{\infty}) \cap \cu \right\vert \leq 3^{dm} \leq l^d.   
\end{equation*}
Then, we can calculate the expectation
    \begin{align*}
    &\E\left[\exp\left(\left(\frac{ l^{d+1} \left\vert (\C^e_{\infty} \bigtriangleup \C_{\infty}) \cap \cu \right\vert^2 }{C}\right) ^{\frac{d-1}{(3d+1)d}}\right)\right] \\
    =& \E\left[\indc_{\{ (\C^e_{\infty} \bigtriangleup \C_{\infty}) \cap \cu = \emptyset\}}\right] + \E\left[\exp\left(\left(\frac{l^{d+1} \left\vert (\C^e_{\infty} \bigtriangleup \C_{\infty}) \cap \cu \right\vert^2}{C}\right) ^{\frac{d-1}{(3d+1)d}}\right)\indc_{\{(\C^e_{\infty} \bigtriangleup \C_{\infty}) \cap \cu \neq \emptyset\}}\right]\\    
    \leq & 1 + \exp\left(\left(\frac{ l^{3d+1}}{C}\right) ^{\frac{d-1}{(3d+1)d}}\right)\P[ (\C^e_{\infty} \bigtriangleup \C_{\infty}) \cap \cu\neq \emptyset]\\
    \leq & 1 + \exp\left(\frac{ l^{\frac{d-1}{d}}  }{C^{\frac{d-1}{(3d+1)d}}} \right)\P[ (\C^e_{\infty} \bigtriangleup \C_{\infty}) \cap \cu \neq \emptyset].
    \end{align*}
    We use the estimate~\eqref{eq:ObservationSize} and select a constant $C$ large enough such that
    \begin{equation*}
    \E\left[\exp\left(\left(\frac{l^{d+1}\left\vert (\C^e_{\infty} \bigtriangleup \C_{\infty}) \cap \cu \right\vert^2 }{C}\right) ^{\frac{d-1}{(3d+1)d}}\right)\right] \leq 1 +  2\exp\left(\frac{ l^{\frac{d-1}{d}}  }{C^{\frac{d-1}{(3d+1)d}}} \right)\exp\left(-c_4 l^{\frac{d-1}{d}}\right) \leq 2.
    \end{equation*}
    This completes the proof of the inequality~\eqref{eq:ClusterSizeLong}. The estimate~\eqref{eq:ClusterSizeLongSum}, is then a consequence of the inequality~\eqref{eq:OSum}, by noting that the sum $\sum_{e \in \Bd \setminus \Bd(3\cu)} \dist(e ,\cu)^{-d-1}$ is finite. Finally, we define 
    \begin{align*}
        \M_{\mathrm{dense}, \alpha} := \sup\left\{3^m \in \N : 3^{\left(\frac{d}{2} - \alpha\right)m }\left| \frac{\left| \C_\infty \cap \cu_m \right|}{|\cu_m|} - \theta(\p) \right| \geq 1\right\},
    \end{align*}
    and use Lemma~\ref{lemma1.5} to obtain that this random variable satisfies the stochastic integrability estimate $\M_{\mathrm{dense}, \alpha} \leq \O_s(C)$.
\end{proof}

We complete this section by stating and proving a version of the concentration estimate of Lemma~\ref{lem:ClusterSize}  involving the homogenized heat-kernel. This result is used in Lemma~\ref{l.lemma4.2}.
\begin{proposition}\label{prop:DensityContrentrationGaussian}
There exists a positive constant $C(d,\p) < \infty$ such that, for any time $t > 0$, and any vertex $y \in \Zd$, one has the estimate
\begin{equation}\label{eq:DensityContrentrationGaussian}
   \left| \int_{\C_\infty} \bar{p}(t, x- y) \, dx - \theta(\p) \right| \leq \O_{\frac{2(d-1)}{3d^2+2d-1}} \left( C t^{-\frac{1}{2}} \right).
\end{equation}
As a corollary, for any $\alpha > 0$ and $y \in \Zd$, there exist a positive constant $C(d, \p, \alpha) < \infty$, an exponent $s(d, \p, \alpha) > 0$, and a minimal time $\T_{\mathrm{dense}, \alpha}(y) \leq \O_s(C)$ such that, for every time $t > \T_{\mathrm{dense}, \alpha}(y)$, we have
\begin{equation}\label{eq:DensityContrentrationGaussianT}
   \left| \int_{\C_\infty} \bar{p}(t, x- y) \, dx - \theta(\p) \right| \leq  C t^{-\left(\frac{1}{2}-\alpha\right)}.
\end{equation}
\end{proposition}

\begin{proof}
Without loss of generality, we suppose that $y=0$. The strategy is similar to the one of the proof of Proposition~\ref{prop:DensityContrentration}. We denote by $X := \int_{\C_\infty} \bar{p}(t, x) \, dx - \theta(\p)$, apply the concentration inequality stated in Proposition~\ref{prop:SpectralInequality} and verify the two conditions~\eqref{eq:SpectralCondition} and \eqref{eq:SpectralCondition2}. For the term involving the expectation, we have 
\begin{align*}
\vert\E[X]\vert &= \left \vert \int_{\Zd} \bar{p}(t, x) \indc_{\{x \in \C_{\infty}\}} \, dx - \theta(\p) \right \vert = \theta(p)\left \vert \int_{\Zd} \bar{p}(t, x) \, dx - \int_{\Rd} \bar{p}(t, x) \, dx \right \vert \leq \frac{C}{\sqrt{t}},
\end{align*}
by the estimate on the gradient of the heat kernel. We then focus on the variation, i.e.,
\begin{align*}
\sum_{e \in \Bd}(X^e - X)^2 = \sum_{e \in \Bd}\left( \int_{\Zd} \bar{p}(t, x) \left(\indc_{\{x \in \C^e_{\infty}\}} - \indc_{\{x \in \C_{\infty}\}}\right) \, dx \right)^2.    
\end{align*}
We apply a multiscale analysis: we define the balls and annuli
\begin{align*}
    B_{-1} := \emptyset, \qquad \forall n \geq  1, \hspace{2mm} B_n := \{x \in \Zd : \vert x \vert \leq 3^n \sqrt{t}\}, \qquad \forall n \geq 0, \hspace{2mm} A_n := B_n \setminus B_{n-1}.
\end{align*}
We also define, for any subset $A \subset \Zd$, $I^e_{A} := \int_A  \bar{p}(t, x) \left(\indc_{\{x \in \C^e_{\infty}\}} - \indc_{\{x \in \C_{\infty}\}}\right) \, dx$.   
This notation is useful to localize the random variables $\left( Y^e - Y\right)$. We write
\begin{align*}
\sum_{e \in \Bd}(Y^e - Y)^2 = \sum_{e \in \Bd} \left(I^e_{\Zd}\right)^2 = \sum_{e \in \Bd}\left(\sum_{n=0}^{\infty}I^e_{A_n}\right)^2.    
\end{align*}
Then, we use the Cauchy-Schwarz inequality to factorize the sum 
\begin{align*}
\left(\sum_{n=0}^{\infty}I^e_{A_n}\right)^2 = \left(\sum_{n=0}^{\infty}I^e_{A_n}3^n \times 3^{-n}\right)^2 \leq \left(\sum_{n=0}^{\infty}3^{2n}\left(I^e_{A_n}\right)^2 \right) \left(\sum_{n=0}^{\infty}3^{-2n} \right) \leq 2 \sum_{n=0}^{\infty}3^{2n}\left(I^e_{A_n}\right)^2.    
\end{align*}
With Fubini's theorem, we obtain 
\begin{align}\label{eq:YLocalization}
\sum_{e \in \Bd}(Y^e - Y)^2 \leq 2 \sum_{n=0}^{\infty}3^{2n} \sum_{e \in \Bd}\left(I^e_{A_n}\right)^2.   
\end{align}
We fix an integer $n \in \N$ and estimate the quantity $\sum_{e \in \Bd}\left(I^e_{A_n}\right)^2$. The strategy is similar to the proof of Proposition~\ref{prop:DensityContrentration} and we adapt the proof of Lemma~\ref{lem:ClusterSize} from the case of cubes to the case of balls
\begin{align*}
\sum_{e \in \Bd}\left(I^e_{A_n}\right)^2 &=   \sum_{e \in \Bd(B_{n+1})}\left(I^e_{A_n}\right)^2 + \sum_{e \in \Bd \setminus \Bd(B_{n+1})}\left(I^e_{A_n}\right)^2 \\ 
& \leq \left(\max_{A_n}\bar{p}(t, x)\right)^2 \left(\sum_{e \in \Bd(B_{n+1})}  \vert ( \C^e_{\infty} \bigtriangleup \C_{\infty} ) \cap A_n \vert^2 + \sum_{e \in \Bd \setminus \Bd(B_{n+1})}\vert ( \C^e_{\infty} \bigtriangleup \C_{\infty} ) \cap A_n \vert^2  \right) \\
& \leq \frac{1}{(2\pi t \sigk)^{d/2}}\exp\left(- \frac{3^{2n}}{2\sigk}\right) \left(\sum_{e \in \Bd(B_{n+1})}  \vert ( \C^e_{\infty} \bigtriangleup \C_{\infty} ) \cap B_n \vert^2 + \sum_{e \in \Bd \setminus \Bd(B_{n+1})}\vert ( \C^e_{\infty} \bigtriangleup \C_{\infty} ) \cap B_n \vert^2 \right)\\
& \leq \O_{\frac{d-1}{(3d+1)d}}\left(C 3^{dn}t^{-\frac{d}{2}}\exp\left(-\frac{3^{2n}}{2\sigk}\right)\right).
\end{align*}
We put this inequality back in equation~\eqref{eq:YLocalization} and use the estimate~\eqref{eq:OSum} to conclude
\begin{align*}
\sum_{e \in \Bd}(Y^e - Y)^2   \leq \O_{\frac{d-1}{(3d+1)d}}\left(C t^{-\frac{d}{2}}  \left(\sum_{n=0}^{\infty}3^{(d+2)n} \exp\left(-\frac{3^{2n}}{2\sigk}\right)\right)\right) \leq   \O_{\frac{d-1}{(3d+1)d}}\left(C t^{-\frac{d}{2}}\right).
\end{align*}
Finally, for any exponent $\alpha >0$, we define 
\begin{align*}
\T_{\mathrm{dense}, \alpha}(0) :=  \sup\left\{t \in (0, \infty) :  t^{\left(\frac{1}{2}-\alpha\right)} \left| \int_{\C_\infty} \bar{p}(t, x) \, dx - \theta(\p) \right| \geq 1 \right\},   
\end{align*}
and apply Lemma~\ref{lemma1.5} to conclude Proposition~\ref{prop:DensityContrentrationGaussian}.
\end{proof}

\section{Quantification of the weak norm of the flux on the infinite cluster} \label{appendixb}
In this appendix, we prove a quantification of the $\aH^{-1}$-norm of the flux on the cluster. We recall that the flux on the cluster associated to the direction $e_k$ is defined by 
\begin{equation} \label{def.tiltedflux}
\tilde{\mathbf{g}}_{e_k} : \C_{\infty} \rightarrow \Rd, \qquad \tilde{\mathbf{g}}_{e_k} = \a(\D \chi_{e_k} + e_k) - \dsigk e_k.
\end{equation}
The main estimate is stated in the following proposition.
\begin{proposition}\label{prop:WeakNormFlux}
Fix a point $y \in \Zd$, for each exponent $\alpha >0$, there exist a positive constant ${C:= C(\lambda, d, \p, \alpha) < \infty}$, an exponent $s := s(\lambda, d, \p, \alpha) > 0$, and a random variable $\M_{\mathrm{flux}, \alpha}(y)$ satisfying the stochastic integrability estimate
\begin{equation*}
\M_{\mathrm{flux}, \alpha}(y) \leq \O_s\left( C \right),
\end{equation*}
 such that, for every radius $r \geq \M_{\mathrm{flux}, \alpha}(y)$, one has
\begin{equation}\label{eq:WeakNormFlux}
 \sum_{k=1}^d \norm{\tilde{\mathbf{g}}_{e_k}}_{\aH^{-1}(\C_\infty \cap B_r(y))} \leq 
 C r^{ \alpha}.
\end{equation}
\end{proposition} 
Without loss of generality, we assume $y = 0$. The strategy of the proof is to make use of another centered flux defined on the entire space $\Zd$,
$$
\mathbf{g}_{e_k} : \Zd \rightarrow \Rd, \qquad \mathbf{g}_{e_k} = \a (\D \chi_{e_k} + e_k) - \ahom e_k.
$$
 The homogenized conductance $\ahom$ is defined in~\cite[Definition 5.1]{AD2} by the formula: for each $p \in \Rd$,
\begin{equation*}
\frac 12 p \cdot \ahom p := \lim_{n \rightarrow \infty} \E \left[ \nu \left( \cu_m , p \right) \right],
\end{equation*}
where the energy $ \nu \left( \cu_m , p \right)$ is defined by
\begin{equation} \label{def.nuapp}
 \nu \left( \cu_m , p \right) := \inf_{u \in l_p + C_0 \left( \C_\infty \cap \cu_m \right)} \frac{1}{2|\cu_m|}\int_{\cu_m} \nabla u \cdot \a \nabla u,
\end{equation}
where the notation $l_p$ denotes the affine function of slope $p$ (i.e., for each point $x \in \Zd$, $l_p(x) = p\cdot x$) and the symbol $C_0 \left( \C_\infty \cap \cu_m \right)$ denotes the set of functions defined on the set $\C_\infty \cap \cu_m$, valued in $\R$, which are equal to $0$ on the boundary $\C_\infty \cap \partial \cu_m$.
The reason we introduce this quantity is that, building upon the results of~\cite{AD2}, we can prove the following $\underline{H}^{-1}$-estimate: there exists a non-negative random variable $\M_{\mathrm{flux-}\Zd, \alpha}$ satisfying the stochastic integrability estimate $\M_{\mathrm{flux-}\Zd, \alpha} \leq \O_s \left( C \right)$ such that, for every $r \geq \M_{\mathrm{flux-}\Zd, \alpha}$,
\begin{equation} \label{eq.gsmallH-1}
\sum_{k=1}^d \left\| \mathbf{g}_{e_k} \right\|_{\underline{H}^{-1}\left( \Zd \cap B_r \right)} \leq C r^{\alpha},
\end{equation}
where the $\underline{H}^{-1}(\Zd \cap B_r)$-norm is defined by the formula
$$
\norm{\mathbf{g}_{e_k}}_{\aH^{-1}(\Zd \cap B_r)} = \sup_{\norm{\varphi}_{\aH^1(\Zd \cap B_r)} \leq 1} \frac{1}{\vert \Zd \cap B_r \vert} \int_{\Zd \cap B_r} \varphi \mathbf{g}_{e_k},
$$
and where the $\aH^1( \Zd \cap B_r)$-norm of a function $\varphi : \Zd \cap B_r \to \R$, denotes the discrete normalized Sobolev norm defined by the formula
$$
\norm{\varphi}_{\aH^1( \Zd \cap B_r)}^2 := r^{-1} \left\| \varphi \right\|_{\aL^2(\Zd \cap B_r)} + \left\| \nabla \varphi \right\|_{\aL^2(\Zd \cap B_r)}.
$$
Once this result is established, we set the value
\begin{equation} \label{def.sigkB1}
\sigk := 2 \theta(\p)^{-1} \ahom,
\end{equation}
and deduce Proposition~\ref{prop:WeakNormFlux} from the estimate. The main difference between the estimates~\eqref{eq:WeakNormFlux} and~\eqref{eq.gsmallH-1} is that in the former estimate, the $\aH^{-1}$-norm is computed on the ball $B_r$ while in the latter is computed on the intersection $\C_\infty \cap B_r$. This makes an important difference and motivates the introduction of the diffusivity $\sigk$ in~\eqref{def.sigkB1} and of the new flux $\tilde{\mathbf{g}}_{e_k}$ in~\eqref{def.tiltedflux}. In the following paragraph, we give an heuristic argument explaining why we expect Proposition~\ref{prop:WeakNormFlux} to hold assuming that the estimate~\eqref{eq.gsmallH-1} is valid.

We start by using the constant test function equal to $1$ in the definition of the $\underline{H}^{-1}(B_r)$ norm in the estimate~\eqref{eq.gsmallH-1} shows
\begin{equation}\label{eq:gtest}
\int_{ \Zd \cap B_r} \a(\D\chi_{e_k} + e_k) \simeq \int_{\Zd \cap B_r} \ahom e_k,
\end{equation}
where the symbol $\simeq$ means that the two quantities on the left and right sides differ by a small term, which by~\eqref{eq.gsmallH-1} is of order~$r^{-(1-\alpha)}$. Since the function $\a(\D\chi_{e_k} + e_k)$ is defined to be equal to $0$ outside the infinite cluster, the left side of~\eqref{eq:gtest} can be rewritten
$$
\int_{ \Zd \cap B_r} \a(\D\chi_{e_k} + e_k) = \int_{\C_\infty \cap B_r} \a(\D\chi_{e_k} + e_k).
$$
For the right-hand side of~\eqref{eq:gtest}, using that the density of the cluster has density $\theta(\p)$, one expects
$$
\int_{\Zd \cap B_r} \ahom e_k \simeq\int_{\C_\infty \cap B_r} \theta(\p)^{-1} \ahom e_k.
$$
This shows
$$
 \int_{\C_\infty \cap B_r} \a(\D\chi_{e_k} + e_k) \simeq \int_{\C_\infty \cap B_r} \theta(\p)^{-1} \ahom e_k.
$$
Thus if we want the estimate~\eqref{eq:WeakNormFlux} to hold, the only admissible value for the coefficient $\sigk$ is~$ 2 \theta(\p)^{-1} \ahom$, indeed testing the constant function equal to $1$ in the definition of the $\underline{H}^{-1}(\C_\infty \cap B_r)$-norm of the estimate~\eqref{eq:WeakNormFlux} shows
$$
 \int_{\C_\infty \cap B_r} \a(\D\chi_{e_k} + e_k) \simeq \int_{\C_\infty \cap B_r} \frac 12 \sigk e_k.
$$

\begin{remark}\label{rmk:sigk}
 We note that the identity~\eqref{def.sigkB1} is the definition of the diffusivity $\sigk$ used in this article: thanks to this definition and the result of Proposition~\ref{prop:WeakNormFlux}, we are able to prove Theorem~\ref{mainthm}, and then to recover the invariance principle stated in~\eqref{intro.inv.princ}.
\end{remark}

The rest of this section is organized as follows. We first explain how to prove the estimate~\eqref{eq.gsmallH-1} by using the results of~\cite{gu2019efficient} and the strategies of stochastic homogenization in the uniformly elliptic setting presented in~\cite{armstrong2017quantitative}. We then show how to deduce Proposition~\ref{prop:WeakNormFlux} from the inequality~\eqref{eq.gsmallH-1}.

\begin{proof}[Proof of the estimate \eqref{eq.gsmallH-1}]
We first extend the function $\mathbf{g}_{e_k}$ from $\Zd$ to $\Rd$ and let $[\mathbf{g}_{e_k}]$ be the function defined on $\Rd$, which is equal to $\mathbf{g}_{e_k}$ on $\Zd$ and which is piecewise constant on the unit cubes $z + \left[ -\frac 12, \frac 12 \right)^d$. We have the identity $\norm{\mathbf{g}_{e_k}}_{\aH^{-1}(\Zd \cap B_r)} \simeq C\norm{[\mathbf{g}_{e_k}]}_{\aH^{-1}(B_r)}$ up to a constant $C$ depending only on the dimension, where $\aH^{-1}(B_r)$ is the standard Sobolev norm. We then want to control the continuous $\aH^{-1}(B_r)$ norm of $[\mathbf{g}_{e_k}]$. The strategy is to apply the multiscale Poincar\'e inequality stated in~\cite[Remark~D.6, equation~(D.28)]{armstrong2017quantitative}. Its rescaled version reads
\begin{equation} \label{e.mulstacejnf}
\norm{[\mathbf{g}_{e_k}]}_{\aH^{-1}(B_r))} \leq C r \left(\int_{0}^{1}  \left(\int_{\Rd} r^{-d} e^{-\frac{\vert x \vert}{r}} \vert\Phi_{r^2t} \star [\mathbf{g}_{e_k}]\vert^2(x) \,dx  \right)\, dt  \right)^{\frac{1}{2}},
\end{equation}
where the function $\Phi_{t}$ is the standard heat kernel defined by $\Phi_{t} := \frac{1}{(2 \pi t)^{\frac d2}} \exp \left( - \frac{|x|^2}{2t} \right)$ and the operator~$\star$ is the standard convolution on $\Rd$. We then apply the following results:
\begin{itemize}
\item The spatial average of the flux decays: one has the estimate, for each $t >0$,
\begin{equation}\label{Algoperco.est.1675}
\vert\Phi_{t} \star [\mathbf{g}_{e_k}]\vert \leq \O_s\left(C t^{- \frac{d}{4}}\right).
\end{equation}
A proof of this result can be found in \cite[Section~3.1, Proposition~1.1]{gu2019efficient}.
\item The flux is essentially bounded: one has the estimate, for each $t > 0$,
\begin{equation}\label{fluxbounded.est}
\vert\Phi_{t} \star [\mathbf{g}_{e_k}]\vert \leq \O_s(C ),
\end{equation}
To prove this estimate, we first note that the bound on the corrector stated in~\eqref{est:optcorr1602} imply the following Lipschitz estimate on the corrector (by choosing $x$ and $y$ to be two neighboring points): for each vector $p \in B_1$, and each edge $e \in \Bd$,
\begin{equation} \label{eq:lip.corr}
\left| \nabla \chi_p(e) \right| \leq \O_s \left( C \vert p \vert \right).
\end{equation}
This estimate is also stated in~\cite[Remark 1.1]{AD2}. The inequality~\eqref{fluxbounded.est} is then a consequence of the estimate~\eqref{eq:lip.corr} and the property~\eqref{eq:OSum} of the $\O_s$ notation.
\end{itemize}

We then truncate the integral in the right side of~\eqref{e.mulstacejnf} at the value $t = r^{- 2}$ and obtain
\begin{multline*}
\norm{\mathbf{g}_{e_k}}_{\aH^{-1}(\Zd \cap B_r)} \leq C r \left(\int_{0}^{r^{-2}}  \left(\int_{\Rd} r^{-d}e^{-\frac{\vert x \vert}{r}} \vert\Phi_{r^2t} \star [\mathbf{g}_{e_k}]\vert^2(x) \,dx   \right)\, dt  \right)^{\frac{1}{2}} \\ + Cr \left(\int_{r^{-2}}^{1}  \left(\int_{\Rd} r^{-d} e^{-\frac{\vert x \vert}{r}} \vert\Phi_{r^2t} \star [\mathbf{g}_{e_k}]\vert^2(x) \,dx   \right)\, dt  \right)^{\frac{1}{2}}.
\end{multline*}
To estimate the first term in the right side, we apply the estimate~\eqref{fluxbounded.est} and, to estimate the second term, we apply the estimate~\eqref{Algoperco.est.1675}. Together with the property~\eqref{eq:OSum} of the $\O_s$ notation, this gives
\begin{equation} \label{mainstep1propb1}
\norm{\mathbf{g}_{e_k}}_{\aH^{-1}(\Zd \cap B_r)} \leq \left\{ \begin{array}{lcl} 
    \O_s(C) & ~\mbox{if}~ d \geq 3, \\
    \O_s\left(\log^{\frac{1}{2}}(1+r)\right) & ~\mbox{if}~ d =2.
    \end{array} \right.
\end{equation}
Finally, for every exponent $\alpha > 0$, we set 
\begin{align*}
\M_{\mathrm{flux-}\Zd, \alpha} := \sup \left\{r \in \R^+ :  r^{-\alpha}\sum_{k=1}^d\norm{\mathbf{g}_{e_k}}_{\aH^{-1}(\Zd \cap B_r)} \geq 1\right\},     
\end{align*}
and apply Lemma~\ref{lemma1.5}.
This completes the proof of the estimate~\eqref{eq.gsmallH-1}.
\end{proof}

\begin{proof}[Proof of Proposition~\ref{prop:WeakNormFlux}]
We fix an exponent $\alpha > 0$. We define the exponent $q :=  \max \left(\frac{(2d-1)}\alpha , 2d\right)$ and split the proof into 3 steps.

\medskip

\textit{Step 1.} In this step, we establish the inequality, for any radius $r \geq \M_{q} \left( \Pa\right)$,
\begin{equation}\label{eq:FluxDecom}
\norm{\tilde{\mathbf{g}}_{e_k}}_{\aH^{-1}(\C_{\infty} \cap B_r)} \leq C r^{\frac{\alpha}{2}} \left(\norm{\mathbf{g}_{e_k}}_{\aH^{-1}(\Zd \cap B_r)} + \norm{\dsigk e_k(\indc_{\C_\infty} - \theta(\p)) }_{\aH^{-1}(\Zd \cap B_r)}\right),
\end{equation}
where $C$ is a constant depending only on the parameters $\lambda, d, \p$. We recall the definition of the $\underline{H}^{-1}$-norm on the infinite cluster
$$
\norm{\tilde{\mathbf{g}}_{e_k}}_{\aH^{-1}(\C_\infty \cap B_r)} = \sup_{\norm{\varphi}_{\aH^1(\C_\infty \cap B_r)} \leq 1} \frac{1}{\vert \C_\infty \cap B_r \vert} \int_{\C_\infty \cap B_r} \varphi \tilde{\mathbf{g}}_{e_k}.
$$
We fix a function $\varphi : \C_\infty \cap B_r \to \R$ such that $\norm{\varphi}_{\aH^1(\C_\infty \cap B_r)} \leq 1$. The main idea is to extend the function $\varphi$ from the infinite cluster to $\Zd$. To this end, we use the coarsened function $\left[ \varphi  \right]_{\Pa}$ introduced in Section~\ref{sectionPGC}.
We extend the function $\tilde{\mathbf{g}}_{e_k}$ by $0$ outside the infinite cluster so that we have
$$\int_{\C_\infty \cap B_r} \varphi \tilde{\mathbf{g}}_{e_k} = \int_{\Zd \cap B_r} [\varphi]_{ \Pa} \tilde{\mathbf{g}}_{e_k}.$$
Since the radius $r$ is assumed to be larger than the minimal scale $\M_{q}\left( \Pa \right)$, the ratio $\frac{\vert \Zd \cap B_r \vert}{\vert \C_{\infty} \cap B_r \vert}$ is bounded from above by a constant $C(d, \p)$. Then, we compute
\begin{align} \label{eq:2149fr}
\lefteqn{\frac{1}{\vert \C_{\infty} \cap B_r \vert} \int_{\C_\infty \cap B_r} \varphi \tilde{\mathbf{g}}_{e_k} } \qquad & \\ \notag &= \frac{1}{\vert \C_\infty \cap B_r \vert} \int_{\Zd \cap B_r} [\varphi]_{ \Pa} \left(\tilde{\mathbf{g}}_{e_k} - \mathbf{g}_{e_k}\right) + \frac{1}{\vert \C_\infty \cap B_r \vert} \int_{\Zd \cap B_r} [\varphi]_{\Pa} \mathbf{g}_{e_k} \\
& \leq \left(\frac{\vert \Zd \cap B_r \vert}{\vert \C_{\infty} \cap B_r \vert} \right)\norm{[\varphi]_{\Pa}}_{\aH^{1}( \Zd \cap B_r)} \left(\norm{\tilde{\mathbf{g}}_{e_k} - \mathbf{g}_{e_k}}_{\aH^{-1}( \Zd \cap B_r)} + \norm{\mathbf{g}_{e_k}}_{\aH^{-1}( \Zd \cap B_r)}\right) \notag\\
& \leq C(d,\p) \norm{[\varphi]_{ \Pa}}_{\aH^{1}(\Zd \cap B_r)} \left(\norm{\dsigk e_k(\indc_{\C_\infty} - \theta(\p)) }_{\aH^{-1}( \Zd \cap B_r)} + \norm{\mathbf{g}_{e_k}}_{\aH^{-1}( \Zd \cap B_r)}\right), \notag
\end{align}
where we used the equation $\ahom = \frac{1}{2}\theta(\p)\sigk$ to go from the second line to the third line. We then use the estimates~\eqref{est.nablacoarsenL2} and~\eqref{eq:estcoarseminusu} to estimate the $\aH^1$-norm of the coarsened function $\varphi$ in terms of the $\aH^1$-norm of the function $\varphi$, and the assumption $r \geq \M_{q}(\Pa)$ to estimate the size of the cubes of the partition. This gives
$$
\norm{[\varphi]_{ \Pa}}_{\aH^{1}(\Zd \cap B_r)} \leq C r^{\frac{\alpha}{2}} \norm{\varphi}_{\aH^{1}(\C_\infty \cap B_r)}.
$$
Combining the previous estimate with the inequality~\eqref{eq:2149fr} completes the proof of Step 1.

\medskip

\textit{Step 2: Control over the quantity $\norm{\sigk e_k(\indc_{\C_\infty} - \theta(\p)) }_{\aH^{-1}( B_r)}$.}  
We let $\cu_m$ be the triadic cube such that $\cu_{m-1} \subset B_r \subset \cu_m$. We note that 
$$
\norm{\sigk e_k(\indc_{\C_\infty} - \theta(\p)) }_{\aH^{-1}(\Zd \cap B_r)} \leq C\norm{\sigk e_k(\indc_{\C_\infty} - \theta(\p)) }_{\aH^{-1}(\cu_m)},
$$ 
where the constant $C$ depends only on the dimension $d$. We apply another version of the multiscale Poincar\'e inequality, which is stated in~\cite[Proposition 1.7]{armstrong2017quantitative} (in the continuous setting, the extension to the discrete setting considered here does not affect the proof) and reads
\begin{align*}
\begin{split}
\norm{\sigk e_k(\indc_{\C_\infty} - \theta(\p)) }_{\aH^{-1}(\cu_m)} \leq C \sum_{n = 0}^{m-1} 3^n \left( \frac{1}{\vert 3^n \Zd \cap \cu_m \vert} \sum_{y \in 3^n \Zd \cap \cu_m} \bar{\sigma}^4 (\indc_{\C_\infty} - \theta(\p))^{2}_{y + \cu_n} \right)^{\frac{1}{2}},
\end{split}
\end{align*}
where we recall the notation $(f)_{y + \cu_n} = \frac1{|\cu_n|} \sum_{x \in y + \cu_n} f(x)$.
We apply Proposition~\ref{prop:DensityContrentration}
\begin{align*}
\left(\indc_{\C_\infty} - \theta(\p)\right)_{y+\cu_n} \leq \O_s \left(C3^{-\frac{dn}2}\right).     
\end{align*}
Using that the dimension is larger than $2$ and the property~\eqref{eq:OSum} of the $\O_s$ notation, we obtain
\begin{equation*}
\norm{\sigk e_k(\indc_{\C_\infty} - \theta(\p)) }_{\aH^{-1}(\cu_m)} \left\{ \begin{array}{lcl} 
    \O_s(C) & ~\mbox{if}~ d \geq 3, \\
    \O_s(C m) & ~\mbox{if}~ d =2.
    \end{array} \right.
\end{equation*}
We then apply Lemma~\ref{lemma1.5} to the collection of random variables
\begin{equation*}
    X_m := 3^{ -  \frac{\alpha m}2 }\norm{\sigk e_k(\indc_{\C_\infty} - \theta(\p)) }_{\aH^{-1}(\cu_m)},
\end{equation*}
to construct a minimal scale $\M_{\mathrm{cluster}, \frac{\alpha}{2}}$ such that, for any radius $r\geq \M_{\mathrm{cluster}, \frac{\alpha}{2}}$,
\begin{equation}\label{eq:DensityWeakNorm}
\norm{\sigk e_k(\indc_{\C_\infty} - \theta(\p)) }_{\aH^{-1}(\Zd \cap B_r)} \leq C r^{\frac{\alpha}2}.
\end{equation}

\medskip

\textit{Step 3: The conclusion.} We let $ \M_{\mathrm{flux-} \Zd, \frac{\alpha}{2}}$ be the minimal scale provided by equation~\eqref{eq.gsmallH-1} with the exponent $\frac \alpha 2$. We define the random variable $\M_{\mathrm{flux}, \alpha}(0)$ according to the formula
\begin{equation*}
    \M_{\mathrm{flux}, \alpha}(0) := \max \left( \M_{\mathrm{cluster},  \frac{\alpha}{2}}, \M_{\mathrm{flux-} \Zd, \frac{\alpha}{2}}, \M_{q}(\Pa) \right).
\end{equation*}
Combining the main results~\eqref{eq:FluxDecom} of Step 1 and~\eqref{eq:DensityWeakNorm} of Step 2 shows, for any $r \geq \M_{\mathrm{flux},\alpha}(0)$,
\begin{align*}
    \norm{\tilde{\mathbf{g}}_{e_k}}_{\aH^{-1}(\C_\infty \cap B_r)} & \leq C r^{\frac{\alpha}{2}} \left(\norm{\mathbf{g}_{e_k}}_{\aH^{-1}(\Zd \cap B_r)} + \norm{\dsigk e_k(\indc_{\C_\infty} - \theta(\p)) }_{\aH^{-1}(\Zd \cap B_r)}\right) \\
                                    & \leq C r^{\frac{\alpha}{2}} \left(r^{ \frac{\alpha}{2}} +r^{\frac{\alpha}{2}}\right) \\
                                    & \leq C r^{\alpha} .
\end{align*}
Thus, we obtain the main result~\eqref{eq:WeakNormFlux} of Proposition~\ref{prop:WeakNormFlux}.
\end{proof}

\begin{remark}\label{rmk:Flux}
One result used in this article is a variation of Proposition~\ref{prop:WeakNormFlux}. We are interested in another function $\tilde{\mathbf{g}}^*_{e_k} : \C_{\infty} \to \Rd$, satisfying the identity, for each function $u : \C_{\infty} \rightarrow \mathbb{R}$, 
\begin{equation}\label{eq:FLuxConjugate}
\D^* \cdot \left( u \left(\a \left( \D \chi_{e_k} + e_k \right) - \dsigk e_k \right) \right) = \D^*u  \cdot \tilde{\mathbf{g}}^{*}_{e_k}.
\end{equation}
One can check that the quantity $\tilde{\mathbf{g}}^{*}_{e_k}$ is different from $\tilde{\mathbf{g}}_{e_k}$ with an exact formula
\begin{equation*}
    \tilde{\mathbf{g}}^*_{e_k} := \begin{pmatrix}
 T_{- e_1} \left[\a \left( \D \chi_{e_k} + e_k \right) - \dsigk e_k \right]_1 \\[3mm]
\vdots \\[3mm]
 T_{- e_d} \left[\a \left( \D \chi_{e_k} + e_k \right) - \dsigk e_k \right]_d \\
\end{pmatrix},
\end{equation*}
where the (minor) difference comes from the translation when applying the finite difference operator. The $\underline{H}^{-1}$-norm of the function $ \tilde{\mathbf{g}}^*_{e_k}$ can also be controlled and one has the following property: for any exponent $\alpha > 0$, any vertex $y \in \Zd$, and any radius $r \geq \M_{\mathrm{flux}, \alpha}(y)$, one has the estimate
\begin{equation*}
 \sum_{k=1}^d \norm{\tilde{\mathbf{g}}^*_{e_k}}_{\aH^{-1}(\C_\infty \cap B_r(y))} \leq C r^{ \alpha}.
\end{equation*}
The proof is identical to the proof of Proposition~\ref{prop:WeakNormFlux} and the details are left to the reader.
\end{remark}

\bibliographystyle{abbrv}
\bibliography{holes}

\newcommand{\noop}[1]{} \def\cprime{$'$}
\begin{thebibliography}{10}

\bibitem{aizenman1980}
M.~Aizenman, F.~Delyon, and B.~Souillard.
\newblock Lower bounds on the cluster size distribution.
\newblock {\em J. Statist. Phys.}, 23(3):267--280, 1980.

\bibitem{AS18}
C.~Alves and A.~Sapozhnikov.
\newblock Decoupling inequalities and supercritical percolation for the vacant
  set of random walk loop soup.
\newblock {\em arXiv preprint arXiv:1808.01277}, 2018.

\bibitem{ABDH}
S.~Andres, M.~T. Barlow, J.-D. Deuschel, and B.~M. Hambly.
\newblock Invariance principle for the random conductance model.
\newblock {\em Probab. Theory Related Fields}, 156(3-4):535--580, 2013.

\bibitem{ACDS}
S.~Andres, A.~Chiarini, J.-D. Deuschel, and M.~Slowik.
\newblock Quenched invariance principle for random walks with time-dependent
  ergodic degenerate weights.
\newblock {\em Ann. Probab.}, 46(1):302--336, 2018.

\bibitem{andres2015invariance}
S.~Andres, J.-D. Deuschel, and M.~Slowik.
\newblock Invariance principle for the random conductance model in a degenerate
  ergodic environment.
\newblock {\em Ann. Probab.}, 43(4):1866--1891, 2015.

\bibitem{ADS16}
S.~Andres, J.-D. Deuschel, and M.~Slowik.
\newblock Harnack inequalities on weighted graphs and some applications to the
  random conductance model.
\newblock {\em Probab. Theory Related Fields}, 164(3-4):931--977, 2016.

\bibitem{andres2016heat}
S.~Andres, J.-D. Deuschel, and M.~Slowik.
\newblock Heat kernel estimates for random walks with degenerate weights.
\newblock {\em Electron. J. Probab.}, 21, 2016.

\bibitem{andres2018green}
S.~Andres, J.-D. Deuschel, and M.~Slowik.
\newblock Green kernel asymptotics for two-dimensional random walks under
  random conductances.
\newblock {\em arXiv preprint arXiv:1808.08126}, 2018.

\bibitem{ADS19}
S.~Andres, J.-D. Deuschel, and M.~Slowik.
\newblock Heat kernel estimates and intrinsic metric for random walks with
  general speed measure under degenerate conductances.
\newblock {\em Electron. Commun. Probab.}, 24:Paper No. 5, 17, 2019.

\bibitem{AP}
P.~Antal and A.~Pisztora.
\newblock On the chemical distance for supercritical {B}ernoulli percolation.
\newblock {\em Ann. Probab.}, 24(2):1036--1048, 1996.

\bibitem{ABM17}
S.~Armstrong, A.~Bordas, and J.-C. Mourrat.
\newblock Quantitative stochastic homogenization and regularity theory of
  parabolic equations.
\newblock {\em Anal. PDE}, 11(8):1945--2014, 2018.

\bibitem{AD2}
S.~Armstrong and P.~Dario.
\newblock Elliptic regularity and quantitative homogenization on percolation
  clusters.
\newblock {\em Comm. Pure and Appl. Math.}, 71(9):1717--1849, 2018.

\bibitem{AKM1}
S.~Armstrong, T.~Kuusi, and J.-C. Mourrat.
\newblock Mesoscopic higher regularity and subadditivity in elliptic
  homogenization.
\newblock {\em Comm. Math. Phys.}, 347(2):315--361, 2016.

\bibitem{AKM2}
S.~Armstrong, T.~Kuusi, and J.-C. Mourrat.
\newblock The additive structure of elliptic homogenization.
\newblock {\em Invent. Math.}, 208(3):999--1154, 2017.

\bibitem{armstrong2017quantitative}
S.~Armstrong, T.~Kuusi, and J.-C. Mourrat.
\newblock {\em Quantitative stochastic homogenization and large-scale
  regularity}, volume 352 of {\em Grundlehren der mathematischen
  Wissenschaften}.
\newblock Springer International Publishing, 2019.

\bibitem{armstrong2017optimal}
S.~Armstrong and J.~Lin.
\newblock Optimal quantitative estimates in stochastic homogenization for
  elliptic equations in nondivergence form.
\newblock {\em Arch. Ration. Mech. Anal.}, 225(2):937--991, 2017.

\bibitem{AS}
S.~N. Armstrong and C.~K. Smart.
\newblock Quantitative stochastic homogenization of convex integral
  functionals.
\newblock {\em Ann. Sci. \'Ec. Norm. Sup\'er. (4)}, 49(2):423--481, 2016.

\bibitem{Ar67}
D.~G. Aronson.
\newblock Bounds for the fundamental solution of a parabolic equation.
\newblock {\em Bull. Amer. Math. Soc.}, 73:890--896, 1967.

\bibitem{AL1}
M.~Avellaneda and F.-H. Lin.
\newblock Compactness methods in the theory of homogenization.
\newblock {\em Comm. Pure Appl. Math.}, 40(6):803--847, 1987.

\bibitem{AL2}
M.~Avellaneda and F.-H. Lin.
\newblock {$L^p$} bounds on singular integrals in homogenization.
\newblock {\em Comm. Pure Appl. Math.}, 44(8-9):897--910, 1991.

\bibitem{Ba}
M.~T. Barlow.
\newblock Random walks on supercritical percolation clusters.
\newblock {\em Ann. Probab.}, 32(4):3024--3084, 2004.

\bibitem{BD}
M.~T. Barlow and J.-D. Deuschel.
\newblock Invariance principle for the random conductance model with unbounded
  conductances.
\newblock {\em Ann. Probab.}, 38(1):234--276, 2010.

\bibitem{BH}
M.~T. Barlow and B.~M. Hambly.
\newblock Parabolic {H}arnack inequality and local limit theorem for
  percolation clusters.
\newblock {\em Electron. J. Probab.}, 14:no. 1, 1--27, 2009.

\bibitem{bella2018liouville}
P.~Bella, B.~Fehrman, and F.~Otto.
\newblock A {L}iouville theorem for elliptic systems with degenerate ergodic
  coefficients.
\newblock {\em Ann. Appl. Probab.}, 28(3):1379--1422, 2018.

\bibitem{BS19}
P.~Bella and M.~Sch{\"a}ffner.
\newblock Quenched invariance principle for random walks among random
  degenerate conductances.
\newblock {\em arXiv preprint arXiv:1902.05793}, 2019.

\bibitem{BDKY}
I.~Benjamini, H.~Duminil-Copin, G.~Kozma, and A.~Yadin.
\newblock Disorder, entropy and harmonic functions.
\newblock {\em Ann. Probab.}, 43(5):2332--2373, 2015.

\bibitem{BM03}
I.~Benjamini and E.~Mossel.
\newblock On the mixing time of a simple random walk on the super critical
  percolation cluster.
\newblock {\em Probab. Theory Related Fields}, 125(3):408--420, 2003.

\bibitem{BB}
N.~Berger and M.~Biskup.
\newblock Quenched invariance principle for simple random walk on percolation
  clusters.
\newblock {\em Probab. Theory Related Fields}, 137(1-2):83--120, 2007.

\bibitem{BBHK}
N.~Berger, M.~Biskup, C.~E. Hoffman, and G.~Kozma.
\newblock Anomalous heat-kernel decay for random walk among bounded random
  conductances.
\newblock {\em Ann. Inst. Henri Poincar\'{e} Probab. Stat.}, 44(2):374--392,
  2008.

\bibitem{biskupsurvey}
M.~Biskup.
\newblock Recent progress on the random conductance model.
\newblock {\em Probab. Surv.}, 8:294--373, 2011.

\bibitem{BBo}
M.~Biskup and O.~Boukhadra.
\newblock Subdiffusive heat-kernel decay in four-dimensional i.i.d. random
  conductance models.
\newblock {\em J. Lond. Math. Soc. (2)}, 86(2):455--481, 2012.

\bibitem{BP}
M.~Biskup and T.~M. Prescott.
\newblock Functional {CLT} for random walk among bounded random conductances.
\newblock {\em Electron. J. Probab.}, 12:no. 49, 1323--1348, 2007.

\bibitem{Bou10}
O.~Boukhadra.
\newblock Heat-kernel estimates for random walk among random conductances with
  heavy tail.
\newblock {\em Stochastic Process. Appl.}, 120(2):182--194, 2010.

\bibitem{Bou09}
O.~Boukhadra.
\newblock Standard spectral dimension for the polynomial lower tail random
  conductances model.
\newblock {\em Electron. J. Probab.}, 15:no. 68, 2069--2086, 2010.

\bibitem{BK}
R.~M. Burton and M.~Keane.
\newblock Density and uniqueness in percolation.
\newblock {\em Comm. Math. Phys.}, 121(3):501--505, 1989.

\bibitem{Ch15}
Y.~Chang.
\newblock Supercritical loop percolation on {$\Bbb Z^d$} for {$d\geq 3$}.
\newblock {\em Stochastic Process. Appl.}, 127(10):3159--3186, 2017.

\bibitem{CD16}
A.~Chiarini and J.-D. Deuschel.
\newblock Invariance principle for symmetric diffusions in a degenerate and
  unbounded stationary and ergodic random medium.
\newblock {\em Ann. Inst. Henri Poincar\'{e} Probab. Stat.}, 52(4):1535--1563,
  2016.

\bibitem{DM1}
G.~Dal~Maso and L.~Modica.
\newblock Nonlinear stochastic homogenization.
\newblock {\em Ann. Mat. Pura Appl. (4)}, 144:347--389, 1986.

\bibitem{DM2}
G.~Dal~Maso and L.~Modica.
\newblock Nonlinear stochastic homogenization and ergodic theory.
\newblock {\em J. Reine Angew. Math.}, 368:28--42, 1986.

\bibitem{dario2018optimal}
P.~Dario.
\newblock Optimal corrector estimates on percolation clusters.
\newblock {\em arXiv preprint arXiv:1805.00902}, 2018.

\bibitem{Da93}
E.~B. Davies.
\newblock Large deviations for heat kernels on graphs.
\newblock {\em J. London Math. Soc. (2)}, 47(1):65--72, 1993.

\bibitem{de1989invariance}
A.~De~Masi, P.~A. Ferrari, S.~Goldstein, and W.~D. Wick.
\newblock An invariance principle for reversible {M}arkov processes.
  {A}pplications to random motions in random environments.
\newblock {\em J. Statist. Phys.}, 55(3-4):787--855, 1989.

\bibitem{D99}
T.~Delmotte.
\newblock Parabolic {H}arnack inequality and estimates of {M}arkov chains on
  graphs.
\newblock {\em Rev. Mat. Iberoamericana}, 15(1):181--232, 1999.

\bibitem{delmotte1999parabolic}
T.~Delmotte.
\newblock Parabolic {H}arnack inequality and estimates of {M}arkov chains on
  graphs.
\newblock {\em Rev. Mat. Iberoamericana}, 15(1):181--232, 1999.

\bibitem{DNS18}
J.-D. Deuschel, T.~A. Nguyen, and M.~Slowik.
\newblock Quenched invariance principles for the random conductance model on a
  random graph with degenerate ergodic weights.
\newblock {\em Probab. Theory Related Fields}, 170(1-2):363--386, 2018.

\bibitem{DR1984}
R.~Durrett.
\newblock {\em Brownian motion and martingales in analysis}.
\newblock Wadsworth Mathematics Series. Wadsworth International Group, Belmont,
  CA, 1984.

\bibitem{EGMN}
A.-C. Egloffe, A.~Gloria, J.-C. Mourrat, and T.~N. Nguyen.
\newblock Random walk in random environment, corrector equation and homogenized
  coefficients: from theory to numerics, back and forth.
\newblock {\em IMA J. Numer. Anal.}, 35(2):499--545, 2015.

\bibitem{EV10}
L.~C. Evans.
\newblock {\em Partial differential equations}, volume~19 of {\em Graduate
  Studies in Mathematics}.
\newblock American Mathematical Society, Providence, RI, second edition, 2010.

\bibitem{flegel2017homogenization}
F.~Flegel, M.~Heida, and M.~Slowik.
\newblock Homogenization theory for the random conductance model with
  degenerate ergodic weights and unbounded-range jumps.
\newblock {\em arXiv preprint arXiv:1702.02860}, 2017.

\bibitem{GG}
M.~Giaquinta and E.~Giusti.
\newblock On the regularity of the minima of variational integrals.
\newblock {\em Acta Math.}, 148:31--46, 1982.

\bibitem{GM79}
M.~Giaquinta and G.~Modica.
\newblock Regularity results for some classes of higher order nonlinear
  elliptic systems.
\newblock {\em J. Reine Angew. Math.}, 311/312:145--169, 1979.

\bibitem{GHV18}
A.~Giunti, R.~H\"{o}fer, and J.~Vel\'{a}zquez.
\newblock Homogenization for the {P}oisson equation in randomly perforated
  domains under minimal assumptions on the size of the holes.
\newblock {\em Comm. Partial Differential Equations}, 43(9):1377--1412, 2018.

\bibitem{Giu}
E.~Giusti.
\newblock {\em Direct methods in the calculus of variations}.
\newblock World Scientific Publishing Co., Inc., River Edge, NJ, 2003.

\bibitem{GO14}
A.~Gloria, S.~Neukamm, and F.~Otto.
\newblock A regularity theory for random elliptic operators.
\newblock {\em arXiv preprint arXiv:1409.2678}, 2014.

\bibitem{GNO}
A.~Gloria, S.~Neukamm, and F.~Otto.
\newblock Quantification of ergodicity in stochastic homogenization: optimal
  bounds via spectral gap on {G}lauber dynamics.
\newblock {\em Invent. Math.}, 199(2):455--515, 2015.

\bibitem{GO1}
A.~Gloria and F.~Otto.
\newblock An optimal variance estimate in stochastic homogenization of discrete
  elliptic equations.
\newblock {\em Ann. Probab.}, 39(3):779--856, 2011.

\bibitem{GO2}
A.~Gloria and F.~Otto.
\newblock An optimal error estimate in stochastic homogenization of discrete
  elliptic equations.
\newblock {\em Ann. Appl. Probab.}, 22(1):1--28, 2012.

\bibitem{GO15}
A.~Gloria and F.~Otto.
\newblock The corrector in stochastic homogenization: optimal rates, stochastic
  integrability, and fluctuations.
\newblock {\em arXiv preprint arXiv:1510.08290}, 2015.

\bibitem{GO115}
A.~{Gloria} and F.~{Otto}.
\newblock {Quantitative results on the corrector equation in stochastic
  homogenization.}
\newblock {\em {J. Eur. Math. Soc. (JEMS)}}, 19(11):3489--3548, 2017.

\bibitem{Grafakos}
L.~Grafakos.
\newblock {\em Classical and modern {F}ourier analysis}.
\newblock Pearson Education, Inc., Upper Saddle River, NJ, 2004.

\bibitem{grimmett1999}
G.~Grimmett.
\newblock {\em Percolation}, volume 321 of {\em Grundlehren der Mathematischen
  Wissenschaften [Fundamental Principles of Mathematical Sciences]}.
\newblock Springer-Verlag, Berlin, second edition, 1999.

\bibitem{gu2019efficient}
C.~Gu.
\newblock An efficient algorithm for solving elliptic problems on percolation
  clusters.
\newblock {\em arXiv preprint arXiv:1907.13571}, 2019.

\bibitem{kesten1990}
H.~Kesten and Y.~Zhang.
\newblock The probability of a large finite cluster in supercritical
  {B}ernoulli percolation.
\newblock {\em Ann. Probab.}, 18(2):537--555, 1990.

\bibitem{KMT1}
J.~Koml{\'o}s, P.~Major, and G.~Tusn{\'a}dy.
\newblock An approximation of partial sums of independent {RV}'-s, and the
  sample {DF}. {I}.
\newblock {\em Zeitschrift f{\"u}r Wahrscheinlichkeitstheorie und verwandte
  Gebiete}, 32(1-2):111--131, 1975.

\bibitem{KMT2}
J.~Koml{\'o}s, P.~Major, and G.~Tusn{\'a}dy.
\newblock An approximation of partial sums of independent {RV}'s, and the
  sample {DF}. {II}.
\newblock {\em Zeitschrift f{\"u}r Wahrscheinlichkeitstheorie und verwandte
  Gebiete}, 34(1):33--58, 1976.

\bibitem{K1}
S.~M. Kozlov.
\newblock Averaging of differential operators with almost periodic rapidly
  oscillating coefficients.
\newblock {\em Mat. Sb. (N.S.)}, 107(149)(2):199--217, 317, 1978.

\bibitem{LNO}
A.~Lamacz, S.~Neukamm, and F.~Otto.
\newblock Moment bounds for the corrector in stochastic homogenization of a
  percolation model.
\newblock {\em Electron. J. Probab.}, 20:no. 106, 30, 2015.

\bibitem{M}
P.~Mathieu.
\newblock Quenched invariance principles for random walks with random
  conductances.
\newblock {\em J. Stat. Phys.}, 130(5):1025--1046, 2008.

\bibitem{MP}
P.~Mathieu and A.~Piatnitski.
\newblock Quenched invariance principles for random walks on percolation
  clusters.
\newblock {\em Proc. R. Soc. Lond. Ser. A Math. Phys. Eng. Sci.},
  463(2085):2287--2307, 2007.

\bibitem{MR}
P.~Mathieu and E.~Remy.
\newblock Isoperimetry and heat kernel decay on percolation clusters.
\newblock {\em Ann. Probab.}, 32(1A):100--128, 2004.

\bibitem{mourrat2011variance}
J.-C. Mourrat.
\newblock Variance decay for functionals of the environment viewed by the
  particle.
\newblock {\em Ann. Inst. Henri Poincar{\'e}, Probab. Stat.}, 47(1):294--327,
  2011.

\bibitem{NS}
A.~Naddaf and T.~Spencer.
\newblock Estimates on the variance of some homogenization problems, 1998,
  {unpublished preprint}.

\bibitem{neukamm2017stochastique}
S.~Neukamm, M.~Schaffner, and A.~Schlomerkemper.
\newblock Stochastic homogenization of nonconvex discrete energies with
  degenerate growth.
\newblock {\em SIAM J. Math. Anal.}, 49(3):1761--1809, 2017.

\bibitem{Os82}
H.~Osada.
\newblock Homogenization of diffusion processes with random stationary
  coefficients.
\newblock In {\em Probability theory and mathematical statistics ({T}bilisi,
  1982)}, volume 1021 of {\em Lecture Notes in Math.}, pages 507--517.
  Springer, Berlin, 1983.

\bibitem{PV1}
G.~C. Papanicolaou and S.~R.~S. Varadhan.
\newblock Boundary value problems with rapidly oscillating random coefficients.
\newblock In {\em Random fields, {V}ol. {I}, {II} ({E}sztergom, 1979)},
  volume~27 of {\em Colloq. Math. Soc. J\'anos Bolyai}, pages 835--873.
  North-Holland, Amsterdam, 1981.

\bibitem{PP}
M.~Penrose and A.~Pisztora.
\newblock Large deviations for discrete and continuous percolation.
\newblock {\em Adv. in Appl. Probab.}, 28(1):29--52, 1996.

\bibitem{P}
A.~Pisztora.
\newblock Surface order large deviations for {I}sing, {P}otts and percolation
  models.
\newblock {\em Probab. Theory Related Fields}, 104(4):427--466, 1996.

\bibitem{procaccia2016quenched}
E.~Procaccia, R.~Rosenthal, and A.~Sapozhnikov.
\newblock Quenched invariance principle for simple random walk on clusters in
  correlated percolation models.
\newblock {\em Probab. Theory and Related Fields}, 166(3-4):619--657, 2016.

\bibitem{Sap13}
A.~Sapozhnikov.
\newblock Random walks on infinite percolation clusters in models with
  long-range correlations.
\newblock {\em Ann. Probab.}, 45(3):1842--1898, 2017.

\bibitem{SS}
V.~Sidoravicius and A.-S. Sznitman.
\newblock Quenched invariance principles for walks on clusters of percolation
  or among random conductances.
\newblock {\em Probab. Theory Related Fields}, 129(2):219--244, 2004.

\bibitem{stein1970singular}
E.~M. Stein.
\newblock {\em Singular integrals and differentiability properties of
  functions}.
\newblock Princeton Mathematical Series, No. 30. Princeton University Press,
  Princeton, N.J., 1970.

\bibitem{Y1}
V.~V. Yurinski{\u\i}.
\newblock Averaging of symmetric diffusion in a random medium.
\newblock {\em Sibirsk. Mat. Zh.}, 27(4):167--180, 215, 1986.

\end{thebibliography}

\end{document}